\documentclass{amsart}

\usepackage{amsthm}

\newtheorem*{theorem*}{Theorem}

\usepackage{macros-christian}
\usepackage[qeds]{formatting-christian}
\usepackage{parskip}
\usepackage{multicol}
\usepackage{marvosym}
\usepackage{verbatim}

\renewcommand{\Cat}{\mathsf{Cat}}
\newcommand{\cSet}{\mathsf{cSet}}

\renewcommand{\Set}{\mathsf{Set}}
\newcommand{\sSet}{\mathsf{sSet}}

\newcommand{\Ho}{\mathsf{Ho}}

\newcommand{\bd}{\partial}

\newcommand{\tBox}{\Box^+} 

\renewcommand{\Im}{\mathrm{Im}}
\newcommand{\cmin}{\cSet_\varnothing}
\newcommand{\cneg}{\cSet_0}
\newcommand{\cpos}{\cSet_1}
\newcommand{\cboth}{\cSet_{01}}
\newcommand{\cA}{\cSet_A}
\newcommand{\cB}{\cSet_B}
\newcommand{\Boxmin}{\Box_\varnothing}
\newcommand{\Boxneg}{\Box_0}
\newcommand{\Boxpos}{\Box_1}
\newcommand{\Boxboth}{\Box_{01}}
\newcommand{\BoxA}{\Box_A}
\newcommand{\BoxB}{\Box_B}
\newcommand{\BoxPO}{\Boxmin^n \cup_{\bd \Boxmin^n} i^* \bd \BoxA^n}
\newcommand{\adjoint}{\dashv}
\newcommand{\Nerve}{\mathsf{N}_\Box}

\newcommand{\sk}{\mathrm{sk}}
\newcommand{\gjq}{\gamma_{j : q, \mu}}
\newcommand{\gjqt}{\widetilde{\gamma}_{j : q,\mu}}

\newcommand{\wBox}{\widetilde{\Box}}

\newcommand{\pushoutcorner}[1][dr]{\save*!/#1+1.2pc/#1:(1,-1)@^{|-}\restore}

\newcommand{\pshf}[1]{\mathsf{Set}^{ #1 ^{\mathrm{op}}}}

\newtheorem{examples}[thm]{Examples}

\DeclareFontFamily{U}{mathx}{\hyphenchar\font45}
\DeclareFontShape{U}{mathx}{m}{n}{
      <5> <6> <7> <8> <9> <10>
      <10.95> <12> <14.4> <17.28> <20.74> <24.88>
      mathx10
      }{}
\DeclareSymbolFont{mathx}{U}{mathx}{m}{n}
\DeclareFontSubstitution{U}{mathx}{m}{n}
\DeclareMathAccent{\widecheck}{0}{mathx}{"71}
\DeclareMathAccent{\widecheck}{0}{mathx}{"71}

\begin{document}
\title{Cubical models of higher categories without connections}
\author[B.~Doherty]{Brandon Doherty}

\begin{abstract}
We prove that each of the model structures for ($n$-trivial, saturated) comical sets on the category of marked cubical sets having only faces and degeneracies (without connections) is Quillen equivalent to the corresponding model structure for ($n$-trivial, saturated) complicial sets on the category of marked simplicial sets, as well as to the corresponding comical model structures on cubical sets with connections. As a consequence, we show that the cubical Joyal model structure on cubical sets without connections is equivalent to its analogues on cubical sets with connections and to the Joyal model structure on simplicial sets. We also show that any comical set without connections may be equipped with connections via lifting, and that this can be done compatibly on the domain and codomain of any fibration or cofibration of comical sets.
\end{abstract}

\maketitle

\section*{Introduction}
In \cite{doherty-kapulkin-lindsey-sattler}, \emph{cubical sets with connections} were shown to model the homotopy theory of $(\infty,1)$-categories. Specifically, a \emph{cubical Joyal model structure} was established on each of the categories $\cSet_0$, $\cSet_1$ and $\cSet_{01}$, respectively consisting of cubical sets with  negative connections, positive connections, and both kinds of connections. Each of these model structures was shown to be Quillen equivalent, via triangulation, to the Joyal model structure on simplicial sets. Similar techniques were used in \cite{doherty-kapulkin-maehara} to show that the \emph{comical model structures} on the categories $\cneg^+$, $\cpos^+$, and $\cboth^+$ of marked cubical sets with connections (meaning cubical sets with connections having markings on cubes of positive dimension), modelling $(\infty,n)$-categories for all $n \geq 0$ (including $n = \infty$), are equivalent to the corresponding complicial model structures on the category $\sSet^+$ of marked simplicial sets developed in \cite{verity:weak-complicial-1,ozornova-rovelli}.

\emph{Minimal cubical sets}, \ie cubical sets having only faces and degeneracies, without connections, are also of  interest. These are the cubical sets which were first studied by Kan \cite{kan:abstract-htpy-1} prior to the introduction of simplicial sets. More recently, minimal cubical sets have been used to model the homotopy theory of $\infty$-groupoids, as in \cite{CisinskiAsterisque,jardine:categorical-homotopy-theory}. In \cite{doherty-kapulkin-lindsey-sattler} and \cite{doherty-kapulkin-maehara}, versions of the cubical Joyal and comical model structures were established on the category $\cmin$ of minimal cubical sets and the category $\cmin^+$ of marked minimal cubical sets. However, these model structures were not shown to be equivalent to those with connections, or to their simplicial analogues.

In this paper, we will prove the following:

\begin{theorem*}[\cf \cref{T-min-Quillen-eqv,T-Quillen-equiv-unmarked}]
The triangulation adjunction $T : \cmin \rightleftarrows \sSet : U$ defines a Quillen equivalence between the cubical Joyal model structure on $\cSet_\varnothing$ and the Joyal model structure on $\sSet$. Likewise, the marked triangulation adjunction $T : \cmin^+ \rightleftarrows \sSet^+ : U$ defines a Quillen equivalence between the each of the comical model structures on $\cSet_\varnothing^+$ and the corresponding complicial model structure on $\sSet^+$.
\end{theorem*}
 
Among the convenient properties of the category of minimal cubical sets is its simple presentation, with its objects possessing fewer structure maps than in alternative categories of cubical sets. Moreover, the minimal cube category has a useful universal property: as shown in \cite[Thm.~4.2]{grandis-mauri}, it is initial among strict monoidal categories equipped with an interval object, allowing for easy construction of co-cubical objects in such categories. On the other hand, cubical sets with connections are often easier to work with in homotopy-theoretic applications. For instance, the cube categories with connections, unlike the minimal cube category, are strict test categories \cite{maltsiniotis:connections-strict-test-cat,buchholtz-morehouse:varieties-of-cubes}. Similarly, all cubical groups with connections are Kan, but this does not hold for minimal cubical groups \cite{tonks:cubical-groups-kan}.
  
Most relevantly to the present work, the techniques used in \cite{doherty-kapulkin-lindsey-sattler,doherty-kapulkin-maehara} to prove Quillen equivalences of cubical models with their simplicial counterparts involve a functor from simplicial to cubical sets which can only be defined in the presence of connections. Thus we cannot simply adapt these techniques to the minimal cubical setting. Instead, we will compare $\cmin^+$ with the three categories of marked cubical sets with connections, by means of adjoint triples $i_! \adjoint i^* \adjoint i_*$ induced by the inclusions of the relevant cube categories. We will also use this comparison to show that any fibrant minimal cubical set can be obtained from a cubical set with connections by forgetting the connection structure. We will focus primarily on proving the desired Quillen equivalences in the marked case; we will obtain the unmarked result as a corollary of this, by showing the cubical Joyal model structure to be equivalent to one of the comical model structures.


\subsection*{Organization of the paper}

In \cref{sec:prelim}, we introduce the four categories of marked cubical sets which will be our focus and the adjoint triples by which we will compare them, and review essential background. We also state or prove various technical lemmas about model categories and the combinatorics of cubical sets which will be of use in subsequent sections.

In \cref{section:wcf}, we prove the main theorem of the paper, \cref{T-min-Quillen-eqv}, which states that the triangulation adjunction defines a Quillen equivalence between each of the comical model structures on $\cmin^+$ and the corresponding complicial model structure on $\sSet^+$. We do this via an extensive study of the adjoint triples $i_! \adjoint i^* \adjoint i_*$ between categories of marked cubical sets induced by the inclusions of the minimal cube category into the cube categories with connections, showing that each of these adjunctions is a Quillen equivalence. As part of our proof, we define the concept of a \emph{weak connection structure}, an assignment of connections on the cubes of a marked minimal cubical set $X \in \cmin^+$ which are compatible with the structure maps of $X$, but which are not required to satisfy the cubical identities for composition of connections. We describe techniques for constructing such structures on comical sets, the fibrant objects of the comical model structures, and show that they may be lifted along the fibrations of the comical model structures.

\cref{sec:surj} is concerned with a question raised by the techniques used in \cref{section:wcf}, namely: to what extent is it meaningful to say that comical sets ``have connections''? In this section we strengthen the definition of a weak connection structure to obtain that of a \emph{strong connection structure}, and show that a marked minimal cubical set is isomorphic to the image under $i^*$ of a marked cubical set with connections if and only if it admits such a structure. Thus, a reasonable interpretation of the statement that a marked minimal cubical set $X$ ``has connections'' is that there exists a strong connection structure on $X$. We show that every comical set in $\cmin^+$ admits a strong connection structure, and furthermore, that given a fibration or cofibration of comical sets $f \colon X \to Y$ in $\cmin^+$, such structures can be chosen so as to be preserved by $f$. However, this cannot always be done when $f$ is an arbitrary map, or even an equivalence of comical sets.

Finally, in Appendix \ref{app:inf-1}, we transfer the Quillen equivalences established in \cref{section:wcf} to the model structures for $(\infty,1)$-categories established in \cite{doherty-kapulkin-lindsey-sattler}. Specifically, we obtain Quillen equivalences between the cubical Joyal model structures on the various categories of (unmarked) cubical sets considered in \cite{doherty-kapulkin-lindsey-sattler}, and use the two-out-of three property to show that the triangulation adjunction $T : \cmin \rightleftarrows \sSet : U$ is a Quillen equivalence bewteen the cubical Joyal and Joyal model structures. We likewise obtain a similar equivalence between the cubical marked model structure on the category $\cmin'$ of cubical sets with marked edges and the marked model structure on the category $\sSet'$ of simplicial sets with marked edges.

\subsection*{Acknowledgements}

While working on this paper, the author was supported by a grant from the Knut and Alice Wallenberg Foundation, entitled  ``Type Theory for Mathematics and Computer Science'' (principal investigator: Thierry Coquand).

\section{Background}\label{sec:prelim}

\subsection{Model categories}

Before focusing specifically on marked cubical sets, we consider some general model-categorical results and concepts which will be of use in the proofs that follow.

\begin{definition}
A class of \emph{pseudo-generating trivial cofibrations} for a model category $\catC$ is a class of cofibrations $S$ such that a map in $\catC$ with fibrant codomain is a fibration if and only if it has the right lifting property against $S$.
\end{definition}

Note that by \cite[Prop.~7.15]{joyal-tierney:qcat-vs-segal}, if $\catC$ has a class of pseudo-generating trivial cofibrations $S$, then to show that an adjunction $\catC \rightleftarrows \catD$ is Quillen, it suffices to show that the left adjoint preserves cofibrations and sends the maps of $S$ to trivial cofibrations.

We next consider some results which will be of use in identifying Quillen equivalences.

\begin{proposition}[{\cite[Cor.~1.3.16]{hovey:book}}]\label{QuillenEquivCreate-original}
Let $F : \mathsf{C} \rightleftarrows \mathsf{D} : U $ be a Quillen adjunction between model categories.
Then the following are equivalent.
\begin{enumerate}
\item $F \adjoint U$ is a Quillen equivalence.
\item $F$ reflects weak equivalences between cofibrant objects and, for every fibrant $Y$, the derived counit $F\widetilde{UY} \to  Y$ is a weak equivalence.
\item $U$ reflects weak equivalences between fibrant objects and, for every cofibrant $X$, the derived unit $X \to U (FX)'$ is a weak equivalence.
\end{enumerate}
\end{proposition}

\begin{corollary}\label{QuillenEquivCreate}
Let $F : \mathsf{C} \rightleftarrows \mathsf{D} : U $ be a Quillen adjunction between model categories.

\begin{enumerate}
\item\label{QuillenEquivUnit} If $U$ preserves and reflects weak equivalences, then the adjunction is a Quillen equivalence if and only if, for all cofibrant $X \in \mathsf{C}$, the unit $X \to UFX$ is a weak equivalence.
\item\label{QuillenEquivCounit} If $F$ preserves and reflects weak equivalences, then the adjunction is a Quillen equivalence if and only if, for all fibrant $Y \in \mathsf{D}$, the counit $FUY \to Y$ is a weak equivalence. \qed
\end{enumerate}
\end{corollary}

\begin{proposition}\label{nat-we-QE}
Let $F \colon \catC \to \catD$ and $G \colon \catD \to \catC$ be a pair of left (resp. right) Quillen functors. 
If both composites $GF \colon \catC \to \catC$ and $FG \colon \catD \to \catD$ are related to the identities on the respective categories by zigzags of natural weak equivalences between left (resp. right) Quillen functors, then both $F$ and $G$ are left (resp. right) Quillen equivalences.
\end{proposition}

\begin{proof}
A natural weak equivalence between left or right Quillen functors between model categories gives rise to a natural isomorphism of their derived functors between homotopy categories. Moreover, the construction of the left or right derived functor is functorial up to natural isomorphism. Thus we obtain a natural isomoprhism $\Ho G \circ \Ho F \cong \Ho (GF) \cong \Ho (\id_{\catC}) \cong \id_{\Ho \catC}$, and likewise a natural isomorphism $\Ho F \circ \Ho G \cong \id_{\Ho \catD}$. We thus see that the derived functors $\Ho F$ and $\Ho G$ are inverse equivalences of categories.
\end{proof}

Next we prove a general lemma which we will use to verify that certain natural transformations of functors into model categories are in fact natural weak equivalences. The proof is a generalization of a standard argument which can be found, for instance, in the proof of \cite[Prop. 6.21]{doherty-kapulkin-lindsey-sattler}. We begin by recalling a basic concept of Reedy category theory which plays a key role in this lemma.

\begin{definition}\label{skel-def}
For $n \geq -1$, a presheaf $X$ on an EZ-Reedy category $\mathsf{A}$ is \emph{$n$-skeletal} if $X_a$ has no non-degenerate elements for all $a \in \mathsf{A}$ such that $\mathrm{deg}(a) > n$. The \emph{$n$-skeleton} of a presheaf $X$ on $\mathsf{A}$ is its maximal $n$-skeletal subcomplex.
\end{definition}

In particular, as the degree function on a Reedy category takes non-negative values, we may note that the only $(-1)$-skeletal presheaf on any EZ-Reedy category is the empty presheaf.


\begin{lemma}\label{nat-weq}
Let $\mathsf{A}$ be an EZ-Reedy category, and let $\catC$ denote a reflective subcategory of $\pshf{\mathsf{A}}$. Denote the left adjoint of the inclusion $\catC \hookrightarrow \pshf{\mathsf{A}}$ by $L \colon \pshf{\mathsf{A}} \to \catC$. Let $\eta \colon F \Rightarrow G$ be a natural transformation of functors from $\catC$ into a model category $\catM$. Suppose that the following criteria are satisfied:

\begin{enumerate}
\item \label{F-G-pres-colim} $F$ and $G$ preserve colimits.
\item \label{mono-to-cof} the composites $FL$ and $GL$ send monomorphisms in $\pshf{\mathsf{A}}$ to cofibrations in $\catM$.
\item \label{weq-on-rep} For some $n \geq - 1$, for every object $a \in \mathsf{A}$ of degree less than or equal to $n$, the component of $\eta$ at the object $L \mathsf{A}(-,a)$ is a weak equivalence.
\end{enumerate}

Then all components of $\eta$ at the images under $L$ of $n$-skeletal objects of $\pshf{\mathsf{A}}$ are weak equivalences. Moreover, if assumption \ref{weq-on-rep} holds for all $n \geq -1$, then all components of $\eta$ are weak equivalences.
\end{lemma}

\begin{proof}
We begin by considering the case in which $\catC = \pshf{\mathsf{A}}$, so that $L = \id$. We proceed by induction on $n \geq -1$. For the base case $n = -1$, we recall that the only $(-1)$-skeletal presheaf is the empty presheaf $\varnothing$. Item \ref{F-G-pres-colim} implies that both $F$ and $G$ send $\varnothing$ to the initial object of $\catM$. Hence $\eta_\varnothing$ is the identity on this object, which is indeed a weak equivalence.

Now suppose \ref{weq-on-rep} is satisfied for some $n \geq 0$, and that the statement holds for all $-1 \leq n' < n$. Consider an $n$-skeletal $X \in \pshf{\mathsf{A}}$. By \cite[Prop. 4.10]{isaacson:symmetric}, we have a pushout diagram in $\pshf{\mathsf{A}}$:
\[
\xymatrix{
\coprod\limits_{\mathsf{A}(-,a) \to X} \bd \mathsf{A}(-,a) \ar[r] \ar[d] & \mathrm{sk}_{n-1} X \ar[d] \\
\coprod\limits_{\mathsf{A}(-,a) \to X} \mathsf{A}(-,a) \ar[r] & X \pushoutcorner \\
}
\]
where the coproducts are taken over all non-degenerate $\mathsf{A}(-,a) \to X$ with $\mathrm{deg}(a) = n$, and $\bd \mathsf{A}(-,a)$ denotes the $(n-1)$-skeleton of $\mathsf{A}(-,a)$.

Applying $F$ and $G$ to this diagram and using assumption \ref{F-G-pres-colim}, we thus have the following diagram in $\catM$:
\[
\xymatrix{
\coprod\limits_{\mathsf{A}(-,a) \to X} F \bd \mathsf{A}(-,a) \ar[rr] \ar[dd] \ar[dr]^{\eta} & &  F \mathrm{sk}_{n-1} X \ar[dd]|{\hole} \ar[dr]^{\eta} \\
& \coprod\limits_{\mathsf{A}(-,a) \to X} G \bd \mathsf{A}(-,a) \ar[rr] \ar[dd] & & G \mathrm{sk}_{n-1} X \ar[dd] \\
\coprod\limits_{\mathsf{A}(-,a) \to X} F \mathsf{A}(-,a) \ar[rr]|(0.565){\hole} \ar[dr]^{\eta} & &  F X \ar[dr]^{\eta} \\
& \coprod\limits_{\mathsf{A}(-,a) \to X} G \mathsf{A}(-,a) \ar[rr] & & G X \\
}
\]

(Note that the component of $\eta$ at a coproduct is simply the map between coproducts induced by its components at each cofactor; this follows from assumption \ref{F-G-pres-colim}.) Assumptions \ref{F-G-pres-colim} and \ref{mono-to-cof} imply that all objects in this diagram are cofibrant, thus we may apply the gluing lemma. We verify that its hypotheses are satisfied:

\begin{itemize}
\item The front and back faces are pushouts by assumption \ref{F-G-pres-colim}.
\item The vertical maps are cofibrations by assumption \ref{mono-to-cof}.
\item The components of $\eta$ at $\mathrm{sk}_{n-1} X$ and each $\bd \mathsf{A}(-,a)$ are weak equivalences by the induction hypothesis, while the components at the objects $\mathsf{A}(-,a)$ are weak equivalences by assumption \ref{weq-on-rep}. Thus the back-to-front maps between the coproducts are weak equivalences as coproducts of weak equivalences between cofibrant objects.
\end{itemize}

Thus, by the gluing lemma, we see that $\eta_X \colon FX \to GX$ is a weak equivalence as well.

Now suppose \ref{weq-on-rep} holds for all $n$, and consider a general $X \in \pshf{\mathsf{A}}$; we may express $X$ as the transfinite composite of the inclusions $\mathrm{sk}_n X \hookrightarrow \mathrm{sk}_{n+1} X$. We thus obtain a diagram in $\catM$:
\[
\xymatrix{
F \sk_{0} \ar[d]^{\eta} X \ar[r] & F \sk_1 X \ar[d]^{\eta} \ar[r] & \ldots \ar[r] & F X \ar[d]^{\eta} \\
G \sk_{0} X \ar[r] & G \sk_1 X \ar[r] & \ldots \ar[r] & G X \\
}
\]
By assumption \ref{F-G-pres-colim}, the top and bottom rows are transfinite compositions, and $\eta_X$ is the induced map between the colimit objects. Moreover, all the horizontal maps are cofibrations between cofibrant objects by assumptions \ref{F-G-pres-colim} and \ref{mono-to-cof}; thus these transfinite compositions are homotopy colimits. Since each component $\eta_{\sk_n X}$ is a weak equivalence, it thus follows that $\eta_X$ is a weak equivalence as well.

We now consider the general case of a reflective subcategory $\catC \subseteq \pshf{\mathsf{A}}$. Assuming that $F$ and $G$ satisfy \ref{F-G-pres-colim} through \ref{weq-on-rep} with respect to $\catC$ for some $n$, we may verify that the composites $FL, GL$ and the natural transformation $\eta_L$ satisfy \ref{F-G-pres-colim} through \ref{weq-on-rep} with respect to $\pshf{\mathsf{A}}$ viewed as a reflective subcategory of itself. For \ref{F-G-pres-colim} this follows from the fact that $L$ preserves colimits as a left adjoint, while for the other two assumptions it is trivial. It thus follows that $\eta_{LX}$ is a weak equivalence for every $n$-skeletal $X$ in $\pshf{\mathsf{A}}$. Similarly, if \ref{F-G-pres-colim} and \ref{mono-to-cof} hold and \ref{weq-on-rep} holds for all $n$, then $\eta_{LX}$  is a weak equivalence for all $X \in \pshf{\mathsf{A}}$. But \cite[Prop. 1.3]{gabriel-zisman} implies that $L$ is essentially surjective, hence all components of $\eta$ are weak equivalences.
\end{proof}

\begin{remark}
Although \cref{nat-weq} is stated in significant generality, and we will indeed have use for this general result in \cref{section:wcf}, it is most readily applicable in the case where $\catC = \pshf{\mathsf{A}}$ comes equipped with a model structure having monomorphisms as its cofibrations, and $F$ and $G$ are left Quillen functors; in this case assumptions \ref{F-G-pres-colim} and \ref{mono-to-cof} are  automatically satisfied.
\end{remark}

\subsection{Cubical sets without markings}

We now turn our attention to some specific presheaf categories which will be central to our work. \emph{Cube categories}, for our purposes, will mean subcategories of the category of posets whose objects are the interval posets $[1]^n, n \geq 0$, and whose morphisms are generated under composition by some or all of the following four classes of maps:

\begin{itemize}
  \item \emph{faces} $\partial^n_{i,\varepsilon} \colon [1]^{n-1} \to [1]^n$ for $i = 1, \ldots , n$ and $\varepsilon = 0, 1$ given by:
  \[ \partial^n_{i,\varepsilon} (x_1, x_2, \ldots, x_{n-1}) = (x_1, x_2, \ldots, x_{i-1}, \varepsilon, x_i, \ldots, x_{n-1})\text{;}  \]
  \item \emph{degeneracies} $\sigma^n_i \colon [1]^n \to [1]^{n-1}$ for $i = 1, 2, \ldots, n$ given by:
  \[ \sigma^n_i ( x_1, x_2, \ldots, x_n) = (x_1, x_2, \ldots, x_{i-1}, x_{i+1}, \ldots, x_n)\text{;}  \]
  \item \emph{negative connections} $\gamma^n_{i,0} \colon [1]^n \to [1]^{n-1}$ for $i = 1, 2, \ldots, n-1$ given by:
  \[ \gamma^n_{i,0} (x_1, x_2, \ldots, x_n) = (x_1, x_2, \ldots, x_{i-1}, \max\{ x_i , x_{i+1}\}, x_{i+2}, \ldots, x_n) \text{.} \]
  \item \emph{positive connections} $\gamma^n_{i,1} \colon [1]^n \to [1]^{n-1}$ for $i = 1, 2, \ldots, n-1$ given by:
  \[ \gamma^n_{i,1} (x_1, x_2, \ldots, x_n) = (x_1, x_2, \ldots, x_{i-1}, \min\{ x_i , x_{i+1}\}, x_{i+2}, \ldots, x_n) \text{.} \]
\end{itemize}

These maps obey the following \emph{cubical identities}:

\begin{multicols}{2}
$\partial_{j, \varepsilon'} \partial_{i, \varepsilon} = \partial_{i+1, \varepsilon} \partial_{j, \varepsilon'}$ for $j \leq i$;

$\sigma_i \sigma_j = \sigma_j \sigma_{i+1} \quad \text{for } j \leq i$;

$\sigma_j \partial_{i, \varepsilon} = \left\{ \begin{array}{ll}
\partial_{i-1, \varepsilon} \sigma_j   & \text{for } j < i \text{;} \\
\id                                                       & \text{for } j = i \text{;} \\
\partial_{i, \varepsilon} \sigma_{j-1} & \text{for } j > i \text{;}
\end{array}\right.$

$\gamma_{j,\varepsilon'} \gamma_{i,\varepsilon} = \left\{ \begin{array}{ll} \gamma_{i,\varepsilon} \gamma_{j+1,\varepsilon'} & \text{for } j > i \text{;} \\
\gamma_{i,\varepsilon}\gamma_{i+1,\varepsilon} & \text{for } j = i, \varepsilon' = \varepsilon \text{;}\\
\end{array}\right.$

$\gamma_{j,\varepsilon'} \partial_{i, \varepsilon} =  \left\{ \begin{array}{ll}
\partial_{i-1, \varepsilon} \gamma_{j,\varepsilon'}   & \text{for } j < i-1 \text{;} \\
\id                                                         & \text{for } j = i-1, \, i, \, \varepsilon = \varepsilon' \text{;} \\
\partial_{i, \varepsilon} \sigma_i         & \text{for } j = i-1, \, i, \, \varepsilon = 1-\varepsilon' \text{;} \\
\partial_{i, \varepsilon} \gamma_{j-1,\varepsilon'} & \text{for } j > i \text{;} 
\end{array}\right.$

$\sigma_j \gamma_{i,\varepsilon} =  \left\{ \begin{array}{ll}
\gamma_{i-1,\varepsilon} \sigma_j  & \text{for } j < i \text{;} \\
\sigma_i \sigma_i           & \text{for } j = i \text{;} \\
\gamma_{i,\varepsilon} \sigma_{j+1} & \text{for } j > i \text{.} 
\end{array}\right.$
\end{multicols}

For $A \subseteq \{0,1\}$, let $\Box_A$ denote the cube category with morphisms generated by faces, degeneracies, and connections $\gamma_{i,\varepsilon}$ for $\varepsilon \in A$. (We will often omit braces and commas in writing subsets of $\{0,1\}$ for the sake of readability.) We have thus defined four cube categories:

\begin{itemize}
\item $\Boxmin$, having only faces and degeneracies;
\item $\Boxneg$, having faces, degeneracies, and negative connections;
\item $\Boxpos$, having faces, degeneracies, and positive connections;
\item $\Boxboth$, having faces, degeneracies, and both kinds of connections.
\end{itemize}

Other cube categories, such as the full subcategory of the category of posets on the objects $[1]^n$, have been studied elsewhere; see \cite{buchholtz-morehouse:varieties-of-cubes} for an overview. In this paper, however, we will concern ourselves only with the four cube categories described above; all references to cube categories here should be interpreted to refer only to these four, and our results should not be understood as applying to any others.

For each such $A$, let $\cSet_A$ denote the presheaf category $\pshf{\Box_A}$. The representable presheaves in $\cSet_A$ will be denoted $\Box^n_A$. 

We adopt the convention of writing the action of cubical operators on the right. For instance, the $(1, 0)$-face of an $n$-cube $x \colon \Box^n_A \to X$ will be denoted $x \partial_{1, 0}$. By a \emph{degenerate} cube of a cubical set $X$, we will understand one that is in the image of either a degeneracy or a connection.
A \emph{non-degenerate} cube is one that is not degenerate.

We will occasionally represent cubical sets using pictures.
In doing so, we will follow the conventions used in \cite{doherty-kapulkin-lindsey-sattler}, in which $0$-cubes are represented as vertices, $1$-cubes as arrows, $2$-cubes as squares, and $3$-cubes as cubes.

For a $1$-cube $f$, we draw
\[
\xymatrix{ x \ar[r]^f & y} \]
to indicate $x = f \partial_{1,0}$ and $y = f \partial_{1,1}$.
For a $2$-cube $s$, we draw
\[
\xymatrix{
 x
  \ar[r]^h
  \ar[d]_f
&
   y
  \ar[d]^g
\\
  z
  \ar[r]^k
&
 w
}
\]
to indicate $s\partial_{1,0} = f$,  $s\partial_{1,1} = g$, $s\partial_{2,0} = h$,  and $s\partial_{2,1} = k$. 

As for the convention when drawing $3$-dimensional boxes, we use the following ordering of axes:
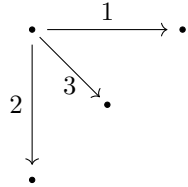
\[
\begin{tikzpicture}
\filldraw
(0,0) circle [radius = 1pt]	
(2,0) circle [radius = 1pt]	
(0,-2) circle [radius = 1pt]	
(1,-1) circle [radius = 1pt];

\draw[->] (0.2,0) -- (1.8,0) node [midway, above] {$1$};
\draw[->] (0.1,-0.1) -- (0.9,-0.9) node [midway, below] {$3$};
\draw[->] (0,-0.2) -- (0,-1.8) node [midway, left] {$2$};
\end{tikzpicture}
\]
For readability, we do not label $2$- and $3$-cubes.
Similarly, if a specific $0$-cube is irrelevant for the argument or can be inferred from the context, we represent it by $\bullet$, and we omit labels on edges whenever the label is not relevant for the argument.

Lastly, a degenerate $1$-cube $x \sigma_1$ on $x$ is represented by
\[
\xymatrix{ x \ar@{=}[r] & x\text{,}} \]
while a $2$- or $3$-cube whose boundary agrees with that of a degenerate cube is assumed to be degenerate unless indicated otherwise.
For instance, a $2$-cube depicted as
\[
\xymatrix{
 x
  \ar@{=}[r]
  \ar[d]_f
&
   x
  \ar[d]^f
\\
  y
  \ar@{=}[r]
&
 y
}
\]
represents $f \sigma_1$.

The following result gives a \emph{standard form} for maps in $\Boxboth$, which specializes in evident ways to the subcategories $\Boxmin, \Boxneg, \Boxpos$.

\begin{theorem}[{\cite[Thm.~5.1]{grandis-mauri}}] \label{normal-form}
  Every map in the category $\Box_{01}$ can be factored uniquely as a composite
  \[ (\partial_{c_1, \varepsilon'_1} \ldots \partial_{c_r, \varepsilon'_r})
     (\gamma_{b_1,\varepsilon_1} \ldots \gamma_{b_q,\varepsilon_q})
     (\sigma_{a_1} \ldots \sigma_{a_p})\text{,} \]
  where $1 \leq a_1 < \ldots < a_p$, $1 \leq b_1 \leq \ldots \leq b_q$, $b_i < b_{i+1}$ if $\varepsilon_{i} = \varepsilon_{i+1}$, and $c_1 > \ldots > c_r \geq 1$.   \qed
\end{theorem}

\begin{corollary}\label{epi-or-mono}
A map in any cube category $\BoxA$ is a monomorphism if and only if its standard form contains only face maps, and an epimorphism if and only if its standard form contains only connection and degeneracy maps. In particular, every monomorphism in any cube category is contained in $\Boxmin$. \qed
\end{corollary}

In view of the result above, we will regularly have occasion to refer to composites of face maps, or to composites of face and degeneracy maps; in both of these cases, the terminology should be understood to include empty composites, defined to be identities, unless otherwise noted. In particular, when factoring a general map $\phi$ as $\phi_1 \ldots \phi_p$ or some similar expression, the case $p = 0$ is understood as the case where $\phi$ is an identity.

\begin{corollary}\label{epi-mono-factor}
In any cube category $\BoxA$, each map has a unique factorization as an epimorphism followed by a monomorphism. \qed
\end{corollary}

\begin{corollary}\label{Box-Reedy}
Each of the cube categories defined above admits the structure of an EZ-Reedy category, in which:
\begin{itemize}
\item $\mathrm{deg}([1]^{n}) = n$;
\item $(\Box_A)_{+}$ is generated under composition by the face maps;
\item $(\Box_A)_{-}$ is generated under composition by the degeneracy and connection maps.\qed
\end{itemize}
\end{corollary}

Interpreting the standard definition of a skeletal presheaf in terms of this Reedy structure gives us a natural concept of a skeletal cubical set.

\begin{definition}\label{skel-def-unmarked}
For any $A \subseteq \{0,1\}$ and $n \geq - 1$, a cubical set $X \in \cSet_A$ is \emph{$n$-skeletal} if it has no non-degenerate cubes above dimension $n$.
\end{definition}

From here until the end of this subsection, fix a specific $A \subseteq \{0,1\}$; all of our definitions and results will be valid for all choices of $A$.

We now prove some combinatorial lemmas involving the standard forms of maps in $\Box_A$, which will be of use when proving our main results in \cref{section:wcf}.

\begin{lemma}\label{face-shift}
For any map $\phi \colon [1]^m \to [1]^n$ in $\Box_A$ and any face map $\bd_{i,\mu} \colon [1]^n \to [1]^{n+1}$, the standard form of the composite $\bd_{i,\mu} \phi$ contains at least one face map.
\end{lemma}

\begin{proof}
Writing $\phi$ in standard form as in the statement of \cref{normal-form}, we may express $\bd_{i,\mu}\phi$ as follows:

$$
\bd_{i,\mu} \phi = \bd_{i,\mu}\partial_{c_1, \varepsilon'_1} \ldots \partial_{c_r, \varepsilon'_r}
     \gamma_{b_1,\varepsilon_1} \ldots \gamma_{b_q,\varepsilon_q}
     \sigma_{a_1} \ldots \sigma_{a_p}
$$

Although the standard form of $\phi$ might not contain any face maps, the string of faces $\bd_{i,\mu}\partial_{c_1, \varepsilon'_1} \ldots \partial_{c_r, \varepsilon'_r}$ is necessarily non-empty, as it contains $\bd_{i,\mu}$. By repeatedly applying the cubical identity for composition of face maps, we can rearrange this string into one whose indices are in strictly decreasing order; thus we obtain the standard form of $\bd_{i,\mu} \phi$, containing at least one face map.
\end{proof}

\begin{lemma}\label{degen-shift}
Let $\phi \colon [1]^m \to [1]^n$ and $\psi \colon [1]^n \to [1]^k$ be maps in $\Box_{A}$, with $\phi$ an epimorphism. Then we may characterize the standard form of the composite $\psi \phi$ as follows:

\begin{itemize}
\item if the standard form of either $\phi$ or $\psi$ contains at least one degeneracy, then so does that of $\psi \phi$;
\item if neither of the standard forms of $\phi$ and $\psi$ contains a degeneracy, but at least one contains a connection, then that of $\psi \phi$ contains at least one connection, but no degeneracies.
\end{itemize}
\end{lemma}

\begin{proof}
Writing $\psi$ in standard form as $\delta \psi'$, where $\delta$ is a monomorphism and $\psi'$ an epimorphism, we may note that the standard form of $\psi \phi$ is the concatenation of that of $\delta$ with that of $\psi' \phi$; thus we may assume without loss of generality that $\psi$ is an epimorphism.

We first consider the case where the standard forms of both $\phi$ and $\psi$ contain no degeneracies, with at least one containing a connection. Then $\psi \phi$ is a composite of connections; by repeatedly applying the cubical identities for comoposition of connection maps, we may rewrite this composite into standard form. Thus we do indeed obtain a standard form of $\psi \phi$ containing at least one connection, but no degeneracies.

Now consider the case where at least one of the standard forms of $\phi$ and $\psi$ contains a degeneracy. By repeatedly applying the cubical identities for composition of degeneracies with degeneracies and connections, the concatenation of these standard forms can be expressed as a string of such maps whose rightmost map is a degeneracy. Therefore, without loss of generality, we may assume that $\phi = \phi'\sigma_i$ for some degeneracy $\sigma_i \colon [1]^m \to [1]^{m-1}$ and some epimorphism $\phi' \colon [1]^{m-1} \to [1]^n$.

Writing $\psi \phi'$ in standard form, we may express $\psi \phi$ as follows:

$$
\psi \phi = \partial_{c_1, \varepsilon'_1} \ldots \partial_{c_r, \varepsilon'_r}
     \gamma_{b_1,\varepsilon_1} \ldots \gamma_{b_q,\varepsilon_q}
     \sigma_{a_1} \ldots \sigma_{a_p} \sigma_i
$$

Similarly to the proof of \cref{face-shift}, the string of degeneracies $\sigma_{a_1} \ldots \sigma_{a_p} \sigma_i$ is non-empty, as it contains at least the map $\sigma_i$. Repeatedly applying the cubical identities for composition of degeneracies, we can rearrange this string into one whose indices are in strictly increasing order; thus we obtain the standard form of $\psi \phi$, containing at least one degeneracy map.
\end{proof}

Many of our proofs will require detailed study of maps whose standard forms only contain connections; thus we now define some terminology and prove some results relating to such maps. 

\begin{definition}\label{max-index-def}
Let $\phi$ be a (non-empty) composite of connection maps in $\Box_A$. The \emph{maximal index} of $\phi$ is the largest index of any connection map appearing in the standard form of $\phi$.
\end{definition}

For a map $\phi$ written in standard form as $\phi = \gamma_{i_1,\varepsilon_1} \ldots \gamma_{i_p,\varepsilon_p}$, the maximal index is simply $i_p$; this follows from the definition of the standard form. If we are given an expression for a map $\phi$ as a composite of connections, not necessarily written in standard form, then it may not be the case that the largest index appearing in this expression is in fact the maximal index of $\phi$; however, the following result shows that it does give a lower bound on the maximal index of $\phi$.

\begin{lemma}\label{con-shift}
Let $\phi = \gamma_{i_1,\varepsilon_1} \ldots \gamma_{i_p,\varepsilon_p}$ be a (non-empty) composite of connection maps in $\Box_A$. Then the maximal index of $\phi$ is greater than or equal to $\max_k(i_k)$. 
\end{lemma}

\begin{proof}
We proceed by induction on $p$. In the base case $p = 1$ we have $\phi = \gamma_{i_1,\varepsilon_1}$, which is already written in standard form; thus the maximal index of $\phi$ is $i_1$, which is trivially equal to  $\max_k(i_k)$.

Now let $p \geq 2$ and suppose the result holds for composites of $p-1$ connection maps, and let $\phi$ be a composite of $p$ connection maps. Let $\psi = \gamma_{i_2,\varepsilon_2} \ldots \gamma_{i_p, \varepsilon_p}$; this is a composite of $p-1$ connections, so by the induction hypothesis, it may be written in standard form as $\gamma_{j_1,\mu_1} \ldots \gamma_{j_{p-1},\mu_{p-1}}$, where $j_{p-1} \geq \max_{k \geq 2}(i_k)$. 

Let $q$ be minimal such that either $j_q >  i_1 + q - 1$ or $j_q = i_1 + q - 1$ and $\mu_q = 1 - \varepsilon_1$; if no such $q$ exists then let $q = p$. We will show that for all $1 \leq r \leq q$ we have:
\[
\phi = \gamma_{j_1,\mu_1} \ldots \gamma_{j_{r-1},\mu_{r-1}} \gamma_{{i_1+r-1},\varepsilon_1} \gamma_{j_r,\mu_r} \ldots \gamma_{j_{p-1}}
\]
In the extreme cases $r = 1$ and $r = q = p$ we interpret this statement as $\phi = \gamma_{i_1,\varepsilon_1}\gamma_{j_1,\mu_1} \ldots \gamma_{j_{p-1},\mu_{p-1}}$ and $\phi = \gamma_{j_1,\mu_1} \ldots \gamma_{j_{p-1},\mu_{p-1}} \gamma_{i_1 + p - 1,\varepsilon_1}$, respectively. Thus the base case $r = 1$ reduces to $\phi = \gamma_{i_1,\varepsilon_1} \psi$, which is true by the definition of $\psi$.

Now let $2 \leq r \leq q$, and suppose the statement is true for $r-1$, i.e. that
\[
\phi = \gamma_{j_1,\mu_1} \ldots \gamma_{j_{r-2},\mu_{r-2}} \gamma_{{i_1+r-2},\varepsilon_1} \gamma_{j_{r-1},\mu_{r-1}} \ldots \gamma_{j_{p-1}}
\]
By the minimality of $q$, we have $j_{r-1} \leq i_1 + r - 2$, with equality only if $\mu_{r-1} = \varepsilon_{1}$. Thus we may apply the cubical identity for composition of connections to see that $\gamma_{i_1+r-2,\varepsilon_1} \gamma_{j_{r-1},\mu_{r-1}} = \gamma_{j_{r-1},\mu_{r-1}} \gamma_{i_1 + r - 1,\varepsilon_1}$, thus proving the statement for $r$.

By induction, we see that the statement is true for all $1 \leq r \leq q$; in particular, the case $r = q$ gives us
\[
\phi = \gamma_{j_1,\mu_1} \ldots \gamma_{j_{q-1},\mu_{q-1}} \gamma_{{i_1+q-1},\varepsilon_1} \gamma_{j_q,\mu_q} \ldots \gamma_{j_{p-1}}
\]
In fact, this is the standard form of $\phi$. To see this, note that by our choice of $q$ we have $j_{q-1} \leq i_1 + q - 2 < i_1 + q - 1 \leq j_q$, and the last inequality is an equality only if $\mu_q = 1 - \varepsilon_1$. The necessary conditions hold on all other adjacent pairs of connection maps because $\gamma_{j_1,\mu_1} \ldots \gamma_{j_{p-1},\mu_{p-1}}$ was already in standard form.

If $q < p$, then the largest index of a connection in this standard form is $j_{p-1} \geq \max_{k \geq 2}(i_k)$. In this case, we have $i_1 \leq i_1 + q - 1 \leq j_q \leq j_{p-1}$, so in fact $j_{p-1} \geq \max_k(i_k)$. On the other hand, if $q = p$ then the largest index of a connection in this standard form is $i_1 + p - 1$. In this case, we have $i_1 + p - 1 > j_{p-1} \geq \max_{k}(i_k)$, and furthermore $i_1 + p - 1 > i_1$, so that $i_1 + p - 1 > \max_{k}(i_k)$.
\end{proof}

The cartesian product of cubical sets is not well-behaved in general; for  instance, it is not the case that $\Box^m_A \times \Box^n_A \cong \Box^{m+n}_A$ in any of the categories under consideration here. In view of this, we typically work with an alternative monoidal product on our categories of cubical sets.

We define this monoidal product as follows: the assignment $([1]^m,[1]^n) \mapsto [1]^{m+n}$ extends naturally to morphisms to define a functor $\Box_A \times \Box_A \to \Box_A$. Postcomposing this functor with the Yoneda embedding, and taking the left Kan extension along the product of the Yoneda embedding with itself, we obtain a bifunctor $\otimes \colon \cSet \times \cSet \to \cSet$, which we call the \emph{geometric product}.
\[
\xymatrix@C+0.5cm{
  \Box_A \times \Box_A
  \ar[r]
  \ar@{^{(}->}[d]
&
  \cSet_A
\\
  \cSet_A \times \cSet_A
  \ar[ru]_{\otimes}
&
}
\]
This defines a non-symmetric, biclosed monoidal structure on $\cSet_A$, with the unit given by $\Box^0_A$. 

\begin{proposition}[{\cite[Prop.~1.24]{doherty-kapulkin-lindsey-sattler}}]\label{geo-prod-description}
  Given cubical sets $X, Y \in \cSet_A$, we have the following description of their geometric product $X \otimes Y$.
\begin{itemize}
	\item For $n \geq 0$, the $n$-cubes in $X \otimes Y$ are the formal products $x \otimes y$ of pairs $x \in X_k$ and $y \in Y_\ell$ such that $k+\ell = n$, subject to the identification $(x\sigma_{k+1})\otimes y = x\otimes(y\sigma_{1})$.
 	\item For $x \in X_k$ and $y \in Y_\ell$, the faces, degeneracies, and connections of the $(k+\ell)$-cube $x \otimes y$ are computed as follows:
 	\begin{itemize}
 		\item $(x\otimes y)\partial_{i,\varepsilon} = 
 		\begin{cases} (x\partial_{i,\varepsilon})\otimes y & 1 \leq i \leq k \\ 
 		x\otimes( y\partial_{i-k,\varepsilon})  & k + 1 \leq i \leq k+l
 		\end{cases}$
 		\item $(x\otimes y)\sigma_{i} = 
 		\begin{cases} (x\sigma_{i})\otimes y     & 1 \leq i \leq k + 1 \\
 		x\otimes (y\sigma_{i-k}) & k + 1 \leq i \leq k+l+ 1 
 		\end{cases}$  
 		\item $(x\otimes y)\gamma_{i,\varepsilon} = 
 		\begin{cases} (x\gamma_{i, \varepsilon})\otimes y     & 1 \leq i \leq k \\ 
 		x\otimes (y\gamma_{i-k,\varepsilon}) & k + 1 \leq i \leq k+l
 		\end{cases}$
 	\end{itemize}
 \end{itemize}
 In particular, an $n$-cube $x \otimes y$ of $X \otimes Y$ is non-degenerate exactly when both $x$ and $y$ are non-degenerate in $X$ and $Y$, respectively. \qed
 \end{proposition}

The restriction of the nerve functor defines a functor $\Box_A \to \sSet$; taking the left Kan extension of this functor along the Yoneda embedding, we obtain the \emph{triangulation} functor $T \colon \cSet_A \to \sSet$.

\[
\xymatrix@C+0.5cm{
  \BoxA
  \ar[r]
  \ar[d]
&
  \sSet
\\
  \cSet_A
  \ar[ru]_{T}
&
}
\]

This functor has a right adjoint $U \colon \sSet \to \cSet_A$ given by $(UX)_{n} = \sSet((\Delta^1)^n,X)$. Intuitively, we think of triangulation as creating a simplicial set $TX$ from a cubical set $X$ by subdividing the cubes of $X$ into simplices.

\subsection{Marked cubical sets}\label{sec:mcSet}

To model $(\infty,n)$-categories for general $n$, we make use of cubical sets which have some of their cubes \emph{marked}, designating these cubes as equivalences. To define these objects precisely, for each $A \subseteq \{0,1\}$ we introduce a new category $\tBox_A$, an enlargement of $\BoxA$.
The category $\tBox_A$ consists of objects of the form $[1]^n$ for $n \geq 0$, as well as objects $[1]^n_e$ for $n \geq 1$.
The maps of $\tBox_A$ are generated by the generating maps of $\Box_A$, along with the following:

\begin{itemize}
\item $\varphi^n \colon [1]^n \to [1]^n_e$ for $n \geq 1$;
\item $\zeta^{n}_{i} \colon [1]^{n}_e \to [1]^{n-1}$ for $n \geq 1, 1 \leq i \leq n$;
\item $\xi^{n}_{i,\varepsilon} \colon [1]^{n}_e \to [1]^{n-1}$ for $n \geq 2, 1 \leq i \leq n - 1, \varepsilon \in A$
\end{itemize}

subject to the usual cubical identities, plus the following:

\begin{multicols}{2}
  $\zeta_i \varphi = \sigma_i$;

  $\xi_{i, \varepsilon} \varphi = \gamma_{i, \varepsilon}$;

  $\sigma_i \zeta_j = \sigma_j \zeta_{i+1}$ for $j \leq i$;

  $\gamma_{j, \varepsilon} \xi_{i, \delta} = \left\{ \begin{array}{ll}
	\gamma_{i, \varepsilon'} \xi_{j, \varepsilon} & \text{for } j > i\text{;} \\
	\gamma_{i, \varepsilon'} \xi_{i+1, \delta} & \text{for } j=i, \\ & \varepsilon' = \varepsilon\text{;}
	\end{array}\right.$	

 $\sigma_j \xi_{i, \varepsilon} =  \left\{ \begin{array}{ll}
	\gamma_{i-1, \varepsilon} \zeta_j    & \text{for } j < i \text{;} \\
	\sigma_i \zeta_i                        & \text{for } j = i \text{;} \\
	\gamma_{i, \varepsilon} \zeta_{j+1} & \text{for } j > i \text{.} 
	\end{array}\right.$
\end{multicols}

\begin{proposition-qed}[{\cite[Prop.\ 2.1]{campion-kapulkin-maehara}}] \label{tBox-Reedy}
  Each category $\tBox_A$ is an EZ Reedy category with the Reedy structure defined as follows:
  \begin{itemize}
    \item $\deg ([1]^0) = 0$, $\deg([1]^n) = 2n-1$ for $n \geq 1$, and $\deg ([1]^n_e) = 2n$ for $n \geq 1$;
    \item $(\tBox_A)_+$ is generated by the maps $\bd^n_{i,\varepsilon}$ and $\varphi^n$ under composition;
    \item $(\tBox_A)_-$ is generated by the maps $\sigma^{n}_{i}, \gamma^{n}_{i,\varepsilon}, \zeta^n_{i}$, and $\xi^{n}_{i,\varepsilon}$ (as applicable) under composition. \qedhere
  \end{itemize}
\end{proposition-qed}

For the remainder of this subsection, fix a specific $A \subseteq \{0,1\}$.

A \emph{structurally marked cubical set} is a contravariant functor $X \colon (\tBox_A)^\op \to \Set$ and a morphism of structurally marked cubical sets is a natural transformation of such functors.
We will write $\cSet^{++}_A$ for the category of structurally marked cubical sets.
The representable presheaf at the object $[1]^n$ will be denoted $\Box^n_A$, as in $\cSet_A$,  while the representable presheaf at the object $[1]^n_e$ will be denoted $\widetilde{\Box}^n_A$.

Structurally marked cubical sets should be thought of as cubical sets with (possibly multiple) labels on their cubes of positive dimension, called \emph{markings}, such that each degenerate cube has, in particular, one distinguished marking: for a cube of the form $x \sigma_i$ this is $x \zeta_i$, while for a cube of the form  $x \gamma_{i,\varepsilon}$ this is $x \xi_{i,\varepsilon}$. Specifically, the image under $X$ of the object $[1]^n_e$ is thought of as the set of markings on $n$-cubes of $X$, with an element $x \in [1]^n_e$ being a marking on the $n$-cube $x \varphi$. For $n \geq 1$, the underlying cubical set of $\widetilde{\Box}^n_A$ is $\Box^n_A$, with a unique marking on the unique non-degenerate $n$-cube, while all other non-degenerate cubes are unmarked. When quantifying a statement over arbitrary representable presheaves, we will often use the generic notation $\Box^n_{A,(e)}$, with the understanding that this may refer to $\Box^n_A$ for $n \geq 0$ or $\wBox^n_{A}$ for $n \geq 1$. For $n \geq 1$, the \emph{$n$-marker} is the entire map $\Box^n_A \to \wBox^n_A$, the image under the Yoneda embedding of the map $\varphi \colon [1]^n \to [1]^n_e$ in $\Box^+_A$.

A \emph{marked cubical set} is a structurally marked cubical set in which each cube has at most one marking.
We write $\cSet^+_A$ for the category of marked cubical sets.
Alternatively, we may view a marked cubical set as a pair $(X, eX)$ consisting of a cubical set $X$ together with a set of \emph{marked cubes}, i.e. a subset $eX \subseteq \bigcup\limits_{n \geq 1} X_n$ of cubes of positive dimension that includes all degenerate cubes, with a morphism of marked cubical sets being a map of cubical sets that preserves marked cubes.

The inclusion $\cSet^+_A \hookrightarrow \cSet^{++}_A$ admits a left adjoint $\Im \colon \cSet^{++}_A \to \cSet^+_A$; for $X \in \cSet^{++}_A$ the underlying cubical set of $\Im X$ coincides with that of $X$, with a cube of $\Im X$ being marked if and only if the corresponding cube of $X$ has at least one marking. Thus $\cSet^+_A$ is a reflective subcategory of $\cSet^{++}_A$. We will primarily work with $\cSet^+_A$, referring to $\cSet^{++}_A$ only when it is necessary or convenient to do so in the course of studying $\cSet^+_A$, such as when applying \cref{nat-weq}.


Given $X \in \cA^+$, we will occasionally denote its underlying cubical set by $|X| \in \cA$. Note that this defines a functor $|-| \colon \cA^+ \to \cA$; this functor and its adjoints will be studied further in Appendix \ref{app:inf-1}.

Note that we have an embedding $\BoxA^+ \to \cA^+$, given by the composite of the Yoneda embedding with the reflector $\Im$; this functor is fully faithful, as all representables are contained in $\cA^+$.

\begin{definition}
For $n \geq -1$, a marked cubical set is \emph{$n$-skeletal} if its underlying cubical set is $n$-skeletal (in the sense of \cref{skel-def-unmarked}). The \emph{regular $n$-skeleton} of a marked cubical set $X$, denoted $\sk_{n}X$, is the regular subcomplex of $X$ whose underlying cubical set is $\sk_{n}|X|$.
\end{definition}

Note that this technically conflicts with \cref{skel-def} when interpreted in terms of the Reedy structure of \cref{tBox-Reedy}; this choice of terminology, however, has the advantage of corresponding intuitively to the dimensions of non-degenerate cubes in a marked cubical set, and will not cause any confusion or ambiguity in our proofs.

\begin{definition}
A map $X \to Y$ in $\cSet^+_A$ is:

\begin{itemize}
\item \emph{regular} if it creates markings, i.e. a cube $x$ of $X$ is marked if and only if $f(x)$ is marked in $Y$;
\item a \emph{regular subcomplex inclusion} if it is regular and a monomorphism;
\item \emph{entire} if its underlying cubical set map $|X| \to |Y|$ is an isomorphism.
\end{itemize}
\end{definition}

When representing marked cubical sets visually, we will label marked 1-cubes with a tilde. Marked 2-cubes will typically be labelled with tildes in their interiors. For instance, the following diagram represents a marked 2-cube whose $(2,0)$-face is also marked.

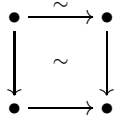
\[
\xymatrix{
\bullet \ar[r]^{\sim} \ar[d] \ar@{}[dr]|{\sim} & \bullet \ar[d] \\
\bullet \ar[r] & \bullet \\
}
\]

In cases where it is not feasible to label marked cubes in illustrations, markings will simply be described in accompanying text. 

We next study some convenient properties of the presheaf category $\cSet_A^{++}$ which descend to its reflective subcategory $\cSet^+_A$. We begin by noting that every marked cubical set is the colimit (in $\cSet^+$) of its cubes and markings.

\begin{proposition}\label{mcset-colim}
Every object $X \in \cSet^+_A$ may be constructed as
\[
\colim_{\Box^n_{A,(e)} \to X} \Box^n_{A,(e)}
\]
where the symbol $\Box^n_{A,(e)}$ represents an arbitrary representable marked cubical set, i.e. an arbitrary object of the form $\Box^n_A$ or $\wBox^n_A$, and the colimit is taken in $\cA^+$.
\end{proposition}

\begin{proof}
The analogous colimit construction in $\cA^{++}$ holds by a standard result about presheaf categories, and this colimit is preserved by the left adjoint functor $\Im$. That the stated colimit construction holds in $\cSet^+$ thus follows from the fact that both $X$ and all representables are contained in the full subcategory $\cSet^+$ and that the composite $\cSet^+ \hookrightarrow \cSet^{++} \xrightarrow{\Im} \cSet^+$ is naturally isomorphic to the identity. 
\end{proof}

The following result follows from a similar proof to the above, involving the standard construction of a monomorphism of presheaves on an EZ-Reedy category by skeletal induction.

\begin{proposition}[{\cite[Prop.~2.10]{campion-kapulkin-maehara}]}]\label{cell-model-marked}
The class of monomorphisms in $\cSet^+_A$ is the saturation of the set of boundary inclusions $\bd \Box^n_A \hookrightarrow \Box^n_A$ for $n \geq 0$ and markers $\Box^n_A \to \wBox^n_A$ for $n \geq 1$. \qed
\end{proposition}

The following family of functors relate general marked cubical sets to those which are \emph{$n$-trivial} for some $n$, i.e. those which have all of their cubes above dimension $n$ marked.

\begin{definition}
For $n \geq 0$, the \emph{$n$-trivialization} functor $\tau_n \colon \cSet^+_A \to \cSet^+_A$ sends $X \in \cSet^+_A$ to the marked cubical set $\tau_n X$ obtained by marking all cubes of $X$ of dimension greater than $n$. 

In both cases, for a map $f \colon X \to Y$, the map $\tau_n f$ acts identically to $f$ on underlying cubical sets.
\end{definition}

The geometric product of cubical sets admits a natural generalization to the marked case; we refer to this monoidal product as the \emph{lax Gray tensor product}, as it is a realization in the marked cubical setting of the more general higher categorical concept of the Gray tensor product (see \cite{campion-kapulkin-maehara} for further discussion of this).

\begin{definition}
For  $X, Y \in \cSet^+_A$, the \emph{lax Gray tensor product} $X \otimes Y$ is defined as follows:

\begin{itemize}
\item The underlying cubical set of $X \otimes Y$ is $|X| \otimes |Y|$, \ie the geometric product of the underlying cubical sets of $X$ and $Y$;
\item A cube $x \otimes y$ is marked if either $x$ is marked in $X$ or $y$ is marked in $Y$.
\end{itemize}
\end{definition}

As with the geometric product of unmarked cubical sets, this defines a non-symmetric, biclosed monoidal product $\otimes \colon \cSet^+_A \times \cSet^+_A \to \cSet^+_A$, with $\Box^0_A$ as its unit. 

Marked cubical sets admit a family of model structures modelling $(\infty,n)$-categories for any $n \geq 0$, including $n = \infty$. We now introduce these model structures, and the terminology needed to describe them. 

\begin{definition}\label{crit-face-def}
For $n \geq 1$, $1 \leq i \leq n$, and $\varepsilon \in \{0.1\}$, a monomorphism $\delta \colon [1]^m \to [1]^n$ in $\Box_A$ defines a \emph{critical face} of $\Box^n_A$ with respect to the face $\bd_{i,\varepsilon}$ if the standard form of $\delta$ does not contain any of the following sets of face maps:

\begin{enumerate}
\item\label{comical-middle} $\bd_{i,\varepsilon}$ or $\bd_{i,1-\varepsilon}$;
\item\label{comical-high} for some $j > i$, the face map $\bd_{j,\varepsilon}$, as well as $\bd_{k, 1-\varepsilon}$ for all $j > k > i$;
\item\label{comical-low} for some $j < i$, the face map $\bd_{j,\varepsilon}$, as well as $\bd_{k, 1-\varepsilon}$ for all $j < k < i$.
\end{enumerate}
\end{definition}

Note that if we specialize to the case $m = 1$ in the definition above, we obtain a unique \emph{critical edge} of $\Box^n_A$ with respect to $\bd_{i,\varepsilon}$, whose standard form is given by 
\[
\bd_{n,1-\varepsilon}\bd_{n-1,1-\varepsilon} \ldots \bd_{i+1,1-\varepsilon} \bd_{i-1,1-\varepsilon} \ldots \bd_{1,1-\varepsilon}
\]
This coincides with the definition of the critical edge given in \cite[Def.~1.21]{doherty-kapulkin-lindsey-sattler}.

\begin{definition}
We define certain objects and maps which play key roles in the model structures to be introduced.

\begin{itemize}
\item For $n \geq 1, 1 \leq i \leq n$, and $\varepsilon \in \{0,1\}$, the \emph{$(i,\varepsilon)$-comical cube}, denoted $\Box^n_{A,i,\varepsilon}$, is the marked cubical set whose underlying cubical set is $\Box^n_A$, with all critical faces with respect to $\bd_{i,\varepsilon}$ marked. The \emph{$(i,\varepsilon)$-comical open box}, denoted $\sqcap^n_{A,i,\varepsilon}$, is the regular subcomplex of $\Box^n_{A,i,\varepsilon}$ whose non-degenerate cubes consist of all faces of the $n$-cube other than the interior face $\id_{[1]^n}$ and $\bd_{i,\varepsilon}$. The \emph{$(i,\varepsilon)$-comical open box inclusion} is the regular subcomplex inclusion $\sqcap^n_{A,i,\varepsilon} \hookrightarrow \Box^n_{A,i,\varepsilon}$.
\item For $n \geq 2, 1 \leq i \leq n$, and $\varepsilon \in \{0,1\}$, the marked cubical set $(\Box^n_{A,i,\varepsilon})'$ is obtained from $\Box^n_{A,i,\varepsilon}$ by marking all $(n-1)$-dimensional faces except for $\bd_{i,\varepsilon}$. The \emph{$(i,\varepsilon)$-comical marking extension} is the entire map $\Box^n_{A,i,\varepsilon} \to \tau_{n-2} \Box^n_{A,i,\varepsilon}$, i.e. the map which marks the remaining $(n-1)$-dimensional face.
\item A map in $\cA^+$ is \emph{comical} if it is in the saturation of the set of comical open box inclusions and comical marking extensions. A \emph{comical fibration} is a map having the right lifting property with respect to all comical open box inclusions and comical marking extensions.
\item The marked cubical set $L_A$ is depicted below:
	
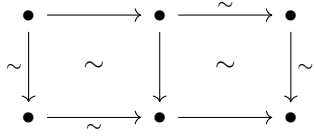
\[
	\begin{tikzcd} [column sep = large, row sep = large]
		\bullet
		\arrow [r]
		\arrow [d, swap, "\sim"]
		\ar [rd, phantom,"\sim"]
          &
		 \bullet
		\arrow [r, "\sim"]
		\arrow [d]
		\arrow [rd, phantom, "\sim"]
		&
		\bullet
		\arrow [d, "\sim"] \\
		\bullet
		\arrow [r, swap, "\sim"] 
		&
		\bullet
		\arrow [r]
		&
		\bullet
	\end{tikzcd}
	\]
	The \emph{elementary Rezk map} is the entire map $L_A \to \tau_0 L_A$, i.e. the entire map which marks the three unmarked edges of $L_A$. In general, a \emph{Rezk map} is any map of the form $(\bd \Box^m_A \hookrightarrow \Box^m_A) \hatotimes (L_A \to \tau_0 L_A) \hatotimes (\bd \Box^n_A \hookrightarrow \Box^n_A)$, where $\hatotimes$ denotes the pushout product.
\item For $n \geq 1$, the \emph{$n$-marker} is the entire map $\Box^n_A \to \wBox^n_A$ induced by the structure map $\varphi \colon [1]^n \to [1]^n_e$ in $\Box^+_A$.
\item A \emph{comical set} is a marked cubical set having the right lifting property with respect to comical open box inclusions and comical marking extensions. A comical set is \emph{saturated} if it has the right lifting property with respect to all Rezk maps, and \emph{$n$-trivial}, for $n \geq 0$, if it has the right lifting property with respect to all markers of dimension greater than $n$ (in other words, if all of its cubes of dimension greater than $n$ are marked).
\end{itemize}
\end{definition}

\begin{remark}
Note that the definition of the Rezk maps given above differs from that used in \cite{doherty-kapulkin-maehara}, in which there are four different elementary Rezk maps, denoted $L_{x,y} \to L_{x,y}'$ for $x, y \in \{1,2\}$. Our elementary Rezk map $L_A \to \tau_0 L_A$ is denoted in that paper by $L_{1,2} \to L_{1,2}'$. The construction used here, however, provides a simpler set of pseudo-generating trivial cofibrations. That the model structures defined here coincide with those of \cite{doherty-kapulkin-maehara} follows from \cite[Thm.~2.17]{doherty-kapulkin-maehara}.
\end{remark}

\begin{proposition}\label{comical-wfs}
For each $A \subseteq \{0,1\}$, comical maps and comical fibrations form the left and right class, respectively, of a weak factorization system on $\cSet_A^+$.
\end{proposition}

\begin{proof}
This follows from a standard application of the small object argument.
\end{proof}

\begin{theorem} \label{comical-model-structure}
Each category $\cSet^+_A$ carries the following model structures:

\begin{enumerate}
  \item \label{model-struct-comical} The \emph{comical model structure} in which
    \begin{itemize}
      \item cofibrations are monomorphisms;
      \item fibrant objects are comical sets;
      \item fibrations with fibrant codomain are characterized by the right lifting property with respect to comical
      open box inclusions and comical marking extensions.
    \end{itemize}
    
    \item \label{model-struct-comical-sat} The \emph{saturated comical model structure} in which
    \begin{itemize}
      \item cofibrations are monomorphisms;
      \item fibrant objects are saturated comical sets;
      \item fibrations with fibrant codomain are characterized by the right lifting property with respect to comical
      open box inclusions, comical marking extensions, and the Rezk maps.
    \end{itemize}
    
      \item \label{model-struct-comical-triv} An \emph{$n$-trivial comical model structure} for each $n \geq 0$, in which
    \begin{itemize}
      \item cofibrations are monomorphisms;
      \item fibrant objects are $n$-trivial comical sets;
      \item fibrations with fibrant codomain are characterized by the right lifting property with respect to comical
      open box inclusions, comical marking extensions, and markers $\BoxA^m \to \widetilde{\Box}_{A}^m$ for $m > n$.
    \end{itemize}
    
      \item \label{model-struct-comical-sat-triv} An \emph{$n$-trivial saturated comical model structure} for each $n \geq 0$, in which
    \begin{itemize}
      \item cofibrations are monomorphisms;
      \item fibrant objects are $n$-trivial saturated comical sets;
      \item fibrations with fibrant codomain are characterized by the right lifting property with respect to comical
      open box inclusions, comical marking extensions, Rezk maps, and markings $\BoxA^m \to \widetilde{\Box}_{A}^m$ for $m > n$.
    \end{itemize}
\end{enumerate}
Moreover, all of these model structures are monoidal with respect to the lax Gray tensor product. \qed
\end{theorem}

\begin{proof}
The construction of the model structures, monoidality with respect to the lax Gray tensor product, and characterizations of fibrations with fibrant codomain in the unsaturated model structures are given by \cite[Thm.~2.7]{doherty-kapulkin-maehara}. The characterizations of fibrations with fibrant codomain in the saturated model structures are given by \cite[Thm.~2.17]{doherty-kapulkin-maehara}.
\end{proof}

\begin{remark}
When viewing comical sets as higher categories, $n$-cubes represent $n$-dimensional cells, with the marked cubes corresponding to equivalences. The four classes of pseudo-generating trivial cofibrations then have the following higher-categorical meanings.
\begin{itemize}
\item The 1-dimensional comical open box inclusions are the inclusions of endpoints into the marked interval; thus marked edges may be lifted along comical fibrations, analogous to the lifting of isomorphisms along isofibrations in 1-category theory.
\item The $n$-dimensional $(i,\varepsilon)$-comical open box inclusions for $n \geq 2$ represent composition of $(n-1)$-dimensional cells, with the $(i,\varepsilon)$-face of the $(i,\varepsilon)$-comical cube representing the composite.
 They also ensure that every morphism presented by a marked edge has a left and right inverse, i.e., is an equivalence.
\item Lifting against the comical marking extensions ensures that a composite of equivalences is again an equivalence.
\item Lifting against the Rezk maps ensures that in a saturated comical set, all cubes which are invertible in a suitable sense are marked.
\item Lifting against markers of dimension greater than $n$ ensures that in an $n$-trivial comical set, all cubes of dimension greater than $n$ are equivalences.
\end{itemize}
\end{remark}

The categories of cubical sets with weak equivalences admit a natural simplicial analogue: we denote by $\sSet^+$ the category of \emph{marked simplicial sets}. Similarly to its cubical counterparts, this is a reflective subcategory of a presheaf category $\sSet^{++}$, and its objects may be thought of as simplicial sets in which some simplices of positive dimension, including all degenerate simplices, are marked, with morphisms being simplicial set maps which preserve markings. (Note that this differs from the category of marked simplicial sets defined in \cite{lurie:htt}; the marked simplicial sets defined in that reference have markings only on their edges.)

As with its cubical counterparts, $\sSet^+$ admits a family of model structures modeling $(\infty,n)$-categories for $n \geq 0$, including $n = \infty$ (the \emph{($n$-trivial, saturated) complicial model structures}). Each of these is defined analogously to its simplicial counterpart. Specifically, cofibrations in all cases are monomorphisms, with a set of pseudo-generating trivial cofibrations consisting of:

\begin{itemize}
\item horn inclusions with specified faces marked, representing composition and isofibration properties;
\item entire maps representing the principle that a composite of marked simplices is again marked;
\item in the saturated case, entire maps representing the principle that an invertible simplex is marked;
\item in the $n$-trivial case, entire maps representing the principle that a simplex above dimension $n$ is marked.
\end{itemize}

A full description of the complicial model structures is beyond the scope of this paper, but may be found in \cite{ozornova-rovelli}.


We next recall the extension of the triangulation functor to marked cubical sets, first developed in \cite{campion-kapulkin-maehara}. To describe this functor, we first need an explicit description of the simplices of $T\Box^n_A = (\Delta^1)^n = N[1]^n$. For $r \geq 0$, observe that since $\Delta^r = N[r]$ and the nerve functor is fully faithful, $r$-simplices $\Delta^r \to (\Delta^1)^n$ can be identified with order-preserving maps $\phi \colon [r] \to [1]^n$. Such a map $\phi$ can be identified with a unique function $\{1, \ldots ,n\} \to \{1, \ldots ,r,\pm \infty\}$, defined as follows:

\[
i \mapsto \left\{\begin{array}{cl}
+\infty, & \pi_i \circ \phi(r) = 0,\\
p, & \pi_i \circ \phi(p-1) = 0~\text{and }~\pi_i \circ \phi(p) = 1,\\
-\infty, & \pi_i \circ \phi(0) = 1.
\end{array}\right.
\]
  
Under this identification, a simplicial operator $\alpha \colon [q] \to [r]$ sends an $r$-simplex $\phi$ to the $q$-simplex $\alpha$ defined as follows:

\[
	(\phi   \alpha)(i) = \left\{\begin{array}{cl}
	+\infty, & \phi(i) > \alpha(q),\\
	p, & \alpha(p-1) < \phi(i) \le \alpha(p),\\
	-\infty, & \phi(i) \le \alpha(0).
	\end{array}\right.
	\]

We let $\iota_n$ denote the inclusion $\{1, \ldots ,n\} \to \{1, \ldots ,n,\pm \infty\}$, viewed as an $n$-simplex of $\Delta^n$.

\begin{definition}
We define the functor $T \colon \Box^+_A \to \sSet^{+}$ as follows:
\begin{itemize}
\item $T[1]^n$ has $(\Delta^1)^n$ as its underlying simplicial set, with an $r$-simplex $\phi \colon \{1, \ldots,n\} \to \{1, \ldots ,r,\pm \infty\}$ unmarked if and only if there exists a sequence $i_1 < \cdots < i_r$ in $\{1, \ldots ,n\}$ such that $\phi(i_p) = p$ for all $p \in \{1, \ldots ,r\}$;
\item $T[1]^n_e$ is obtained from $T[1]^n$ by marking the $n$-simplex $\iota_n$.
\end{itemize}
\end{definition}  
  
By left Kan extension, this definition extends to a colimit-preserving functor $T \colon \cSet^{++}_A \to \sSet^{+}$, with a right adjoint  $U \colon \sSet^{+} \to \cSet^{++}_A$, and these restrict to an adjunction $T : \cSet^+_A \rightleftarrows \sSet^+ : U$. From the definition, we can see that the only unmarked $n$-simplex of $T\Box^n$ is $\iota_n$, while all $n$-simplices of $T\widetilde{\Box}^n$ are marked.

In \cite{doherty-kapulkin-maehara}, triangulation was used to compare the (saturated,$n$-trivial) comical model structures on marked cubical sets with connections with the analogous complicial model structures on simplicial sets. Specifically, we have the following result, which establishes cubical sets with connections as models for $(\infty,1)$-categories.

\begin{theorem}[{\cite[Thm.~5.24]{doherty-kapulkin-maehara}}]\label{QE-with-cons}
For $A \in \{0,1,01\}$, the adjunction $T : \cSet_A^+ \rightleftarrows \sSet^+ : U$ is a Quillen equivalence between the (saturated, $n$-trivial) comical model structure on $\cSet_{A}^+$ and the (saturated, $n$-trivial) complicial model structure on $\sSet^+$. \qed
\end{theorem}

The methods of \cite{doherty-kapulkin-maehara} are not directly applicable to the category $\cSet_\varnothing^+$, as they make use of a functor $Q \colon \sSet^+ \to \cSet_A^+$ which cannot be defined in the absence of connections. Thus we will extend \cref{QE-with-cons} to $\cmin^+$ via an alternate method; namely, we will establish Quillen equivalences between $\cmin^+$ and the categories of cubical sets with connections which commute with triangulation, and apply the two-out-of-three property for Quillen equivalences.

\subsection{Adjunctions induced by inclusions}
We now focus on adjunctions relating our various categories of marked cubical sets. For $A \subseteq B \subseteq \{0,1\}$, we have an inclusion $i \colon \Box_A^+ \hookrightarrow \Box_B^+$; these inclusions form the following commuting diagram.
\begin{equation*}\label{inclusion-diagram}\tag{*}
\begin{tikzcd}
\Box^+_{\varnothing} \arrow[r] \arrow[d] & \Box^+_{0} \arrow[d] \\
\Box^+_{1} \arrow[r] & \Box^+_{01} \\
\end{tikzcd}
\end{equation*}
By pre-composition, each of these inclusions induces a functor $i^* \colon \cB^{++} \to \cA^{++}$. 
For $X \in \cSet_{B}^{++}$ and $n \geq 0$, we have $(i^* X)_n = X_n$, with structure maps computed as in $X$, and markings likewise agreeing with those in $X$. However, if $A \neq B$, then some degenerate cubes of $X$ become non-degenerate in $i^* X$, namely those connections which are not in the image of any map in $\Box_A$. From this description we can see that $i^*$ restricts to a functor $i^* \colon \cB^+ \to \cA^+$. In the case where $A = B$ and, this functor is simply the identity.

From here until the end of this section, we fix a specific choice of $A \subseteq B \subseteq \{0,1\}$.

The functor $i^* \colon \cSet_B^{++} \to \cSet_A^{++}$ has both a left and a right adjoint, given respectively by left and right Kan extension of presheaves along the inclusion $\Box_A^+ \hookrightarrow \Box_B^+$. These functors restrict to define a left and right adjoint of $i^* \colon \cSet_B^+ \to \cSet_A^+$. For our purposes, however, we will require more explicit descriptions of these adjoints, thus we will construct and characterize them directly.

\begin{proposition}\label{i-has-adjoints}
The functor $i^* \colon \cB^+ \to \cA^+$ admits both a left and a right adjoint.
\end{proposition}

\begin{proof}
Both $\cB^+$ and $\cA^+$ are locally presentable, as full subcategories of presheaf categories; by the adjoint functor theorem, it thus suffices to show that $i^*$ preserves all (small) limits and colimits. For $i^* \colon \cB^{++} \to \cA^{++}$, this holds because limits and colimits in presheaf categories are computed objectwise. 

Recall that, given a reflective subcategory $\catC$ of a complete and cocomplete category $\catD$, limits in $\catC$ are computed in $\catD$, while colimits in $\catC$ are computed by applying the reflector to colimits in $\catD$. It thus follows immediately that $i^* \colon \cB^+ \to \cA^+$ preserves limits, as limits in $\cA^+$ and $\cB^+$ are computed in the corresponding presheaf categories.

To show preservation of colimits, we first note that by the definition of the reflector $\Im$, the following diagram commutes:
\[
\begin{tikzcd}
\cB^{++} \arrow[r, "i^*"] \arrow[d, swap, "\Im"] & \cA^{++} \arrow[d, "\Im"] \\
\cB^+ \arrow[r, "i^*"] & \cA^+ \\ 
\end{tikzcd}
\]
Given a diagram $F \colon J \to \cA^{+}$, let $\colim_{+} F$ denote its colimit, and let $\colim_{++} F$ denote its colimit when considered as a diagram in $\cA^{++}$ (more precisely, the colimit of the composite $J \xrightarrow{F} \cA^{+} \hookrightarrow \cA^{++}$).  Denote colimits in $\cB^+$ and $\cB^{++}$ similarly. Then for all diagrams $F$ in $\cA^+$ or $\cB^+$ we have $\colim_{+} F = \Im \colim_{++} F$. Using the commutativity of the diagram above, and the fact that the functor $i^*$ between presheaf categories preserves colimits, we can compute, for any diagram $F \colon J \to \cB^{+}$:
\begin{align*}
\mathrm{colim}_{+} i^* F & =\Im \, \mathrm{colim}_{++} i^* F \\
& = \Im \, i^* \mathrm{colim}_{++} F \\
& = i^* \Im \, \mathrm{colim}_{++} F \\
& = i^* \mathrm{colim}_{+} F \\
\end{align*}

Thus we see that $i^* \colon \cB^+ \to \cA^+$ preserves colimits.
\end{proof}

Let $i_!, i_* \colon \cA^+ \to \cB^+$ denote the left and right adjoints of $i^*$, respectively. We will focus primarily on the adjunction $i_! \adjoint i^*$; thus we let $\eta$ and $\epsilon$ denote the unit and counit of this adjunction. We will now characterize the functors $i_!$ and $i_*$ explicitly. 

\begin{lemma}\label{i-left-colim}
For $X \in \cA^+$, the object $i_! X \in \cB^+$ is given by the colimit
\[
\colim_{\Box^n_{A,(e)} \to X} \Box^n_{B,(e)}
\]
Likewise, $i_!$ sends each map in $\cA^+$ to the induced map between colimits. In particular, for representable marked cubical sets we have a natural isomorphism $i_! \Box^n_{A,(e)} \cong \Box^n_{B,(e)}$.
\end{lemma}

\begin{proof}
We begin by proving the special case of representable marked cubical sets, i.e. that $i_! \Box^n_{A,(e)} = \Box^n_{B,(e)}$. For any $X \in \cB^+$, we have a natural bijection as follows:
\begin{align*}
\cB^+(i_! \Box^n_{A,(e)}, X) & = \cong \cA^+(\Box^n_{A,(e)}, i^* X) \\
& \cong \cB^+(\Box^n_{B,(e)}, X) \\
\end{align*}
Here the first isomorphism follows from the defining property of the adjunction $i_! \adjoint i^*$, while the second follows from the definition of $i^*$. 

By the Yoneda lemma, we thus obtain a natural isomorphism $i_! \Box^n_{A,(e)} \cong \Box^n_{B,(e)}$. Thus we have proven the stated result in the special case of a representable marked cubical set; the general result then follows from this special case, together with \cref{mcset-colim} and the fact that $i_!$ preserves colimits as a left adjoint.
\end{proof}

We can improve this characterization to one which provides an explicit description of the cubes of $i_! X$.

\begin{proposition}\label{i-left-explicit}
For $X \in \cA^+$, the object $i_! X$ is obtained by freely adding the connections of $\Box_B$ to $X$. More precisely, for $n \geq 0$ the $n$-cubes of $X$ consist of all expressions of the form $x \phi$, where $x$ is a non-degenerate $m$-cube of $X$ for some $m \leq n$ and $\phi \colon [1]^n \to [1]^m$ is an epimorphism in $\Box_B$. Such a cube is marked if either $\phi \neq \id$ or $\phi = \id$ and $x$ is marked in $X$. Structure maps are computed using the cubical identities and the structure maps of $X$.

For a map $f \colon X \to Y$ in $\cmin$, $i_! f$ sends $x \phi$ to $(fx) \phi$.
\end{proposition}

\begin{proof}
For a (possibly degenerate) cube $x \colon \BoxA^m \to X$, we have a map $\BoxB^n \to i_! X$ in the colimit cone given by \cref{i-left-colim}; we denote this cube of $i_! X$ by $x \id_{[1]^m}$, or simply $x$. It follows that for every map $\phi \colon [1]^n \to [1]^m$ in $\BoxB$ we have an $n$-cube $x \phi$ in $i_! X$, marked if $\phi$ is a non-identity epimorphism, and that these assignments obey the cubical identities in the sense that for $\psi \colon [1]^m \to [1]^k$ in $\BoxB^n$ we have $(x \phi) \psi = x (\phi \psi)$. 

Next we will show that the structure maps of cubes of the forms defined above are consistent with those of $X$. For $x \colon \BoxA^m \to X$ and $\phi \colon [1]^n \to [1]^m$ in $\BoxA$, the composite $\BoxB^n \xrightarrow{\phi} \BoxB^m \xrightarrow{x} i_! X$ is the map in the colimit cone corresponding to the cube $x \phi$ of $X$. In other words, the image of $x$, viewed as a cube of $i_! X$, under the structure map $\phi$ is precisely the cube of $i_! X$ corresponding to the cube $x \phi$ of $X$.

Now we note that by \cite[Lem.~1.23]{doherty-kapulkin-lindsey-sattler}, every map from a representable marked cubical set into $i_! X$ factors through the colimit cone. In particular, for every $y \colon \BoxB^n \to i_! X$, there is some $x \colon \Box^m_{A,(e)} \to X$ such that there exists a commuting diagram as depicted below:
\[
\begin{tikzcd}
\BoxB^n \arrow[d, swap, "y"] \arrow[d] \arrow[r] & \Box^m_{B,(e)} \arrow[dl, "x"] \\
i_! X & \\
\end{tikzcd}
\]
As every non-identity map into $[1]^m_e$ factors through $ [1]^m$, we may assume that the representable $\Box^m_{B,(e)}$ is $\Box^m_B$ for some $m \geq 0$. Thus we see that $y = x \phi$ for some $\phi \colon [1]^n \to [1]^m$ in $\BoxB$; in other words, every cube of $i_! X$ is of the form described above. Moreover, a cube of $i_! X$ is non-degenerate if and only if it corresponds to a non-degenerate cube of $X$; by the Eilenberg-Zilber lemma, it follows that the cubes of $i_! X$ admit the unique form given in the statement.

It remains to be shown that the markings in $i_! X$ are as described, \ie that a non-degenerate cube $x$ is marked in $i_! X$ if and only if it is marked in $X$. If $x$ is marked in $X$, \ie the cube $x \colon \BoxA^n \to X$ factors through $\wBox^n_A$, then it follows by \cref{i-left-colim} that the corresponding cube $x \colon \BoxB^n \to i_! X$ factors through $\wBox_B^n$, so that $x$ is marked in $i_! X$ as well.

On the other hand, suppose that $x$ is marked in $i_! X$, \ie that the map $x \colon \BoxB^n \to i_! X$ factors through $\wBox^n_B$. Once again applying the fact that every map from a representable marked cubical set into $i_! X$ factors through the colimit cone, for some map $\Box^m_{B,(e)} \to i_! X$ in the colimit cone there exists a commuting diagram as below:
\[
\begin{tikzcd}
\BoxB^n \arrow[dr,swap,"x"]  \arrow[r,"\varphi"] & \wBox_B^n \arrow[d] \arrow[r,"\psi"] & \Box^m_{B,(e)} \arrow[dl] \\
& i_! X & \\
\end{tikzcd}
\]
The only generating maps in $\BoxB^+$ having $[1]^n_e$ as their domain are those of the form $\zeta_i$ or $\xi_{i,\varepsilon}$; if $\psi$ factors through such a map, then by the cubical identities, the composite $\psi \varphi$ factors through $\sigma_i$ or $\gamma_{i,\varepsilon}$, contradicting the assumption that $x$ is non-degenerate. It follows that $\Box^m_{B,(e)} = \wBox^n_B$ and $\psi = \id$. Thus the map $\wBox^n_B \to i_! X$ above is itself part of the colimit cone, implying that $x \colon \BoxA^n \to X$ factors through $\wBox^n_A$.

Thus we have characterized the action of $i_!$ on objects; its action on morphisms then follows from \cref{i-left-colim}.
\end{proof}

From the two results above, we obtain immediate characterizations of the unit and counit of the adjunction $i_! \adjoint i^*$.

\begin{corollary}\label{unit-characterization}
For $A, B, \alpha$ as above:
\begin{itemize}
 \item For $X \in \cA^+$, $i^* i_! X$ is obtained by freely adding the connections of $\Box_B$ to $X$, and then forgetting the additional structure maps. The unit $\eta_X$ is the natural inclusion.
 \item For $X \in \cB^+$, the $n$-cubes of $i_! i^* X$ consist of all expressions of the form $x \phi$, where $x$ is an $m$-cube of $X$ which is not in the image of any epimorphism in $\Box_A$, and $\phi \colon [1]^n \to [1]^m$ is an epimorphism in $\cB^+$. The counit $\epsilon \colon i_! i^* X \to X$ sends such a cube $x \phi$ to the cube $x \phi$ in $X$, \ie the image of $X$ under $\phi$ regarded as a structure map of $X$. \qed
\end{itemize} 
\end{corollary}

Our characterization of the right adjoint $i_*$ will be fairly basic for now; in \cref{i-right-wcs}, we will characterize it further in terms of the theory of weak connection structures to be developed in \cref{section:wcf}.

\begin{proposition}\label{i-right-adj} 
For $X \in \cSet_A^+$, the object $i_* X \in \cB^+$ is defined as follows:
\begin{itemize}
\item for $n \geq 0$, $(i_* X)_n = \cSet_{A}(i^* \Box_B^n,X)$, with structure maps induced by pre-composition;
\item for $n \geq 1$, an $n$-cube of $i_* X$ is marked if and only if the corresponding map $i^* \Box_B^n \to X$ sends the $n$-cube $\id_{[1]^n}$ of $i^* \Box_B^n$ to a marked cube of $X$.
\end{itemize}
\end{proposition}

\begin{proof}
The characterization of $n$-cubes of $i_* X$ follows from a standard argument involving the Yoneda lemma. Likewise, marked $n$-cubes of $i_* X$ correspond to maps $\Box^n_B \to i_*X$ which factor through $\wBox^n_B$; under the adjunction $i^* \adjoint i_*$ these correspond to maps $i^* \Box^n_B \to X$ which factor through $i^* \wBox^n_B$. The given characterization of marked cubes thus follows from the fact that $i^* \wBox^n_B$ is obtained from $i^* \BoxB$ by marking the $n$-cube $\id_{[1]^n}$.
\end{proof}

\begin{proposition}\label{i-monoidal}
The functors $i_!, i^* \colon \cA^+ \to \cB^+$ are monoidal with respect to the lax Gray tensor product.
\end{proposition}

\begin{proof}
For $i^*$ this is immediate from \cref{geo-prod-description} and the definition of the lax Gray tensor product. For $i_!$, we note that $i_!$ preserves colimits as a left adjoint, while $\otimes$ similarly preserves colimits in each variable. Therefore, by \cref{mcset-colim}, it suffices to show that this holds on representables, which can be verified using \cref{i-left-colim}.
\end{proof}

Next we study some of the basic interactions of $i_!$ and $i^*$ with the comical model structures.

\begin{proposition}\label{i-left-on-gens}
The functor $i_! \colon \cA^+ \to \cB^+$ sends each of the following maps in $\cA^+$ to the corresponding map in $\cB^+$:

\begin{itemize}
\item boundary inclusions;
\item comical open box inclusions;
\item comical marking extensions;
\item Rezk maps;
\item markers.
\end{itemize}
\end{proposition}

\begin{proof}
We first note that for each $n$, the $n$-dimensional marker is the image under the Yoneda embedding of the map $\varphi \colon \Box^n \to \wBox^n$; the stated result for markers thus follows from \cref{i-left-colim}. Furthermore, we may note that to show the stated result for Rezk maps, it suffices by \cref{i-monoidal} to prove it for the elementary Rezk map.

We next give the proof for boundary inclusions.
For $n \geq 0$, the object $\bd \Box^n_A$ is the coequalizer of the following diagram:
\[
\bigsqcup\limits_{1 \leq j \leq i \leq n; \varepsilon, \varepsilon' \in \{0,1\}} \Box^{n-2}_A \rightrightarrows \bigsqcup_{1 \leq i \leq n; \varepsilon \in \{0,1\}} \Box^{n-1}_A
\]
Here one of the two maps between the coproducts acts on the component corresponding to a tuple $(i,j,\varepsilon,\varepsilon')$ as the composite $\iota_{j,\varepsilon'} \bd_{i,\varepsilon}$, where $\iota_{j,\varepsilon'}$ denotes the coproduct inclusion corresponding to the pair $(j,\varepsilon')$. Similarly, the other map acts on component $(i,j,\varepsilon,\varepsilon')$ as the composite $\iota_{i+1,\varepsilon} \bd_{j,\varepsilon'}$. Taking the coequalizer thus amounts to taking one $(n-1)$-cube for each face of $\Box^n$ and identifying their $(n-2)$-dimensional faces according to the cubical identities. 

The inclusion $\bd \Box^n_A \hookrightarrow \Box^n_A$ is the map out of the coequalizer induced by the map $\bigsqcup_{1 \leq i \leq n; \varepsilon \in \{0,1\}} \BoxA^{n-1} \to \Box^n$ which acts on the $(i,\varepsilon)$-component as the face map $\bd_{i,\varepsilon}$. It thus follows from \cref{i-left-colim} and the fact that $i_!$ preserves colimits as a left adjoint that $i_!$ sends this map to the analogously defined map in $\cB^+$, \ie $\bd \BoxB^n \hookrightarrow \BoxB^n$.

The proofs for inner, marked or comical open box inclusions, endpoint inclusions into $K$, the saturation map, and the elementary Rezk map are similar to the above; in all of these cases, the maps in question may be explicitly defined in terms of colimits of representables, and these descriptions are identical in $\cA^+$ and $\cB^+$. 
\end{proof}

\begin{corollary}\label{i-free-Quillen}
The adjunction $i_! : \cA^+ \rightleftarrows \cB^+ : i^*$ is Quillen with respect to the ($n$-trivial, saturated) comical model structure on $\cSet^+_A$ and the analogous model structure on $\cSet^+_B$.
\end{corollary}

\begin{proof}
It suffices to show that $i_!$ sends the generating cofibrations of $\cA^+$ to cofibrations, and the pseudo-generating trivial cofibrations to trivial cofibrations. Both of these assertions follow from \cref{i-left-on-gens}.
\end{proof}

\begin{corollary}\label{i-create-fib}
A map $f \colon X \to Y$ be a map in $\cB^+$ is a fibration with fibrant codomain in the (saturated, $n$-trivial) comical model structure if and only if $i^* f$ is a fibration with fibrant codomain in the corresponding model structure on $\cSet_{A}^+$. In particular, an object $X \in \cB^+$ is fibrant if and only if $i^* X$ is fibrant.
\end{corollary}

\begin{proof}
It suffices to show that a map $f$ as in the statement has the right lifting property with respect to the pseudo-generating trivial cofibrations of $\cB^+$ if and only if $i^* f$ has the right lifting property with respect to those of $\cA^+$; this is immediate from \cref{i-left-on-gens}.
\end{proof}

\begin{lemma}\label{i-precomp-cofs}
The functor $i^* \colon \cB^+ \to \cA^+$ preserves and reflects monomorphisms.
\end{lemma}

\begin{proof}
This is immediate from the definition of $i^*$, as for any map $f$ in $\cB^+$, the underlying set maps of $f$ and $i^* f$ coincide.
\end{proof}

The interaction between the adjunctions $i_! \adjoint i_*$ and triangulation will play a key role in establishing marked minimal cubical sets as a model for $(\infty,n)$-category theory.

\begin{proposition}\label{i-T-commutes}
The following diagram of adjunctions commutes:
\[
\xymatrix{
\cSet_A^+ \ar@<1ex>[rr]^{T} \ar@{}|{\rotatebox{-90}{$\adjoint$}}[rr] \ar@<1ex>[ddr]^{i_!} \ar@{}|{\rotatebox{-140}{$\adjoint$}}[ddr] && \sSet^+ \ar@<1ex>[ll]^{U} \ar@<1ex>[ddl]^{U} \\
\\
& \cSet_B^+ \ar@<1ex>[uul]^{i^*} \ar@<1ex>[uur]^{T} \ar@{}|{\rotatebox{-40}{$\adjoint$}}[uur]
}
\]
\end{proposition}

\begin{proof}
By \cref{mcset-colim}, it suffices to verify that the left adjoints $T$ and $Ti_!$ agree on representable presheaves; this is immediate from the definition of the triangulation functors.
\end{proof}

Finally, we note the commutativity of the diagram \cref{inclusion-diagram} implies corresponding results for the induced adjunctions between categories of marked cubical sets, which we may summarize in the following commuting diagram of adjoint triples.

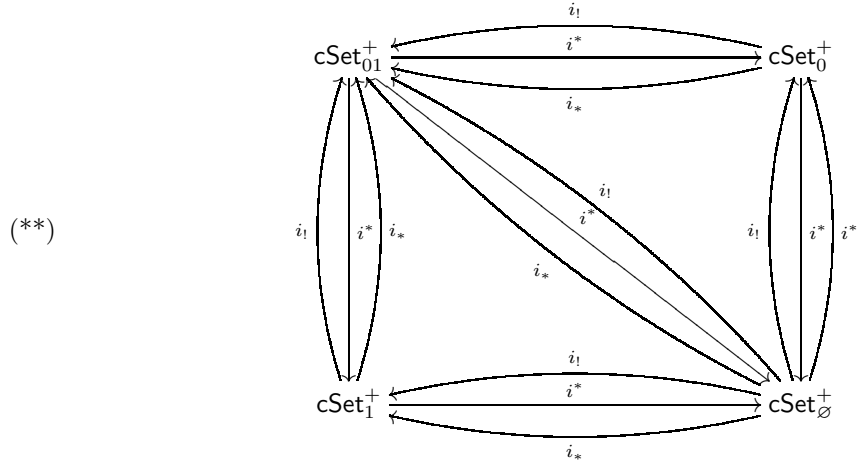
\begin{equation} \label{eq:all-functors-i}\tag{**}
\vcenter{\xymatrix@C+1.5cm{
  \cboth^+
  \ar[rr]^{i^*}
  \ar[dddd]^{i^*}
  \ar[rrdddd]^{i^*}
&
&
  \cneg^+
  \ar@/^1pc/[ll]^{i_*}
  \ar@/_1pc/[ll]_{i_!}
  \ar[dddd]^{i^*}
\\ 
&
&
\\
&
&
\\
\\
\cpos^+
 \ar[rr]^{i^*}
 \ar@/^1pc/[uuuu]^{i_!}
 \ar@/_1pc/[uuuu]_{i_*}
&
&
 \cmin^+
 \ar@/^1pc/[ll]^{i_*}
 \ar@/_1pc/[ll]_{i_!}
 \ar@/^1pc/[uuuu]^{i_!}
 \ar@/_1pc/[uuuu]_{i^*}
   \ar@/^1pc/[lluuuu]^{i_*}
  \ar@/_1pc/[lluuuu]_{i_!}
 \\
}}
\end{equation}

We will ultimately show that all of the adjunctions in this diagram are Quillen equivalences (assuming that the four categories of marked cubical sets depicted are equipped with analogous model structures); see \cref{i-Quillen-eqv}.

\section{Equivalence of comical and complicial model structures}\label{section:wcf}

\subsection{Weak connection structures}

This section is devoted to proving the following:

\begin{theorem}\label{T-min-Quillen-eqv}
The adjunction $T : \cmin^+ \rightleftarrows \sSet^+ : U$ is a Quillen equivalence between the (saturated, $n$-trivial) comical model structure on $\cmin^+$ and the (saturated, $n$-trivial) complicial model structure on $\sSet^+$.
\end{theorem}

Our proof will make use of an extensive study of connections in minimal comical sets, i.e. the construction of (generally non-degenerate) marked cubes in a  comical set in $\cmin^+$ which satisfy the face identities for connections. Specifically, these cubes will be studied via the framework of \emph{weak connection structures}, which we now define.

\begin{definition}
Let $A \subseteq B \subseteq \{0.1\}$. Given a map $f \colon X \to Y$ in $\cA^{+}$, a \emph{weak connection structure on $f$ (with respect to $\BoxB$)} is a map $\Gamma_f \colon i^* i_! X \to Y$ such that $\Gamma_{f} \circ \eta_X = f$, where $\eta$ denotes the unit of the adjunction $i_! : \cA^{+} \rightleftarrows \cB^{+} : i^*$. In the case where $f$ is the inclusion of a regular subcomplex $X \subseteq Y$ (including the case $X = Y$), we instead refer to a \emph{weak connection structure on $X$} or \emph{weak connection structure on $X$ in $Y$}, and use the notation $\Gamma_X$.
\end{definition}

\begin{remark}
Note that there exist versions of the adjoint triple $i_! \adjoint i^* \adjoint i_*$ for the categories of cubical sets studied in \cite{doherty-kapulkin-lindsey-sattler}, \ie those without markings, or with markings only on edges. Thus we can analogously define weak connection structures in those settings, and all of the theory to be developed here and in \cref{sec:surj} applies there as well. In particular, analogues of \cref{T-min-Quillen-eqv} hold for the cubical Joyal and cubical marked model structures; proving these results, by means of a Quillen equivalence with the saturated 1-trivial comical model structure, is the focus of Appendix \ref{app:inf-1}.
\end{remark}

Given a map $f \colon X \to Y$ in $\cA^+$, a weak connection structure $\Gamma_f$, and a map $g \colon Y \to Z$, it is immediate from the definition that the composite $g\Gamma_f$ is a weak connection structure on $gf$. Likewise, given a map $h \colon W \to X$ in $\cA^+$, the composite $\Gamma_f \circ i^* i_! h$ is a weak connection structure on $fh$; this follows from the naturality of $\eta$.

\begin{example}
Given $A \subseteq B \subseteq \{0,1\}, X \in \cA^+, Y \in \cB^+$, any map $f \colon i^* X \to Y$ in $\cB^+$ has a weak connection structure given by the composite $f \circ i^* \epsilon_X$, where $\epsilon$ denotes the counit of the adjunction $i_! : \cA^{+} \rightleftarrows \cB^{+} : i^*$. That this defines a weak connection structure is immediate from the triangle identities.
\end{example}

The concept of a weak connection structure is easiest to grasp intuitively in the case where $f$ is a regular subcomplex inclusion $X \hookrightarrow Y$. Then constructing a weak connection structure $\Gamma_X$ amounts to choosing, for each $m$-cube $x$ of $X$ and $\phi \colon [1]^n \to [1]^m$ in $\BoxB$, an $n$-cube $\Gamma_X(x \phi)$ in $Y$, marked if $\phi$ is not a monomorphism, in a manner compatible with the structure maps of $Y$; we may think of this as ``defining connections in $Y$ on the cubes of $X$'', and denote $\Gamma_X(x \phi)$ by $x(\phi)$. Compatibility with the structure maps of $Y$ then means precisely that for each cube $x \colon \Box^n \to X$, we have $x(\id_{[1]^n}) = x$, and for $\phi \colon [1]^m \to [1]^n$, $\psi \colon [1]^k \to [1]^m$ in $\BoxB$, if $\phi$ is in $\BoxA$ then we have $x\phi (\psi) = x(\phi \psi)$, while if $\psi$ is in $\BoxA$ then we have $x(\phi)\psi = x(\phi \psi)$. Note, however, that the analogous identity need not hold when neither map is in $\BoxA$; indeed, it may be that $x(\phi)$ is not contained in $X$, in which case the expression $x(\phi)(\psi)$ for $\psi$ in $\BoxB \setminus \BoxA$ is undefined. (Technically, if $x(\phi)$ is not in $X$ then $x(\phi)(\psi)$ should be undefined even if $\psi$ is in $\BoxA$, but for ease of notation, in this case we will define this to be the cube $x(\phi)\psi = x(\phi \psi)$ of $Y$.) Even in the case where $x(\phi)$ happens to be in $X$, when viewed as a cube of $i^* i_! X$ it is distinct from the cube $x \phi$ arising from the freely-added structure maps of $i_! X$, and has its own set of free structure maps, which need not be mapped to the same cubes of $Y$ as the free structure maps of $x$. Thus, for instance, if we take $A = \varnothing, B = \{0\}$, and consider a case in which $x (\gamma_{1,0}) \in X$, then we also have a cube $x(\gamma_{1,0})(\gamma_{2,0})$ in $X$, which need not be equal to $x(\gamma_{1,0} \gamma_{2,0})$. Likewise, neither of these cubes would necessarily be equal to $x(\gamma_{1,0})(\gamma_{1,0})$.

Another useful case to consider is that of a weak connection structure $\Gamma_x$ on an $n$-cube $x \colon \BoxA^n \to X$; similarly to the previous case, such a structure consists of a choice of an $m$-cube $x(\phi)$ in $X$ for each $\phi \colon [1]^m \to [1]^n$ in $\BoxB$, consistent with the structure maps of $X$. Here subtleties may arise if $x$ is not a monomorphism, for instance, if its faces are not distinct; given distinct non-degenerate $m$-cubes $d, d'$ of $\BoxA^n$, the free connections on $d$ and $d'$ in $i^* i_! \Boxmin^n$ are distinct, and may thus be mapped to distinct cubes of $X$ by $\Gamma_x$ even if $xd = xd'$. For instance, in the case $A = \varnothing, B = \{0\}$, consider a 2-cube $x \colon \Boxmin^2 \to X$ such that $x \bd_{1,0} = x \bd_{1,1}$. The 2-cubes $x(\bd_{1,0} \gamma_{1,0})$ and $x(\bd_{1,1} \gamma_{1,0})$ may be distinct, although the cubical identities imply that their boundaries will coincide. Therefore, unlike in the case of a weak connection structure on $X \subseteq Y$, here we must think of $x(\bd_{1,\varepsilon} \gamma_{1,0})$ not as ``the image of $x\bd_{1,\varepsilon}$ under $\gamma_{1,0}$'', but rather as ``the image of $x$ under $\bd_{1,\varepsilon} \gamma_{1,0}$''. For weak connection structures on cubes, any expression of the form $y(\phi)$ for $y \neq x$ is undefined. 

Nevertheless, by pre-composition we do obtain weak connection structures on all of  the chosen cubes $x(\phi)$ for  $\phi \colon [1]^m \to [1]^n$ in $\BoxB$, including the faces and degeneracies of $x$. To see how these are obtained, observe that $x(\phi)$, by definition, is the image under $\Gamma_X$ of $\phi$ viewed as an $m$-cube of $i^* \BoxB^n$. In other words, $x(\phi)$ corresponds to the composite $\Gamma_x \overline{\phi}$, where $\overline{\phi}$ denotes the adjunct map $\BoxA^m \to i^* i_! \BoxA^n = i^* \BoxA^n$ corresponding to $\phi \colon \BoxB^m \to \BoxB^n$ . By a standard result, $\overline{\phi}$ factors as $i^* \phi \circ \eta_{\Boxmin^m}$, as shown in the following diagram. 
\[
\xymatrix{
\BoxA^m \ar[d]_{\eta} \ar[dr]^{\overline{\phi}} & \BoxA^n \ar[r]^{x} \ar[d]_{\eta} & X \\
i^* \BoxB^m \ar[r]^{i^* \phi} & i^* \BoxB^n \ar[ur]_{\Gamma_x} \\
}
\]
We denote the composite $\Gamma_x \circ i^* \phi$ by $\phi^* \Gamma_x$. In the case where $\phi$ is in $\BoxA$, this coincides with the weak connection structure obtained by the pre-composition method for more general maps described above; this follows from the naturality of $\eta$ and the fact that in this case, the map $\phi \colon \BoxB^m \to \BoxB^n$ is the image under $i_!$ of $\phi \colon \BoxA^m \to \BoxA^n$.

In view of the discussion above, we see that a minimal cubical set equipped with a weak connection structure does not behave precisely like a cubical set with connections, as we might have hoped. Thus we have chosen the notation $x(\phi)$ for the image of a free connection $x\phi$ under $\Gamma_x$ to avoid any ambiguity which might arise from these counter-intuitive behaviours of weak connection structures.

Despite these issues, we are able to prove an analogue of the Eilenberg-Zilber lemma for weak connection structures. Our proof of this result will strongly resemble standard proofs of the Eilenberg-Zilber lemma for a general EZ-Reedy category, but additional care must be taken due to the issues discussed above.

\begin{definition}
Let $X \hookrightarrow Y$  be a regular subcomplex inclusion in $\cA^{+}$ equipped with a weak connection structure $\Gamma_X$ with respect to $\BoxB$. An $m$-cube $x \colon \BoxA^m \to Y$ in the image of $\Gamma_X$ is \emph{degenerate with respect to $\Gamma_X$} if it is equal to $x'(\phi)$ for some $x' \colon \BoxA^n \to X$ and some non-identity epimorphism $\phi \colon [1]^m \to [1]^n$ in $\BoxB$, and \emph{non-degenerate with respect to $\Gamma_X$} otherwise.
\end{definition}

Note that the degenerate cubes of a (marked) cubical set $X$ are automatically degenerate with respect to any weak connection structure on $X$, as for any $m$-cube $x$ and any $\phi \colon [1]^n \to [1]^m$  in $\BoxA$ we have $x(\phi) = x\phi$. Furthermore, note that being non-degenerate with respect to $\Gamma_X$ implies $x$ is contained in the subcomplex $X$.

\begin{lemma}\label{wcs-sections}
Let $X \hookrightarrow Y$  be a regular subcomplex inclusion in $\cA^{+}$ equipped with a weak connection structure $\Gamma_X$ with respect to $\BoxB$, and let $\phi \colon [1]^m \to [1]^n$, $\psi \colon [1]^k \to [1]^m$ be maps in $\BoxB$, with $\phi$ factoring as $\phi_1 \ldots \phi_p$ for some $p \geq 0$. Then there exist a monomorphism $\delta \colon [1]^a \to [1]^n$ and an epimorphism $\chi \colon [1]^k \to [1]^a$, factoring as $\chi_1 \ldots \chi_q$ with each $\chi_i$ an epimorphism, such that:

\begin{itemize}
\item $x(\phi_1) \ldots (\phi_p) (\psi) = x\delta(\chi_1) \ldots (\chi_q)$;
\item $\phi \psi = \delta \chi$.
\end{itemize}

\end{lemma}

\begin{proof}
We proceed by induction on $p$. In the base case $p = 0$ we have $\phi = \id$. Applying \cref{epi-mono-factor}, let $\lambda \psi'$ be the epi-monic factorization of $\psi$; then we see that $x (\psi) = x \lambda (\psi')$. Thus we may set $\delta = \lambda, \chi = \psi'$.

Now let $p \geq 1$, and suppose that the stated result holds for all $0 \leq p' \leq p-1$. Given $x, \phi, \psi$ as in the statement, we may again let $\lambda \psi'$ be the epi-monic factorization of $\psi$; likewise, let $\lambda' \phi_p'$ be the epi-monic factorization of $\phi_{p} \lambda$. By \cref{epi-or-mono}, $\lambda$ is in $\Boxmin \subseteq \BoxA$. Thus, using the induction hypothesis and the compatibility of $\Gamma_X$ with the structure maps of $\BoxA$, we can compute:

\begin{align*}
x (\phi_1) \ldots (\phi_p) (\psi) & = x (\phi_1) \ldots (\phi_p) (\lambda \psi') \\
& = x (\phi_1) \ldots (\phi_p) \lambda ( \psi') \\
& = x (\phi_1) \ldots (\phi_p \lambda) ( \psi') \\
& = x (\phi_1) \ldots (\phi_{p-1}) (\lambda \phi_p') (\psi') \\
& = x (\phi_1) \ldots (\phi_{p-1}) \lambda (\phi_p') (\psi') \\
\end{align*}

By the induction hypothesis, we may rewrite $x (\phi_1) \ldots (\phi_{p-1}) \lambda$ as $x \delta (\rho_1) \ldots (\rho_r)$ for some monomorphism $\delta$ and some epimorphisms $\rho_1,\ldots,\rho_r$. Thus we see that $x (\phi_1) \ldots (\phi_{p-1}) \lambda (\phi_p') (\psi') = x \delta (\rho_1) \ldots (\rho_q) (\phi_p') (\psi')$. Then each $\rho_i$ is an epimorphism, as are $\phi_p'$ and $\psi'$, and so the same is true of their composite, which we denote $\chi$. Moreover, by similar calculations to those shown above, we have $\delta \chi = \phi \psi$. Therefore, by induction, the result holds for all $p$.
\end{proof}

\begin{proposition}[Eilenberg-Zilber lemma for weak connection structures] \label{wcs-EZ}
Let $X \hookrightarrow Y$  be a regular subcomplex inclusion in $\cA^{+}$ equipped with a weak connection structure $\Gamma_X$ with respect to $\BoxB$. Then for every cube $x \colon \BoxA^m \to Y$ in the image of $\Gamma_X$, there exists a unique $x' \colon \BoxA^n \to X$ and a unique epimorphism $\phi \colon [1]^m \to [1]^n$ such that $x'$ is non-degenerate with respect to $\Gamma_X$ and $x = x'(\phi_1) \ldots (\phi_p)$ for some factorization $\phi = \phi_1 \ldots \phi_p$.
\end{proposition}

\begin{proof}
To prove the existence of $x'$ and $\phi$, we proceed by induction on $m$, the dimension of $x$. In the base case $m = 0$, the fact that there are no non-identity epimorphisms in $\BoxB$ with domain $[1]^0$ implies that $x$ is non-degenerate with respect to $\Gamma_X$. Thus we may take $x' = x, \phi = \id$.

Now let $m \geq 1$, and suppose the statement holds for all $m' \leq m$. If $x$ is non-degenerate with respect to $\Gamma_X$, then once again we may take $x' = x, \phi = \id$. Otherwise, by definition there exist some $k < m$, some $z \colon \BoxA^k \to X$, and some non-identity epimorphism $\psi \colon [1]^m \to [1]^k$ such that $x = z(\psi)$. By the induction hypothesis, there exist some $k' \leq k$, some $z' \colon \BoxA^{k'} \to X$, and some epimorphism $\psi' \colon [1]^{k} \to [1]^{k'}$ factoring as $\psi'_{1} \ldots \psi'_{q}$ such that $z = z'(\psi'_1) \ldots (\psi'_q)$. Thus we may take $x' = z'$ and $\phi = \psi' \psi$.

To show uniqueness of $x'$ and $\phi$, suppose that for some $z \colon \BoxA^m \to  X$, $z' \colon \BoxA^{m'} \to X$, both non-degenerate with respect to $\Gamma$, and some epimorphisms $\phi \colon [1]^n \to [1]^m$ factoring as $\phi_1 \ldots \phi_p$ and $\phi' \colon [1]^n \to [1]^{m'}$ factoring as $\phi_1' \ldots \phi_{p'}'$, the $n$-cubes $z(\phi_1) \ldots  (\phi_p)$ and $z'(\phi'_1) \ldots  (\phi'_{p'})$ are equal. We will show that $z = z'$ and $\phi = \phi'$.

Without loss of generality, assume $m \leq m'$, so that $p \geq p'$. By \cref{Box-Reedy}, $\phi'$ has a section $\lambda$. By \cref{wcs-sections}, we have $z'(\phi'_1) \ldots  (\phi'_{p'}) \lambda = z'$ (note $\lambda$ is in $\BoxA$ by \cref{epi-or-mono}). Likewise, if we let $\delta \psi$ denote the epi-monic factorization of $\phi \lambda$, then $z(\phi_1) \ldots  (\phi_p) \lambda = z \delta (\psi_1) \ldots (\psi_q)$ for some factorization of $\psi$ into epimorphisms $\psi_1 \ldots \psi_q$. By our assumption that $z$ is non-degenerate with respect to $\Gamma_X$, each $\psi_i$ must be the identity, so that $z' = z \delta$. Thus $\delta$ is a monomorphism $[1]^{m'} \to [1]^m$; our assumption that $m \leq m'$ then implies that $\delta$ is the identity, so that $m = m'$ and $z = z'$.

Furthermore, recalling that $\delta \psi = \phi \lambda$, we see that $\lambda$ is a section of $\phi$. As our choice of $\lambda$ was arbitrary, this implies that every section of $\phi'$ is a section of $\phi$. Moreover, since $m = m'$ (and likewise $p = p'$), we may apply the same argument to see that every section of $\phi$ is a section of $\phi'$, so that the sets of sections of $\phi$ and $\phi'$ coincide; by \cref{Box-Reedy} this implies $\phi = \phi'$.
\end{proof}

Finally, we note that weak connection structures can be used to describe the right adjoint $i_* \colon  \cA^{+} \to \cB^{+}$.

\begin{proposition}\label{i-right-wcs}
For $A \subseteq B \subseteq \{0,1\}$ and $X \in \cA^{+}$, an $n$-cube in $i_* X$ consists of an $n$-cube $x \colon \BoxA^n \to X$, together with a weak connection structure $\Gamma_x$ on $x$; such a cube is marked if and only if $x$ is marked. A structure map $\phi \colon X_n \to X_m$ sends such a pair $(x, \Gamma_x)$ to the pair $(x(\phi), \phi^* \Gamma_x)$.
\end{proposition}

\begin{proof}
By \cref{i-right-adj}, $n$-cubes $\BoxB^n \to i_* X$ correspond to maps into $X$ from $i^* \BoxB^n = i^* i_! \BoxA^n$, with structure maps acting by pre-composition. Each such cube can be identified as a weak connection structure $\Gamma_X$ on a unique $n$-cube of $x$ by pre-composing with $\eta$. By definition, the image of such a pair $(x, \Gamma_x)$ under a structure map $\phi$ is the composite $i^* i_! \BoxA^m \xrightarrow{i^* \phi} i^* i_! \BoxA^n \xrightarrow{x} X$, which this identification sends to the weak connection structure $\phi^* \Gamma_X$ on $x (\phi)$. The characterization of marked cubes then follows from \cref{unit-characterization}, which shows that $i^* i_! \wBox_A^n$ is obtained from $i^* i_! \BoxA^n$ by marking the cube $\id_{[1]^n}$.
\end{proof}


\subsection{Lifting weak connection structures}\label{section:lifting}

Our next goal is to prove the following result, which will play a key role in our proof of \cref{T-min-Quillen-eqv}:

\begin{proposition}\label{eta-tcof}
For every $X \in \cmin^+$, the adjunction unit $\eta \colon X \hookrightarrow  i^* i_! X$ is a trivial cofibration in each of the (saturated, $n$-trivial) comical model structures.
\end{proposition}

Our proof of this result will involve showing that in $\cmin^+$, weak connection structures on cubes may be lifted along comical fibrations. 
Thus, from here until the end of Section \ref{section:lifting} we will fix some non-empty subset $A \subseteq \{0,1\}$ and study the adjunction $i_! : \cmin^+ \rightleftarrows \cA^+ : i^*$.

In order for the proofs in this section to be applicable in full generality, we will essentially be working with $A = \{0,1\}$, and one may keep this case in mind for the sake of concreteness. Note, however, that our proofs are equally valid for $\cneg^+$ and $\cpos^+$, and can also be read with either of these categories in mind; this simply results in certain cases of various case analyses becoming vacuous. 

We begin with a result which, together with \cref{nat-weq}, will ultimately allow us to restrict our focus to the standard unmarked cubes.

\begin{proposition}\label{rep-unit-PO}
For all $A \subseteq \{0,1\}, n \geq 1$, the unit map $\eta_{\wBox^n_A}$ is the pushout of $\eta_{\Box^n_A}$ along the $n$-dimensional marker. In other words, the following diagram is a pushout:
\[
\xymatrix{
\Box^n_A \ar[d]_{\eta_{\BoxA^n}} \ar[r] & \wBox^n_A \ar[d]^{\eta_{\wBox_A^n}} \\
i^* i_! \Box^n_A \ar[r] & i^* i_! \wBox^n_A \pushoutcorner \\
}
\]
\end{proposition}

\begin{proof}
By \cref{unit-characterization}, it follows that $\eta_{\wBox^n_A}$  has the same underlying cubical set map as $\eta_{\Box^n_A}$, with the marked cubes of $i^* i_! \wBox^n_A$ being precisely $\id_{[1]^n}$, the degenerate cubes, and the free connections; thus $\id_{[1]^n}$ is the only unmarked cube of $i^* i_! \Box^n_A$ which is marked in $i^* i_! \wBox^n_A$. It follows that the diagram given in the statement is a pushout.
\end{proof}

As our proof of \cref{eta-tcof} will be heavily combinatorial, we first present a simple example to convey the basic idea of the proof; this is essentially a marked and relative version of the proof of \cite[Prop. 1.44]{doherty-kapulkin-lindsey-sattler}.

\begin{example}\label{simple-approximation-example}
Let $f \colon X \to Y$ be a fibration in $\cmin^+$, and suppose we are given a 1-cube $x \colon \Boxmin^1 \to X$ with vertices $x \bd_{1,0} = x_0, x\bd_{1,1} = x_1$, together with a weak connection structure (with respect to a cube category containing negative connections) $\Gamma_{f(x)}$ on $f(x)$ in $Y$. We wish to define a negative connection on $x$ in $X$ lying above the negative connection on $f(x)$ given by $\Gamma_{f(x)}$, i.e. to define a marked cube $x(\gamma_{1,0})$ as below, such that $f(x(\gamma_{1,0})) = f(x)(\gamma_{1,0})$.
\[
\begin{tikzcd}
x_0 \arrow[d,swap,"x"] \arrow[r,"x"] \arrow[phantom,dr,"\sim"] & x_1 \arrow[equal,d] \\
x_1 \arrow[equal,r] & x_1 \\
\end{tikzcd}
\]
The $(1,0), (1,1)$, and $(2,1)$-faces of this cube define a $(2,0)$-comical open box in $X$, whose image in $Y$ has a marked filler given by $f(x)(\gamma_{1,0})$. Lifting this filler along $f$, we obtain the marked cube $x(\widetilde{\gamma}_{1,0})$ depicted below:
\[
\begin{tikzcd}
x_0 \arrow[d,swap,"x"] \arrow[r,"x'"] \arrow[phantom,dr,"\sim"] & x_1 \arrow[equal,d] \\
x_1 \arrow[equal,r] & x_1 \\
\end{tikzcd}
\]
This cube may not be a valid connection on $x$, as its $(2,0)$-face is an arbitrary cube $x'$ obtained by lifting, which is not necessarily equal to $x$. Note, however, that the boundary of $x'$ coincides with that of $x$, and that $f(x') = f(x)$.

To ``correct'' $x(\widetilde{\gamma}_{1,0})$ to a connection on $x$, we may consider the following three-dimensional $(1,0)$-comical open box in $X$ (markings are omitted from the diagram for clarity, but all 2-cubes are marked):
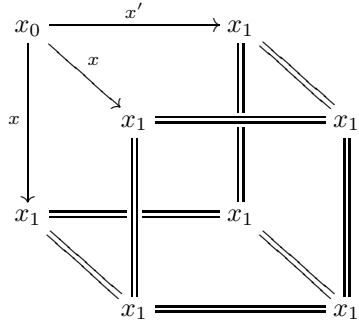
\[
\xymatrix{
x_0 \ar[rr]^{x'} \ar[dd]_{x} \ar[dr]^{x} &  & x_1 \ar@{=}[dd]|{\hole} \ar@{=}[dr] \\
& x_1 \ar@{=}[rr] \ar@{=}[dd] & & x_1 \ar@{=}[dd] \\
x_1 \ar@{=}[rr]|{\hole} \ar@{=}[dr] & & x_1 \ar@{=}[dr] \\
& x_1 \ar@{=}[rr] & & x_1 \\
}
\]
Here the top and back faces are equal to $x\widetilde{\gamma}_{1,0}$, while the left face is missing. The image of this open box in $Y$ has a marked filler given by $f(x)(\gamma_{1,0}\gamma_{2,0})$, which we may lift along $f$. Thus we obtain a marked filler for the open box depicted above; we may define its $(1,0)$-face to be $x(\gamma_{1,0})$. Moreover, all faces of this filler other than $x(\gamma_{1,0})$ are marked; as $f$ has the right lifting property with respect to comical marking extensions, it follows that $x(\gamma_{1,0})$ is marked as well. By construction, the image of $x(\gamma_{1.0})$ under $f$ is the $(1,0)$-face of $f(x)(\gamma_{1,0}\gamma_{2,0})$, i.e. $f(x)(\gamma_{1,0})$.
\end{example}

Our next result allows us to easily identify comical open boxes when dealing with weak connection structures.

\begin{lemma}\label{crit-edge-con-one}
For $1 \leq k \leq n$, $1 \leq i \leq n$, $\varepsilon \in \{0,1\}$ let $\delta$ denote a monomorphism $[1]^k \to [1]^{n+1}$ which defines a critical face with respect to either $\bd_{i,\varepsilon}$ or $\bd_{i+1,\varepsilon}$. Then the composite $\gamma_{i,\varepsilon} \delta \colon [1]^k \to [1]^n$ factors through a degeneracy or connection map $[1]^m \to [1]^{k-1}$.
\end{lemma}

\begin{proof}
We begin with the case where $\delta$ is critical with respect to $\bd_{i,\varepsilon}$; this implies that the standard form of $\delta$ does not contain $\bd_{i+1,\varepsilon}$, or any face map with index $i$. Denote the standard form of $\delta$ by $\bd_{a_1,\mu_1} \ldots \bd_{a_p,\mu_p} \bd_{b_1,\nu_1} \ldots \bd_{b_q,\nu_q}$, where $a_1 > \ldots > a_p > i$ and $i > b_1 > \ldots > b_q$.

For $r < p$ we have $a_r > a_p > i$, implying $a_r \geq i + 2$. Thus we may apply the cubical identities to compute:
\begin{align*}
\gamma_{i,\varepsilon} \delta & = \gamma_{i,\varepsilon} \bd_{a_1,\mu_1} \ldots \bd_{a_p,\mu_p} \bd_{b_1,\nu_1} \ldots \bd_{b_q,\nu_q} \\
& = \bd_{a_1 - 1,\mu_1} \ldots \bd_{a_{p-1} - 1,\mu_{p-1}} \gamma_{i,\varepsilon} \bd_{a_p,\mu_p} \bd_{b_1,\nu_1} \ldots \bd_{b_q,\nu_q} \\
\end{align*}

We first consider the case in which $a_p \geq i + 2$. Then we can compute:

\begin{align*}
\gamma_{i,\varepsilon} \delta & = \bd_{a_1 - 1,\mu_1} \ldots \bd_{a_{p-1} - 1,\mu_{p-1}} \gamma_{i,\varepsilon} \bd_{a_p,\mu_p} \bd_{a_p,\mu_p} \bd_{b_1,\nu_1} \ldots \bd_{b_q,\nu_q} \\
& = \bd_{a_1 - 1,\mu_1} \ldots \bd_{a_{p-1} - 1,\mu_{p-1}} \bd_{a_p - 1,\mu_p} \gamma_{i,\varepsilon} \bd_{b_1,\nu_1} \ldots \bd_{b_q,\nu_q}  \\
& = \bd_{a_1 - 1,\mu_1} \ldots \bd_{a_{p-1} - 1,\mu_{p-1}} \bd_{a_p - 1,\mu_p}  \bd_{b_1,\nu_1} \ldots \bd_{b_q,\nu_q} \gamma_{i - q,\varepsilon} \\
\end{align*} 

Now consider the case $a_p = i + 1$; then by assumption, we must have $\mu_p = 1- \varepsilon$. Thus we may compute:

\begin{align*}
\gamma_{i,\varepsilon} \delta & = \bd_{a_1 - 1,\mu_1} \ldots \bd_{a_{p-1} - 1,\mu_{p-1}} \gamma_{i,\varepsilon} \bd_{i+1,1-\varepsilon} \bd_{a_p,\mu_p} \bd_{b_1,\nu_1} \ldots \bd_{b_q,\nu_q} \\
& = \bd_{a_1 - 1,\mu_1} \ldots \bd_{a_{p-1} - 1,\mu_{p-1}} \bd_{i,1-\varepsilon} \sigma_i \bd_{a_p,\mu_p} \bd_{b_1,\nu_1} \ldots \bd_{b_q,\nu_q} \\
& = \bd_{a_1 - 1,\mu_1} \ldots \bd_{a_{p-1} - 1,\mu_{p-1}} \bd_{i,1-\varepsilon} \bd_{a_p,\mu_p} \bd_{b_1,\nu_1} \ldots \bd_{b_q,\nu_q} \sigma_{i-q} \\ 
\end{align*}

Next we consider the case in which $\delta$ is critical with respect to $(i+1,\varepsilon)$. Here we will again express $\delta$ in standard form as $\bd_{a_1,\mu_1} \ldots \bd_{a_p,\mu_p} \bd_{b_1,\nu_1} \ldots \bd_{b_q,\nu_q}$, where now we have $a_1 > \ldots > a_p > i + 1$ and $i + 1 > b_1 > \ldots > b_q$. Similarly to the previous case, we can compute:

\begin{align*}
\gamma_{i,\varepsilon} \delta & = \gamma_{i,\varepsilon} \bd_{a_1,\mu_1} \ldots \bd_{a_p,\mu_p} \bd_{b_1,\nu_1} \ldots \bd_{b_q,\nu_q} \\
& = \bd_{a_1 - 1,\mu_1} \ldots \bd_{a_p - 1,\mu_p} \gamma_{i,\varepsilon} \bd_{b_1,\nu_1} \ldots \bd_{b_q,\nu_q} \\
\end{align*}

If $b_1 < i$, then we may compute:

\begin{align*}
\gamma_{i,\varepsilon} \delta & = \bd_{a_1 - 1,\mu_1} \ldots \bd_{a_p - 1,\mu_p} \gamma_{i,\varepsilon} \bd_{b_1,\nu_1} \ldots \bd_{b_q,\nu_q} \\
& =   \bd_{a_1 - 1,\mu_1} \ldots \bd_{a_p - 1,\mu_p} \bd_{b_1,\nu_1} \ldots \bd_{b_q,\nu_q} \gamma_{i-q,\varepsilon} \\
\end{align*}

On the other hand, if $b_1 = i$ then we must have $\nu_1 = 1 - \varepsilon$, allowing us to compute:

\begin{align*}
\gamma_{i,\varepsilon} \delta & = \bd_{a_1 - 1,\mu_1} \ldots \bd_{a_p - 1,\mu_p} \gamma_{i,\varepsilon} \bd_{i,1-\varepsilon} \ldots \bd_{b_q,\nu_q} \\
& = \bd_{a_1 - 1,\mu_1} \ldots \bd_{a_p - 1,\mu_p} \bd_{i,1-\varepsilon} \sigma_i \bd_{b_2,\nu_2} \ldots \bd_{b_q,\nu_q} \\
& = \bd_{a_1 - 1,\mu_1} \ldots \bd_{a_p - 1,\mu_p} \bd_{i,1-\varepsilon} \bd_{b_2,\nu_2} \ldots \bd_{b_q,\nu_q} \sigma_{i - q + 1} \\
\end{align*}

Thus we see that $\gamma_{i,\varepsilon} \delta$ factors through a connection or degeneracy in all cases.
\end{proof}

Our proof of \cref{eta-tcof} will make extensive use of the standard form for maps in $\BoxA$; in fact, we will introduce a refinement of this standard form which emphasizes composites of connections with consecutive indices.

\begin{definition}
For $n \geq 1$, $1 \leq i \leq n$, $q \geq 0$, and $\varepsilon \in \{0,1\}$, the map $\gamma_{i:q,\varepsilon} \colon [1]^{n+q} \to [1]^{n}$ in $\BoxA$ is defined to be the composite $\gamma_{i,\varepsilon} \cdots \gamma_{i+q-1,\varepsilon}$. (In the case $q = 0$ we interpret this expression as the identity.)
\end{definition}

Viewing the maps in $\BoxA$ as poset maps, we see that $\gamma_{i:q,0}$ sends an object $(a_1,\ldots,a_{n+q}) \in [1]^{n+q}$ to $(a_1,\ldots,a_{i-1},\mathrm{max}(a_{i},\ldots,a_{i+q}),a_{i+q+1},\ldots,a_{n+q}) \in [1]^n$, and $\gamma_{i:q,1}$ acts similarly, taking a minimum rather than a maximum.

\begin{lemma}\label{ext-gamma-faces}
For $n \geq 1$, $1 \leq i \leq n$, $q \geq 1$, and $\varepsilon \in \{0,1\}$, the map $\gamma_{i:q,\varepsilon}$ composes with face maps as follows:

$\gamma_{i:q,\varepsilon} \partial_{j, \varepsilon'} =  \left\{ \begin{array}{ll}
\partial_{j, \varepsilon'} \gamma_{i-1:q,\varepsilon}   & \mathrm{for } \, j < i \mathrm{;} \\
\gamma_{i:q-1,\varepsilon} & \mathrm{for } \, i \leq j \leq i+q,  \, \varepsilon' = \varepsilon \mathrm{;} \\
\partial_{i, 1-\varepsilon} \sigma_i \ldots \sigma_{i+q-1} & \mathrm{for} \, i \leq j \leq i+q,  \, \varepsilon' = 1-\varepsilon \mathrm{;} \\
\partial_{j-q, \varepsilon'} \gamma_{i:q,\varepsilon} & \mathrm{for} \, j > i + q \mathrm{.} 
\end{array}\right.$
\end{lemma}

\begin{proof}
Immediate from the cubical identities.
\end{proof}

\begin{definition}\label{tail-form-def}
Let $\phi \colon [1]^{n+k} \to [1]^n$ be a (possibly empty) composite of connection maps in $\BoxA$. A \emph{tail form} of $\phi$ consists of a choice of $p, q \geq 0$ such that $p + q = k$; together with a pair of maps $\psi = \gamma_{i_1,\varepsilon_1}\ldots\gamma_{i_p,\varepsilon_p} \colon [1]^{n+p} \to [1]^{n}$ and $\gamma_{i:q,\mu} \colon [1]^{n+p+q} \to [1]^{n+p}$, such that:

\begin{itemize}
\item $\psi \gamma_{i:q,\mu} = \phi$;
\item if $p \neq 0$ then $j \geq i_p$, and if in addition $\varepsilon_p = \mu$ then $j \geq i_p + 2$.
\end{itemize}

A tail form is \emph{trivial} if $q = 0$, and \emph{non-trivial} otherwise.
\end{definition}

Given a non-trivial tail form $\phi = \psi \gjq$, we refer to the value $q - 1$ as the \emph{tail length} of $\phi$. Furthermore, we may note that the maximal index of $\phi$ is $j + q - 1$.

\begin{lemma}\label{tail-form}
Every non-empty composite of connection maps has a unique non-trivial tail form.
\end{lemma}

\begin{proof}
This is immediate from \cref{normal-form}; we may simply select the longest terminal segment with sequential indices and consistent sign (\ie positive or negative) in the standard form of $\phi$ given by that result, and group its maps together as $\gamma_{j:q,\varepsilon}$.
\end{proof}


\begin{remark}
In view of \cref{tail-form}, the existence of trivial tail forms may seem like an unnecessary complication, leading one to wonder why we did not simply require $q \geq 1$ in \cref{tail-form-def}. The answer is that trivial tail forms will be important for bookkeeping purposes in the proof of \cref{eta-tcof}.
\end{remark}

We next consider the effect of face maps on tail forms.

\begin{lemma}\label{crit-edge-con}
Given a non-trivial tail form $\psi \gjq$ and a monomorphism $\delta \colon [1]^k \to [1]^m$ which picks out a critical face with respect to $\bd_{l,\varepsilon}$ for some $j \leq l \leq j + q$, the composite $\phi \delta$ factors through a face or degeneracy map $[1]^k \to [1]^{k-1}$.
\end{lemma}

\begin{proof}
This follows from \cref{crit-edge-con-one}, together with the fact that for any $j \leq l \leq j + q - 1$, the map $\gamma_{j:q,\varepsilon}$ factors as $\gamma_{j,\varepsilon} \ldots \gamma_{l,\varepsilon} \ldots \gamma_{l,\varepsilon}$, where $\gamma_{l,\varepsilon}$ appears $q-l+j$ times in the composite.
\end{proof}

\begin{lemma}\label{face-tail}
Given a non-trivial tail form $\psi \gjq$, where $\psi$ is a composite of $p$ connection maps, the faces $\psi \gjq \bd_{k,\nu}$ may be characterized as follows. If $j \leq k \leq j+q$ and $\nu = \mu$, then $\psi \gjq \bd_{k,\nu} = \psi \gamma_{j:q-1,\mu}$. Otherwise, one of the following holds:

\begin{enumerate}
\item \label{higher-index} $\psi \gjq \bd_{k,\nu}$ is a composite of connection maps, and its maximal index is greater than or equal to $j + q - 1$, the maximal index of $\psi \gjq$;
\item \label{longer-tail} $\psi \gjq \bd_{k,\nu}$  is a composite of connection maps with maximal index $j + q - 2$, and its tail length is greater than or equal to $q-1$, the tail length of $\psi \gjq$;
\item \label{sf-face} the standard form of $\psi \gjq \bd_{k,\nu}$ contains a face map;
\item \label{sf-degen} the standard form of $\psi \gjq \bd_{k,\nu}$ contains a degeneracy map.
\end{enumerate}
\end{lemma}

\begin{proof}
If $j \leq k \leq j+q$, then this is immediate from \cref{ext-gamma-faces}. So suppose that either $k < j$ or $k > j + q$; then by \cref{ext-gamma-faces} it follows that $\psi \gjq \bd_{k,\nu} = \psi \bd_{k',\nu} \gamma_{j':q,\mu}$, where either $j' = j - 1$ and $k' = k$, or $j' = j$ and $k' = k - q$.

Consider the standard form of $\psi \bd_{k',\nu}$. If this standard form contains a face or degeneracy, then the same is true of $\psi \bd_{k',\nu} \gamma_{j':q,\mu}$ by \cref{face-shift,degen-shift}; thus it satisfies condition \ref{sf-face} or condition \ref{sf-degen}. Thus it remains only to consider the case in which the standard form of $\psi \bd_{k',\nu}$ contains only connections. We now assume this to be the case, and write this standard form as $\gamma_{l_1,\alpha_1} \ldots \gamma_{l_r,\alpha_r}$.

Thus $\psi \gjq \bd_{k,\nu} = \gamma_{l_1,\alpha_1} \ldots \gamma_{l_r,\alpha_r}\gamma_{j':q,\mu}$. If $l_r < j'$, or $l_r = j'$ and $\alpha_r = 1 - \mu$, then this expression is in standard form; in this case, we see that the maximal index of this map is $j' + q - 1 \geq j + q - 2$, and its tail length is at least $q - 1$, so that it satisfies either condition \ref{higher-index} or condition \ref{longer-tail}.

On the other hand, if $l_r > j'$ or $l_r = j'$ and $\alpha_r = \mu$, then we may repeatedly apply the cubical identity for composition of connections to rewrite the expression above as $\gamma_{l_1,\alpha_1} \ldots \gamma_{l_{r-1},\alpha_{r-1}} \gamma_{j':q,\mu} \gamma_{l_r + q,\alpha_r}$. Therefore, by \cref{con-shift}, the maximal index of $\psi \gjq \bd_{k,\nu}$ is greater than or equal to $l_r + q \geq j' + q \geq j + q - 1$. Thus condition \ref{higher-index} holds.
\end{proof}

We next introduce approximations of weak connection structures, a combinatorial tool which we will use in constructing weak connection structures on cubes in comical sets. These may be regarded as a generalization of the cubes which were constructed via lifting in \cref{simple-approximation-example}.

\begin{definition}\label{approx-def}
Given $X \in \cmin^+$ and $x \colon \Boxmin^n \to X$ equipped with a weak connection structure $\Gamma_x$, an \emph{approximation} of $\Gamma_x$ consists of an $m$-cube $x(\psi \gjqt)$ for each tail form expression $\phi = \psi \gjq$ for a (possibly empty) composite of connection maps $[1]^m \to [1]^n$ in $\BoxA$, whose boundary is given as follows:
\[
x (\psi \gjqt) \partial_{k, \epsilon} =  \left\{ \begin{array}{ll}
x (\psi \widetilde{\gamma}_{j:q-1,\mu}) & j+1 \leq k \leq j + q \, \mathrm{and} \, \varepsilon = \mu \\
x(\psi \gjq \bd_{k,\varepsilon}) & \mathrm{otherwise} \\
\end{array}\right.
\]
and such that $x (\psi \gjqt)$ is marked if $q \geq 0$.
\end{definition}

Note that in the case of a trivial tail form $\psi\gamma_{j:0,\mu}$, the conditions on the boundary of $x(\psi\widetilde{\gamma}_{j:0,\mu})$ given by the definition of an approximation reduce to the statement that it must coincide with the boundary of $x(\psi)$. Furthermore, we may note that even though $\gamma_{j:0,\mu}$ is equal to the identity regardless of the value of $j$, the cubes $x(\psi \widetilde{\gamma}_{j:0,\mu})$ may be distinct, and none of them need be equal to $x(\psi)$.

\begin{examples}\label{wcs-examples}
To illustrate the concept of an approximation, we consider a few examples.

\begin{itemize}
\item Any weak connection structure on a cube can be regarded as an approximation of itself, by setting $x(\psi\gjqt) = x(\psi\gjq)$ for all tail forms $\psi\gjq$.
\item Given a weak connection structure $\Gamma_x$ on $x \colon \Boxmin^n \to X$, an approximation $\widetilde{\Gamma}_x$ of $\Gamma_x$, and a map $f \colon X \to Y$, we may obtain an approximation $f\widetilde{\Gamma}_x$ of $f\Gamma_x$ by setting $f(x)(\psi \gjqt) = f(x(\psi \gjqt))$.
\item \cref{simple-approximation-example} shows a partial construction of a weak connection structure on a 1-cube $x$ together with an approximation. In this example, the 1-cube $x'$ is $x \widetilde{\gamma}_{1:0,0}$, the 2-cube $x \widetilde{\gamma}_{1,0}$ is $x \widetilde{\gamma}_{1:1,0}$, and the 3-cube constructed in the last step by filling the comical open box is $x \widetilde{\gamma}_{1:2,0}$.
\end{itemize}
\end{examples}

We now show that \cref{approx-def} is consistent on intersections.

\begin{lemma}\label{approx-consistent}
The face assignments given by \cref{approx-def} satisfy the cubical identity for composites of face maps.
\end{lemma}

\begin{proof}
We must show that for $i > i'$ and $\varepsilon, \varepsilon' \in \{0,1\}$, we have $x(\psi \gjqt)\bd_{i,\varepsilon} \bd_{i',\varepsilon'} = x(\psi \gjqt)\bd_{i',\varepsilon'}\bd_{i-1,\varepsilon}$. We proceed by case analysis.

The simplest case is when neither of the face maps $\bd_{i,\varepsilon}, \bd_{i',\varepsilon'}$ has its index between $j+1$ and $j+q$; then we have
\begin{align*}
 x(\psi \gjqt)\bd_{i,\varepsilon} \bd_{i',\varepsilon'} & =  x(\psi \gjqt\bd_{i,\varepsilon}) \bd_{i',\varepsilon'} \\
 & = x(\psi \gjq \bd_{i,\varepsilon} \bd_{i',\varepsilon'})
\end{align*}
 while
\begin{align*} 
x(\psi \gjqt) \bd_{i',\varepsilon'} \bd_{i-1,\varepsilon} & = x(\psi \gjq \bd_{i',\varepsilon'} )\bd_{i-1,\varepsilon} \\
& = x(\psi \gjq \bd_{i',\varepsilon'} \bd_{i-1,\varepsilon})
\end{align*}
and these are equal by the cubical identities.

Next we consider the case where $j+1 \leq i \leq j+q$ and $\varepsilon = \mu$, but the pair $(i',\varepsilon')$ does not satisfy the analogous criteria. Then we can compute:
\begin{align*}
x(\psi \gjqt)\bd_{i,\mu} \bd_{i',\varepsilon'} & = x(\psi \widetilde{\gamma}_{j:q-1,\mu}) \bd_{i',\varepsilon'} \\
& = x(\psi \gamma_{j:q-1,\mu} \bd_{i',\varepsilon'})
\end{align*}
(To see how the second step is obtained, note that because $i' < i \leq  j + q$, if $\varepsilon' = \mu$ then we must have $i' \leq j$ in order for our assumption on $(i',\varepsilon')$ to be satisfied.) Furthermore, we also have:
\begin{align*}
x(\psi \gjqt)\bd_{i',\varepsilon'}\bd_{i-1,\mu} & = x(\psi \gjq \bd_{i',\varepsilon'})\bd_{i-1,\mu} \\
& = x(\psi \gjq \bd_{i',\varepsilon'}\bd_{i-1,\mu}) \\
& = x(\psi \gjq \bd_{i,\mu} \bd_{i',\varepsilon'}) \\ 
& = x(\psi \gamma_{j:q-1,\mu} \bd_{i',\varepsilon'})
\end{align*}
Thus the stated result holds in this case.

Now we will consider the case where $j+1 \leq i' \leq j + q$ and $\varepsilon' = \mu$, but the corresponding statement does not hold for $(i,\varepsilon)$. Here we can compute:
\begin{align*}
x(\psi \gjqt)\bd_{i,\varepsilon} \bd_{i',\mu} & = x(\psi \gjq \bd_{i,\varepsilon})\bd_{i',\mu} \\
& = x(\psi \gjq \bd_{i,\varepsilon}\bd_{i',\mu}) \\
& = x(\psi \gjq \bd_{i',\mu} \bd_{i-1,\varepsilon}) \\
& = x(\psi \gamma_{j:q-1,\mu} \bd_{i-1,\varepsilon})
\end{align*}
And furthermore:
\begin{align*}
x(\psi \gjqt)\bd_{i',\mu}\bd_{i-1,\varepsilon} & = x(\psi \widetilde{\gamma}_{j:q-1,\mu})\bd_{i-1,\varepsilon} \\
& = x(\psi \gamma_{j:q-1,\mu} \bd_{i-1,\varepsilon}) \\
\end{align*}
(In the second step, we have used the fact that if $\varepsilon = \mu$ then $i$ is not between $j + 1$ and $j + 1$, implying that $i - 1$ is not between $j + 1$ and $j + q - 1$.) Thus the result holds in this case as well.

Finally, we consider the case where $j + 1 \leq i' < i \leq j+q$ and $\varepsilon = \varepsilon' = \mu$. (Note that this can only occur if $q \geq 2$.) Then we can compute:
\begin{align*}
x(\psi \gjqt)\bd_{i,\mu} \bd_{i',\mu} & = x(\psi \widetilde{\gamma}_{j:q-1,\mu}) \bd_{i',\mu} \\
& = x(\psi \widetilde{\gamma}_{j:q-2,\mu}) \\
\end{align*}
Here we have used the fact that $i' < i \leq j + q$, implying $i' \leq j + q - 1$. Furthermore, applying a similar argument to $i-1$, we can compute:
\begin{align*}
x(\psi \gjqt)\bd_{i',\mu}\bd_{i-1,\mu} & = x(\psi \widetilde{\gamma}_{j:q-1,\mu})\bd_{i-1,\mu} \\
& = x(\psi \widetilde{\gamma}_{j:q-2}) \\
\end{align*}
Thus we see that the statement holds in all cases.
\end{proof}

The following lemma shows that cubes of approximations satisfy the same comicality properties as the corresponding connections.

\begin{lemma}\label{approx-crit-face}
Let $x \colon \Box^m \to X$ denote an $m$-cube in a minimal marked cubical set, and $\psi \gjq$ a non-trivial tail-form of a map $\phi \colon [1]^n \to [1]^m$ in $\BoxA$. Suppose that $X$ contains cubes satisfying the identities and marking conditions defining all of the faces of $x(\psi \gjqt)$ prescribed by \cref{approx-def}, other than its $(k,\mu)$ face for some $j \leq k \leq j+q$. Then these faces define a $(k,\mu)$-comical open box in $X$.
\end{lemma}

\begin{proof}
That these faces define a $(k,\mu)$-open box follows from \cref{approx-consistent}. To show that this open box is $(k,\mu)$-comical, we must show that for any map $\delta \colon [1]^a \to [1]^n$ defining a non-identity critical face with respect to $\bd_{k,\varepsilon}$, the $a$-cube $x(\psi \gjqt) \delta$ is marked. We fix such a map $\delta$, and express $\delta$ in standard form as $\bd_{i_1,\varepsilon_1} \ldots \bd_{i_p,\varepsilon_p}$.

We first consider the case in which, for all $r$, we have $j + 1 \leq i_s \leq j+q$ and $\varepsilon_r = \mu$. By definition, a critical face with respect to $(k,\mu)$ cannot contain any face maps $\bd_{k-1,\mu}, \bd_{k,\mu}$, or $\bd_{k+1,\mu}$. If $q = 1$ or $q = 2$, or $q = 3$ and $k = j + 2$, then this rules out all possible indices between $j + 1$ and $j + q$, rendering this case vacuous as we have assumed $\delta$ is not the identity. Otherwise, this implies that $\delta$ is a composite of at most $q-2$ face maps, and we can compute that $x(\psi \gjqt) \delta = x(\psi \widetilde{\gamma}_{j:q-r,\mu})$, which is marked by assumption as $q - r \geq 2$.

Now consider the case where for some $r$, either $i_r$ is not between $j + 1$ and $j + q$, or $\varepsilon_r \neq \mu$. Then $x(\psi \gjqt) \bd_{i_r,\varepsilon_r} = x(\psi \gjq \bd_{i_r,\varepsilon_r})$; it follows that $x(\psi \gjqt) \delta = x(\psi \gjq \delta)$. That this cube is marked follows from \cref{crit-edge-con}. 
\end{proof}

Approximations of weak connection structures play a key role in the proof of the following result, in which we lift a weak connection structure on a cube along a comical fibration, assuming that such a lift already exists for its boundary; our approach will generalize that of \cref{simple-approximation-example}.

\begin{proposition}\label{PO-lift}
For all $n \geq 0$, the inclusion $\BoxPO \hookrightarrow i^* \BoxA^n$ is comical.
\end{proposition}

\begin{proof}
By \cref{comical-wfs}, it suffices to show that a lift $\Gamma_x$ exists in any diagram
\[
\xymatrix{
\BoxPO \ar[r]^(0.7){(x, \Gamma_{\bd x})} \ar[d] & X \ar[d]^{f} \\
i^* \BoxA^n \ar[r]^{\Gamma_{fx}} & Y
}
\]
in which the map $f \colon X \to Y$ is a comical fibration.

By the universal property of the pushout, we see that this diagram consists of an $n$-cube $x \colon \Boxmin^n \to X$, together with a weak connection structure $\Gamma_{\bd x}$ on $\bd x \colon \bd \Boxmin^n \to X$, and an extension of the induced weak connection structure $f \Gamma_{\bd x}$ on $f \bd x$ to a weak connection structure $\Gamma_{fx}$ on $fx$; defining a lift amounts to extending $\Gamma_{\bd x}$ to a weak connection structure $\Gamma_x$ on $x$ such that $f \Gamma_x = \Gamma_{fx}$. 

We will define $\Gamma_x$ together with an approximation $\widetilde{\Gamma}_x$, such that $f$ sends $\widetilde{\Gamma}_x$ to $\Gamma_{fx}$ viewed as an approximation of itself. Specifically, we must define the images under $\Gamma_x$ of the non-degenerate cubes of $i^* \BoxA^n$ which are not part of $i^* \bd \BoxA^n$. By \cref{face-shift,degen-shift}, these correspond to the maps $\phi \colon [1]^{n+k} \to [1]^n$ in $\BoxA$ for $k \geq 0$ whose standard forms consist of only connections, with no degeneracies or faces. By \cref{tail-form}, we may identify such maps $\phi$ with non-trivial tail forms $\psi \gjq$.

We proceed by induction on the dimension of these cubes; or equivalently, on the value $p + q$. Our induction hypothesis for $p + q = k$ will be that we have defined the following cubes:

\begin{itemize}
\item $x(\psi\gjq)$ for all tail forms such that $p + q \leq k$;
\item $x(\psi\gjqt)$ for all tail forms such that $p + q \leq k + 1$ and $q \geq 2$;
\item $x(\psi\gjqt)$ for all tail forms with $p = k - 1$ and $q \in \{0,1\}$.
\end{itemize}
and that for all such cubes which we have defined, $f(x(\psi\gjq)) = f(x)(\psi\gjq)$ and $f(x(\psi\gjqt)) = f(x)(\psi\gjqt)$.

Note that by \cref{tail-form}, for the first condition to be satisfied it suffices for it to hold for non-trivial tail forms. Note also that this condition implies that we have defined $x(\phi)$ for all $\phi \colon [1]^{n+k} \to [1]^n$  in $\BoxA$.

As our base case we take $k = 0$. Here the first statement is trivial, as the only cases to consider are the cubes $x(\gamma_{j:0,\mu})$; this reduces to $x(\id)$, which by definition must be equal to $x$. The latter two statements are vacuous in this case.

Now let $k \geq 1$, and suppose that the induction hypothesis is satisfied for $p + q = k - 1$. For the induction hypothesis to be satisfied for $p + q = k$, we must construct the cubes $x(\psi\gjq)$ for all tail forms such that $p + q \leq k$, the cubes $x(\psi\gjqt)$ for all tail forms such that $p + q = k + 1$ and $q \geq 2$, and the cubes $x(\psi\gjqt)$ for all tail forms with $p = k - 1$ and $q \in \{0,1\}$. We begin with the latter construction.

Consider a tail form $\psi \gamma_{j:1,\mu}$ where $p = k - 1$. By the definition of an approximation, the cube $\psi \widetilde{\gamma}_{j:1,\mu}$ must have the following faces:

\[
x (\psi \widetilde{\gamma}_{j:1,\mu}) \partial_{r, \varepsilon} =  \left\{ \begin{array}{ll}
x (\psi \widetilde{\gamma}_{j:0,\mu}) & r = j + 1 \, \mathrm{and} \, \varepsilon = \mu \\
x(\psi \gamma_{j,\mu} \bd_{r,\varepsilon}) & \mathrm{otherwise} \\
\end{array}\right.
\]


Thus the only missing face of $x (\psi \widetilde{\gamma}_{j:1,\mu})$ is its $(j+1,\mu)$-face; all others are equal to $x(\phi')$ for some $\phi' \colon [1]^{n+k-1} \to [1]^n$ and have thus been defined by the induction hypothesis. By \cref{approx-crit-face}, the faces of this cube which have already been defined form a $(j+1,\mu)$-comical open box in $X$. Moreover, the image of this open box under $f$ is the $(j+1,\mu)$-open box on $f(x)(\psi \widetilde{\gamma}_{j:1,\mu})$.

Thus we may lift $f(x)(\psi \gamma_{j:1,\mu})$ along the naive fibration $f$ to obtain a marked filler for this open box; we define this filler to be $x(\psi \widetilde{\gamma}_{j:1,\mu})$, and its $(j+1,\mu)$ face to be $x(\psi \widetilde{\gamma}_{j:0,\mu})$ (it follows from \cref{approx-consistent} that $x(\psi \widetilde{\gamma}_{j:0,\mu})$ has the correct faces). Then by construction,  $f$ sends $x(\psi \widetilde{\gamma}_{j:1,\mu})$ to $f(x)(\psi \gamma_{j:1,\mu})$, and $x(\psi \widetilde{\gamma}_{j:0,\mu})$ to $f(x)(\psi \gamma_{j:1,\mu}) \bd_{j+1,\mu} = f(x)(\psi \gamma_{j:0,\mu})$.

Next we will concern ourselves with the first two constructions specified by the induction hypothesis: $x(\psi \gjq)$ for all non-trivial tail forms such that $p + q \leq k$, and $x(\psi \gjqt)$ for all tail forms such that $p + q \leq k + 1$ and $q \geq 2$. We first note that constructing all cubes of these two forms is equivalent to constructing $x(\psi \gjq)$ and $x(\psi \widetilde{\gamma}_{j:q+1,\mu})$ for all tail forms of the first form specified. Moreover, the assumption that the induction hypothesis is satisfied for $k - 1$ allows us to restrict our attention to the case in which $p + q = k$. 

We will now construct $x(\psi \gjq)$ and $x(\psi \widetilde{\gamma}_{j:q+1,0})$ for all such tail forms, working by a two-layered induction. We first induct downwards on the maximal index of $\psi \gjq$, i.e. the value $j + q - 1$. More precisely, for $0 \leq m \leq n+k$, our induction hypothesis on $m$ will be that we have constructed $x(\psi \gjq)$ and $x(\psi \widetilde{\gamma}_{j:q+1,\mu})$ for all tail forms $\psi \gjq$ with $p + q = k$ and $j + q - 1 > m$. In the base case $m = n+k$ this statement is vacuous, as there are no connection maps $[1]^{n+k} \to [1]^{n+k-1}$ with index greater than $n+k$.
For a given maximal index $m$ satisfying the induction hypothesis, we proceed by induction on the value $q$. For a given $0 \leq q \leq k$, our induction hypothesis will be that we have constructed the desired cubes for all tail forms $\psi \gamma_{j:q',\mu}$ with $p + q' = k$, maximal index $m$, and $q' > q$; in the base case $q = k$ this condition is vacuous, as there are no admissible values of $q'$. Now fix some $1 \leq q \leq k$, and assume the induction hypothesis is satisfied for $q$; to show that it holds for $q-1$, we must construct the desired cubes for all tail forms $\psi \gjq$ with $p = k - q$.

By the definition of an approximation, the faces of $x(\psi \widetilde{\gamma}_{j:q+1,\mu})$ must be as follows:

\[
x (\psi \widetilde{\gamma}_{j:q+1,\mu}) \partial_{a, \varepsilon} =  \left\{ \begin{array}{ll}
x (\psi \widetilde{\gamma}_{j:q,\mu}) & j+1 \leq a \leq j + q +1 \, \mathrm{and} \, \varepsilon = \mu \\
x(\psi \gamma_{j:q+1,0} \bd_{a,\varepsilon}) & \mathrm{otherwise} \\
\end{array}\right.
\]

In particular, we must have $x (\psi \widetilde{\gamma}_{j:q+1,\mu}) \partial_{j,\mu} = x(\psi \gjq)$, which has not yet been constructed. On the other hand, we may analyze the other prescribed faces to show that they have already been constructed.

First consider faces of the form $x(\psi \gjqt)$. In the case $q \geq 2$, this exists by the induction hypothesis on $k-1$, while in the case $q = 1$ (implying $p = k - 1$) it is of the third form specified in the statement of the induction hypothesis on $k$, and hence has already been constructed.

Next consider faces of the form $x(\psi \gamma_{j:q+1,\mu} \bd_{a,\varepsilon})$ where either $1 \leq a \leq j$, $j + q + 2 \leq a \leq n + k$, or $\varepsilon = 1-\mu$. Note that $\psi \gamma_{j:q+1,\mu}$ has maximal index $m+1$ and tail length $q$, both one greater than those of $\psi \gamma_{j:q,\mu}$. Applying \cref{face-tail}, we see that one of the following must hold:

\begin{enumerate}
\item \label{face-case-max-ind} $\psi \gamma_{j:q+1,0} \bd_{a,\varepsilon}$ is a composite of connection maps whose maximal index is greater than or equal to $m+1$;
\item \label{face-case-tail} $\psi\gamma_{j:q+1,0} \bd_{a,\varepsilon}$ is a composite of connection maps with maximal index $m$ and tail length greater than or equal to $q$;
\item \label{face-case-face} the standard form of $\psi \gamma_{j:q+1,0} \bd_{a,\varepsilon}$ contains a face map;
\item \label{face-case-degen} the standard form of $\psi \gamma_{j:q+1,0} \bd_{a,\varepsilon}$ contains a degeneracy map.
\end{enumerate}

In each of these cases, $x(\psi \gamma_{j:q+1,\mu} \bd_{a,\varepsilon})$ has already been constructed, and $f$ maps it to the corresponding face of $f(x)(\psi \gamma_{j:q+1,\mu})$. In case \ref{face-case-max-ind}, this follows from the induction hypothesis on $m$. In case \ref{face-case-tail} it follows from the induction hypothesis on $q$. In case \ref{face-case-face}, $x(\psi \gamma_{j:q+1,\mu}) \bd_{a,\varepsilon}$ exists as part of $\Gamma_{\bd x}$, and is mapped to the corresponding face of $f(x)(\psi \gamma_{j:q+1,\mu})$ by assumption. Finally, in case \ref{face-case-degen}, $x(\psi \gamma_{j:q+1,\mu}) \bd_{a,\varepsilon}$ is a degeneracy of some lower-dimensional cube, hence it has been defined and is mapped to the corresponding face of $f(x)(\psi \gamma_{j:q+1,\mu})$ by the induction hypothesis on $k$.

Therefore, by \cref{approx-crit-face}, the faces of $x(\psi \widetilde{\gamma}_{j:q+1,\mu})$ which are already defined form a $(j,\mu)$-comical open box in $X$, which $f$ maps to the $(j,\mu)$-open box on $f(x)(\psi \gamma_{j:q+1,\mu})$. Thus we may lift $f(x)(\psi \gamma_{j:q+1,\mu})$ along $f$ to obtain a marked filler for this open box. We define this filler to be $x(\psi \widetilde{\gamma}_{j:q+1,\mu})$, and its $(j,\mu)$-face to be $x(\psi \gjq)$; by construction, we have $f(x(\psi \widetilde{\gamma}_{j:q+1,\mu})) = f(x)(\psi \gamma_{j:q+1,\mu})$ and $f(x(\psi \gjq)) = f(x)(\psi \widetilde{\gamma}_{j:q+1,\mu}) \bd_{j,\mu} = f(x)(\psi \gjq)$. That $x(\psi \gjq)$ has the necessary faces follows from \cref{approx-consistent}. Moreover, we may observe that all of its $(n + k - 1)$-dimensional faces other than its $(j,\mu)$-face are marked, as is the image of its $(j,\mu)$-face under $f$. Applying the right lifting property of $f$ with respect to comical marking extensions, we thus see that $x(\psi \widetilde{\gamma}_{j:q+1,\mu}) \bd_{j,\mu} = x(\psi \gjq)$ is marked.

By induction on $m$ and $q$, we thus construct $x(\psi \gjq)$ and $x(\psi \widetilde{\gamma}_{j:q+1,\mu})$ for all tail forms with $p + q = k$. By induction on $k$ we may, in particular, define suitable marked cubes $x(\psi \gjq)$ for all tail forms $\psi \gjq$. By \cref{tail-form}, we have thus defined $x(\phi)$ for all composites of connection maps $\phi$; in other words we have extended $\Gamma_{\bd x}$ to a weak connection structure $\Gamma_x$ lying over $\Gamma_{fx}$, thereby constructing a lift in the given diagram, as desired.
\end{proof}

\begin{proof}[Proof of \cref{eta-tcof}]
The components of the unit are monomorphisms by \cref{unit-characterization}; it thus remains to be shown that they are weak equivalences. By \cref{nat-weq}, it suffices to prove this in the case where $X$ is of the form $\Boxmin^n$ or $\wBox_{\varnothing}^n$. We proceed by induction on $n$; in the base case $n = 0$, the only map to be considered is $\eta_{\Boxmin^0}$, which is the identity on $\Boxmin^0$.

Now let $n \geq 1$ and suppose that all maps $\eta_{\Boxmin^m}$ and $\eta_{\wBox_{\varnothing}^m}$ with $m \leq n - 1$ are trivial cofibrations. By \cref{nat-weq}, it follows that the components of $\eta$ at all $(n-1)$-skeletal objects of $\cmin^+$ are trivial cofibrations; in particular, this holds for $\bd \Boxmin^n$. Thus the pushout inclusion $\Boxmin^n \hookrightarrow \BoxPO$ is a trivial cofibration, as a pushout of the trivial cofibration $\eta_{\bd \Boxmin^n}$ (recall that by \cref{i-left-on-gens}, $i_! \Boxmin^n = \BoxA^n$, and similarly for the boundaries).
\[
\xymatrix{
\bd \Boxmin^n \ar[d]_{\eta}^{\sim} \ar[r] & \Boxmin^n \ar[d]^{\sim} \\
i^* \bd \BoxA^n \ar[r] & \BoxPO \pushoutcorner \\
}
\]
Furthermore, the inclusion $\BoxPO \hookrightarrow i^* \BoxA^n$ is a trivial cofibration by \cref{PO-lift}. Thus the composite inclusion $\Boxmin^n \hookrightarrow i^* \BoxA^n$ is a trivial cofibration as well. That $\eta_{\wBox_{\varnothing}^n}$ is a trivial cofibration follows by \cref{rep-unit-PO}.
\end{proof}

\subsection{Quillen equivalences in the marked case}

Next we will prove the following:

\begin{theorem}\label{i-Quillen-eqv}
For all $A \subseteq B \subseteq \{0,1\}$, the adjunctions $i_! : \cA^+ \rightleftarrows \cB^+ : i^*$ and $i^* : \cB^+ \rightleftarrows \cA^+ : i_*$ are Quillen equivalences between the (saturated, $n$-trivial) comical model structures on $\cA$ and $\cB$.
\end{theorem}

The difficult combinatorial work involved in proving this result was done in the proof of \cref{PO-lift}; all that remains is some homotopical algebra involving the various adjunctions $i_! \adjoint i^*$ and $i^* \adjoint i_*$. Throughout the remainder of this section, fix a particular choice of one of the model structures of \cref{comical-model-structure}; we will assume that all categories of marked cubical sets are equipped with their respective versions of this model structure.

\begin{lemma}\label{composite-left-Quillen}
For any $A \subseteq \{0,1\}$, the composite functor $i^* i_! \colon \cmin^+ \to \cmin^+$ induced by $i \colon \Boxmin^+ \hookrightarrow \BoxA^+$ is left Quillen.
\end{lemma}

\begin{proof}
That $i^* i_!$ preserves cofibrations follows from \cref{i-free-Quillen,i-precomp-cofs}. To see that it preserves weak equivalences, let $X \to Y$ be a weak equivalence in $\cmin^+$. Then we have a commuting diagram:
\[
\xymatrix{
X \ar[d]_{\eta}^{\sim} \ar[r]^{\sim} & Y \ar[d]_{\eta}^{\sim} \\
i^* i_! X \ar[r] & i^* i_! Y \\
}
\]
The top horizontal map is a weak equivalence by assumption, while the two vertical maps are weak equivalences by \cref{eta-tcof}. Thus the bottom horizontal map is a weak equivalence by two-out-of-three.
\end{proof}

\begin{proposition}\label{i-cofree-Quillen}
For every $A \subseteq \{0,1\}$, the adjunction $i^* : \cA^+ \rightleftarrows \cmin^+ : i_*$ is Quillen.
\end{proposition}

\begin{proof}
The functor $i^*$ preserves cofibrations by \cref{i-precomp-cofs}, so it suffices to show that $i^*$ sends the pseudo-generating trivial cofibrations to trivial cofibrations. It follows from \cref{i-left-on-gens} that each pseudo-generating trivial cofibration in $\cA^+$ is the image under $i_!$ of the corresponding map in $\cmin^+$. Thus \cref{composite-left-Quillen} implies that $i^*$ sends each such map to a trivial cofibration in $\cmin^+$.
\end{proof}

\begin{proposition}\label{epsilon-weq}
For every $A \subseteq \{0,1\}$ and $X \in \cA^+$, the counit map $\epsilon \colon i_! i^* X \to X$, where $i_! \adjoint i^*$ is induced by $i \colon \Boxmin^+ \hookrightarrow \BoxA^+$, is a weak equivalence.
\end{proposition}

\begin{proof}
By \cref{nat-weq}, it suffices to prove this statement in the case where $X$ is either $\BoxA^n$ or $\wBox_A^n$ for some $n \geq 0$. By \cref{i-left-on-gens}, each of these objects is the image under $i_!$ of the corresponding object in $\cmin^+$.

Thus it suffices to show that the components of $\epsilon$ at objects in the image of $i_!$ are weak equivalences. To this end, let $X = i_! X_{\varnothing}$. It follows from the triangle identities that $i_! \eta_{X_\varnothing}$ is a section of $\epsilon_{X}$. This section is a weak equivalence by \cref{eta-tcof,i-free-Quillen}, thus the same is true of  $\epsilon_X$ by two-out-of-three.
\end{proof}

\begin{proof}[Proof of \cref{i-Quillen-eqv}]
All adjunctions of the form $i_! \adjoint i^*$ are Quillen by \cref{i-free-Quillen}. It remains to be shown that adjunctions of the form $i^* \adjoint i_*$ are Quillen, and that all of these adjunctions are Quillen equivalences.

We first consider the case $A = \varnothing$. Here $i^* \adjoint i_*$ is Quillen by \cref{i-cofree-Quillen}. By \cref{eta-tcof,epsilon-weq}, the unit and counit of the adjunction $i_! \adjoint i^*$ define natural weak equivalences $\eta \colon \id_{\cmin^+} \Rightarrow i^* i_!$ and $\epsilon \colon i_! i^* \Rightarrow \id_{\cB^+}$. That $i_! \adjoint i^*$ and $i^* \adjoint i_*$ are Quillen equivalences thus follows from \cref{nat-we-QE}.

Next consider the case where $A \neq \varnothing$; in the only non-trivial such cases, \ie those for which $A \subsetneq B$, we have $B = \{0,1\}$ and $A$ either $\{0\}$ or $\{1\}$. Here $i_! \adjoint i^*$ is a Quillen equivalence by \cref{QE-with-cons,i-T-commutes} together with the two-out-of-three property for Quillen equivalences.

Next we will show that $i^* \adjoint i_*$ is a Quillen adjunction. The functor $i^*$ preserves cofibrations by \cref{i-precomp-cofs}. To see that $i^*$ preserves weak equivalences, let $f$ be a weak equivalence in $\cboth^+$. We wish to show that the image of $f$ under $i^* \colon \cboth^+ \to \cA^+$ is a weak equivalence. We have shown that $i^* : \cA^+ \rightleftarrows \cmin^+ : i_*$ is a Quillen equivalence; therefore, by \cref{QuillenEquivCreate-original} and the fact that all cubical sets are cofibrant, $i^* \colon \cA^+ \to \cmin^+$ creates weak equivalences. Thus it suffices to show that the image of $f$ under the composite of $i^* \colon \cboth^+ \to \cA^+$ with $i^* \colon \cA^+ \to \cmin^+$ is a weak equivalence. But by the commutativity of the diagram \cref{eq:all-functors-i} this composite is precisely $i^* \colon \cboth^+ \to \cmin^+$; thus it sends $f$ to a weak equivalence by \cref{i-cofree-Quillen}. That $i^* \adjoint i_*$ is a Quillen equivalence then follows from the commutativity of \cref{eq:all-functors-i}, together with two-out-of-three for Quillen equivalences.
\end{proof}

\begin{corollary}\label{i-create-we}
For all $A \subseteq B \subseteq \{0,1\}$, the functors $i_! \colon \cA^+ \to \cB^+$ and $i^* \colon \cB^+ \to \cA^+$ preserve and reflect weak equivalences.
\end{corollary}

\begin{proof}
Because all marked cubical sets are cofibrant, this follows from \cref{QuillenEquivCreate-original,i-Quillen-eqv}.
\end{proof}

This result allows us to show that analogues of \cref{eta-tcof} hold for all cube categories under consideration.

\begin{proposition}\label{eta-tcof-all}
For all $A \subseteq B \subseteq \{0,1\}$, the unit of the adjunction $i_! \adjoint i^*$ induced by $i \colon \BoxA^+ \hookrightarrow \BoxB^+$ is a natural trivial cofibration.
\end{proposition}

\begin{proof}
Once again, that all components of the unit are monomorphisms follows from \cref{unit-characterization}. To see that they are also weak equivalences, note that the adjunction $i_! \adjoint i^*$ is a Quillen equivalence by \cref{i-Quillen-eqv}, and the right adjoint $i^*$ preserves and reflects weak equivalences by \cref{i-create-we}. The stated result thus follows by \cref{QuillenEquivCreate}.
\end{proof}

We have shown that each comical model structure on $\cmin^+$ is Quillen equivalent to its counterparts on $\cneg^+, \cpos^+$, and $\cboth^+$; it is then immediate that it is also Quillen equivalent to the corresponding complicial model structure on $\sSet^+$.

\begin{proof}[Proof of \cref{T-min-Quillen-eqv}]
This follows from \cref{QE-with-cons,i-Quillen-eqv,i-T-commutes}, together with the two-out-of-three property for Quillen equivalences.
\end{proof}

Finally, we note that \cref{eta-tcof-all} allows us to easily construct weak connection structures on maps between comical sets, as we will see in \cref{wcs-construction}. We will prove this result as a special case of the more general \cref{wcs-extension}, which will have further use in \cref{sec:surj}.

\begin{lemma}\label{wcs-extension}
Let $A \subseteq B \subseteq \{0,1\}$, and let $W \xrightarrow{j} X \xrightarrow{f} Y \xrightarrow{g} Z$ be a triple of maps in $\cA^+$, with $j$ a monomorphism and $g$ a fibration. Suppose that the map $fj$ and the object $Z$ admit weak connection structures $\Gamma_{fj}, \Gamma_Z$ such that the following diagram commutes:
\[
\xymatrix{
i^* i_! W \ar[r] \ar[d]_{\Gamma_{fj}} & i^* i_! Z \ar[d]^{\Gamma_Z} \\
Y \ar[r] & Z
}
\]
Then $f$ admits a weak connection structure $\Gamma_f$, such that $\Gamma_f \circ i^* i_! j = \Gamma_{jf}$ and the following diagram commutes:
\[
\xymatrix{
i^* i_! X \ar[r] \ar[d]_{\Gamma_f} &  i^* i_! Z  \ar[d]^{\Gamma_Z} \\
Y \ar[r] & Z \\
}
\]
\end{lemma}

\begin{proof}
By assumption, we have a diagram of the form:
\[
\xymatrix{
W \ar[r]^{j} \ar[d]_{\eta_W} & X \ar[d]^{f} \\
i^* i_! W \ar[r]^{\Gamma_{jf}} & Y \\
}
\]
This defines a map from the pushout $i^*i_! W \cup_{W} X$ into $Y$. Now consider the following diagram:
\[
\xymatrix{
W \ar[r]^{j} \ar[d]_{\eta_W} & X \ar[d] \ar@/^/[ddr]^{\eta_X} \\
i^* i_! W \ar[r] \ar@/_/[drr]_{i^* i_! j} & i^*i_! W \cup_{W} X \pushoutcorner \ar[dr] \\
&& i^* i_! X \\
}
\]
The map $\eta_W$ is a trivial cofibration by \cref{eta-tcof-all}, thus the same is true of its pushout $X \to i^*i_! W \cup_{W} X$. Likewise, $\eta_X$ is a trivial cofibration, so $i^*i_! W \cup_{W} X \to i^* i_! X$ is a weak equivalence by two-out-of-three. Moreover, by \cref{unit-characterization} and the fact that $j$ is a monomorphism, we may observe that the underlying cubical sets of $i^* i_! W$ and $X$ are subcomplexes of the underlying cubical set of $i^* i_! X$, and their intersection is precisely the underlying cubical set of $W$. It follows that $i^*i_! W \cup_{W} X \to i^* i_! X$ is a monomorphism, making it a trivial cofibration.

Therefore, by the assumption that $g$ is a fibration, a lift $\Gamma_f$ exists in the following diagram:
\[
\xymatrix{
i^*i_! W \cup_{W} X \ar[rr]^(0.65){(\Gamma_{jf},f)} \ar[d] & & Y \ar[d] \\
i^* i_! X \ar[r] & i^* i_! Z \ar[r]^{\Gamma_Z} & Z  \\
}
\]
Pre-composing with the inclusions of $i^* i_!W$ and $X$ into the pushout, the commutativity of the upper triangle implies that $\Gamma_f \eta_X = f$, so that $\Gamma_f$ is indeed a weak connection structure on $f$, and that $\Gamma_f \circ i^* i_! j = \Gamma_{jf}$. Commutativity of the diagram given in the statement is precisely commutativity of the lower triangle.
\end{proof}

\begin{corollary}\label{wcs-construction}
Given $A \subseteq B \subseteq \{0,1\}$, every map in $\cA^+$ whose codomain is a comical set admits a weak connection structure with respect to $\BoxB$.
\end{corollary}

\begin{proof}
Given a map $X \to Y$ with $Y$ a comical set, apply \cref{wcs-extension} to the sequence of maps $\varnothing \to X \to Y \to \BoxA^0$; the hypotheses of \cref{wcs-extension} are trivially satisfied in this case.
\end{proof}

\section{The essential image of $i^*$}\label{sec:surj}

In view of the results of \cref{section:wcf}, we might wish to conclude that ``comical sets have all connections'', i.e. that the theory of comical sets is independent of our choice of cube category; up to homotopy this is certainly true, as can be seen from \cref{i-Quillen-eqv}. Recall, however, that weak connection structures on comical sets do not behave exactly like connections: given a comical set $X \in \cA^+$ equipped with a weak connection structure with respect to $\BoxB$, it may not always be the case that $x(\phi)(\psi) = x(\phi\psi)$ for composable $\phi, \psi$ not contained in $\BoxA$. 

In this section we will consider a method for remedying this issue, and investigate the extent to which the \emph{combinatorial} theory of comical sets in $\cA^+$ is independent of $A$. We can make the statement that comical sets have all connections precise in a combinatorial sense by showing that for any $A \subseteq B \subseteq \{0,1\}$, all comical sets in $\cA^+$ are in the essential image of $i^* \colon \cB^+ \to \cA^+$. In other words, given any comical set $X \in \cA^+$, we can upgrade $X$ to a comical set in $\cB^+$ by a suitable choice of connections on its cubes. However, we will see that for a map of comical sets $f \colon X \to Y$, it is not always possible to choose connections on $X$ and $Y$ in such a way that $f$ will respect these connections. Nevertheless, this can be done if $f$ is either a cofibration or a fibration. 

Note that the results of this section may be viewed as a generalization of previously established results allowing the construction of connections in cubical Kan complexes and other cubical models of $\infty$-groupoids; see \cite[Prop.~2]{brown-higgins:T-complexes} and \cite[Thm.~3.4]{antolini:geometric-realisations}.

For the sake of generality, we will work in the comical model structure on $\cA^+$, with no assumptions of saturation or triviality; our results and proofs, however, generalize easily to the other model structures of \cref{comical-model-structure}. Moreover, they may similarly be adapted to the model structures for $(\infty,1)$-categories on unmarked cubical sets and cubical sets with marked edges to be discussed in Appendix \ref{app:inf-1}.

To prove our desired results, we will strengthen the concept of a weak connection structure. For the remainder of this section, fix $A \subseteq B \subseteq \{0,1\}$.

\begin{definition}\label{scs-def}
Let $X \in \cA^+$. A \emph{strong connection structure} on $X$ is a weak connection structure on $X$ such that for all $x \colon \BoxA^n \to X$ and all $\phi \colon [1]^m \to [1]^n$ , $\psi \colon [1]^k \to [1]^m$ in $\Box_B$, we have $x(\phi)(\psi) = x(\phi \psi)$.
\end{definition}

Recall from Examples \ref{wcs-examples} that for $X \in \cB^+$, the image under $i^*$ of the counit of the adjunction $i_! \adjoint i^*$ gives a weak connection structure on $i^* X$; this is a strong connection structure, as the connections of $X$ must satisfy $(x\phi)\psi = x(\phi\psi)$. Our next result shows that in fact, this class of examples encompasses all strong connection structures, up to isomorphism.

\begin{proposition}\label{i-scs}
Let $f \colon X \to Y$ be a map in $\cA^+$. Then $f$ is isomorphic to the image under $i^*$ of a map $\overline{f} \colon \overline{X} \to \overline{Y}$ in $\cB^+$ if and only if $X$ and $Y$ admit strong connection structures $\Gamma_X, \Gamma_Y$, such that the following diagram commutes:
\begin{equation}\label{eq:Gamma-square}\tag{$\dagger$}
\xymatrix{
i^* i_! X \ar[d]_{\Gamma_X} \ar[r]^{i^* i_! f} & i^* i_! Y  \ar[d]^{\Gamma_Y} \\
X \ar[r]^{f} & Y \\
}
\end{equation}
Moreover, in this situation, $\Gamma_X$ and $\Gamma_Y$ are isomorphic to the images under $i^*$ of the components of the counit of $i_! \adjoint i^*$ at $\overline{X}$ and $\overline{Y}$, respectively. That is to say, the following diagram commutes:

\begin{equation}\label{eq:counit-cube} \tag{$\dagger \dagger$}
\xymatrix{
i^* i_! i^* \overline{X} \ar[rr] \ar[dd]^(0.65){i^* \varepsilon} \ar[dr]^{\cong} &  & i^* i_! i^* \overline{Y} \ar[dd]|{\hole}^(0.65){i^* \varepsilon} \ar[dr]^{\cong} \\
& i^* i_! X \ar[rr]  \ar[dd]^(0.35){\Gamma_X} & & i^* i_! Y \ar[dd]^(0.35){\Gamma_Y} \\
 i^* \overline{X} \ar[rr]|(0.53){\hole} \ar[dr]^{\cong} & & i^* \overline{Y} \ar[dr]^{\cong} \\
&  X \ar[rr] & & Y \\
}
\end{equation}
\end{proposition}

\begin{proof}
If $f$ is in the image of $i^*$ then we may take $\Gamma_X$ and $\Gamma_Y$ to be the images of the relevant components of the counit, and the isomorphisms to be identities.

On the other hand, suppose that $X$ and $Y$ admit strong connection structures such that the diagram \cref{eq:Gamma-square} commutes. Then we may define the object $\overline{X} \in \cB^+$ by setting $\overline{X}_n = X_n$ for all $n$, with markings as in $X$; for $x \in X_n$ and $\phi \colon [1]^m \to [1]^n$ we set $x \phi = x(\phi)$. That this assignment is well-defined, that it satisfies the cubical identities, and that the connections thus chosen are marked all follow from the definition of a  strong connection structure.  We define $\overline{Y}$ similarly, and define $\overline{f}$ to act identically to $f$ in each dimension. That $\overline{f}$ respects connections follows from the commutativity of the diagram \cref{eq:Gamma-square}. 

Then $i^*\overline{X}$ and $i^*\overline{Y}$ are elements of $\cA^+$ having the same sets of cubes in each dimension as $X$ and $Y$, respectively, with identical markings, and $i^*\overline{f}$ is a morphism between them which acts identically to $f$ on these sets of cubes. Thus we may define isomorphisms $i^* \overline{X} \cong X, i^* \overline{Y} \cong Y$ by having them act as the identity in each dimension. That this defines a valid cubical set map follows from the fact that the maps of $\BoxA$ act identically on the cubes of $X$ and $\overline{X}$. This, in turn, follows from the definition of a weak connection structure: for any $\phi \colon [1]^m \to [1]^n$ in $\Box_A$ and $x \colon \BoxA^n \to X$ we have $x(\phi) = x \phi$.

To see that the diagram \cref{eq:counit-cube} commutes, consider the action of $i^* \epsilon$ on the cubes of $i^* i_! i^* \overline{X}$. By \cref{unit-characterization}, for a general $Z \in \cSet_A^+$, the cubes of $i^* i_! i^* Z$ are free connections $z \phi$, where $\phi$ is a map in $\BoxB$ and $z$ a cube of $Z$; then $i^* \varepsilon$ sends such a free connection to the cube $z \phi$ of $Z$, i.e. the image of $z$ under $\phi$ viewed as a structure map of $Z$. For a cube $x$ in $\overline{X}$, the image of $x$ under $\phi$, by definition, is the image of the free connection $x \phi$ under $\Gamma_X$. Thus the left and right faces of the diagram commute; commutativity of the remaining faces is trivial.
\end{proof}

Our next lemma allows for the lifting of strong connection structures along fibrations of comical sets.

\begin{lemma} \label{scs-extension}
Let $X \hookrightarrow Y$ be a monomorphism and $f \colon Y \to Z$ a fibration in $\cA^+$, with $Y$ and $Z$ comical sets. Suppose, moreover, that $X$ and $Z$ admit strong connection structures $\Gamma_X, \Gamma_Z$ such that the following diagram commutes:
\[
\xymatrix{
i^* i_! X \ar[r] \ar[d]_{\Gamma_X} & i^* i_! Z \ar[d]^{\Gamma_Z} \\
 X \ar[r] & Z \\
}
\]
Then $Y$ admits a strong connection structure $\Gamma_Y$, such that the diagram above factors as:
\[
\xymatrix{
i^* i_! X \ar[r] \ar[d]_{\Gamma_X} & i^* i_! Y \ar[r] \ar[d]_{\Gamma_Y} & i^* i_! Z \ar[d]_{\Gamma_Z} \\
 X  \ar[r] & Y \ar[r]^{f} & Z  \\
}
\]
\end{lemma}

\begin{proof}
Proceeding by induction on $n \geq -1$, we will define a sequence of monomorphisms $X \hookrightarrow Y^{-1} \hookrightarrow Y^0 \hookrightarrow \ldots \hookrightarrow Y$, such that for each $n$:

\begin{itemize}
\item $Y^n$ is a regular subcomplex of $Y$;
\item $Y^n$ contains all cubes of $Y$ of dimension less than or equal to $n$;
\item $Y^n$ admits a strong connection structure $\Gamma_{Y^n}$;
\item the given diagram factors as:
\[
\xymatrix{
i^* i_! X \ar[r] \ar[d]_{\Gamma_X} &i^* i_! Y^{-1} \ar[r] \ar[d]_{\Gamma_{Y^{-1}}}  &  \cdots  \ar[r] & i^* i_! Y^n \ar[r] \ar[d]_{\Gamma_{Y^n}} & i^* i_! Y \ar[r] & i^* i_! Z \ar[d]_{\Gamma_Z} \\
 X  \ar[r] & Y^{-1} \ar[r] & \cdots \ar[r] & Y^n \ar[r] & Y \ar[r]^{f} & Z  \\
}
\]
\end{itemize}

For the base case, we define $Y^{-1}$ to be the regular subcomplex of $Y$ whose underlying cubical set is $|X|$. We define $\Gamma_{Y^{-1}}$ to act identically to $\Gamma_X$ (more precisely, the underlying cubical set map of $\Gamma_{Y^{-1}} \colon i^* i_! Y^{-1} \to Y^{-1}$ coincides with that of $\Gamma_{X} \colon i^* i_! X \to X$). Then the diagram given in the statement factors as:
\[
\xymatrix{
i^* i_! X \ar[r] \ar[d]_{\Gamma_X} & i^* i_! Y^{-1} \ar[r] \ar[d]_{\Gamma_{Y^{-1}}} & Y \ar[r] & i^* i_! Z \ar[d]_{\Gamma_Z} \\
 X  \ar[r] & Y^{-1} \ar[r] & Y \ar[r]^{f} & Z  \\
}
\]

Thus the induction hypothesis is satisfied for $n = -1$. Now let $n \geq 0$, and suppose that we have defined $Y^{n'}$ satisfying the induction hypothesis for all $n' < n$. We define $W^n$ by the following pushout diagram:
\[
\xymatrix{
\sk_n Y^{n-1} \ar[d] \ar[r] & Y^{n-1} \ar[d] \\
\sk_n Y \ar[r] & W^n \pushoutcorner \\
}
\]
In other words, $W^n$ is the regular subcomplex of $Y$ generated by $Y^{n-1}$, together with all $n$-cubes of $Y$ which are not contained in $Y^{n-1}$.

Composing $\Gamma_{Y^{n-1}}$ with the inclusion $Y^{n-1} \hookrightarrow Y$, we obtain a weak connection structure on this inclusion. This weak connection structure fits into the following diagram, which commutes by the induction hypothesis:
\[
\xymatrix{
i^* Y^{n-1} \ar[d] \ar[r]^{\Gamma_{Y^{n-1}}} & Y^{n-1} \ar[r] & Y \ar[d]^{f} \\
i^* i_! Z \ar[rr]^{\Gamma_Z} && Z
}
\]
By assumption, $Y \to Z$ is a fibration, so we may apply \cref{wcs-extension} to the sequence of maps $Y^{n-1} \hookrightarrow W^n \hookrightarrow Y \to Z$. Thus we extend $\Gamma_{Y^{n-1}}$ to a map $\Gamma_{W^n} \colon i^* i_! W^n \to Y$ defining a weak connection structure on $W^n \hookrightarrow Y$, such that the following diagram commutes:
\begin{equation*}\label{eq:W-commutes}\tag{\textdaggerdbl}
\xymatrix{
i^* i_! W^n \ar[d]_{\Gamma_{W^n}} \ar[r] & i^* i_! Z \ar[d]^{\Gamma_Z} \\
Y \ar[r] & Z \\
}
\end{equation*}
Now define $Y^n$ to be the regular subcomplex of $Y$ consisting of all cubes in the image of $\Gamma_{W^n}$. Then $W^n \subseteq Y^n \subseteq Y$, as every cube of $W^n$ is the image under $\Gamma_{W^n}$ of the corresponding cube in $i^* i_! W^n$. By definition, $\Gamma_{W^n}$ factors through $Y^n$; we will extend it to a weak connection structure $\Gamma_{Y^n} \colon i^* i_! Y^n \to Y^n$.

By \cref{wcs-EZ}, for every cube $y$ of $Y^n$, there exists a unique cube $w$, non-degenerate with respect to $\Gamma_{W^n}$, and a unique epimorphism $\phi$ in $\BoxB$ such that $y = w(\phi_1) \ldots (\phi_p)$ for some factorization of $\phi$ into a composite of epimorphisms $\phi_1 \ldots \phi_p$. By the construction of $W^n$, such a $w$ must either be contained in $Y^{n-1}$, or be a non-degenerate $n$-cube of $Y$, not contained in $Y^{n-1}$.  In the former case, because $\Gamma_{Y^{n-1}}$ is a strong connection structure, we have $y = w(\phi)$. 

To see that a similar result holds in the latter case, we assume $w$ is a non-degenerate $n$-cube not contained in $Y^{n-1}$, and proceed by induction on $p$. In the base cases $p = 0, p = 1$ we have $\phi = \id, \phi = \phi_1$ respectively, so it is trivial that $y = w(\phi)$. Now let $\phi \geq 2$ and suppose this result holds for $p' < p$. Applying the induction hypothesis, we may rewrite $w(\phi_1) \ldots (\phi_p)$ as $w(\phi_1 \ldots \phi_{p-1})(\phi_p)$. If $\phi_p$ is in $\BoxA$, then we may further rewrite this as $w(\phi_1 \ldots \phi_{p-1}\phi_p) = w(\phi)$. On the other hand, if $\phi_p$ is not in $\BoxA$, then in order for this expression to be defined, $w(\phi_1 \ldots \phi_{p-1})$ must be a cube of $W^n$.  If $w(\phi_1 \ldots \phi_{p-1})$ were contained in $Y^{n-1}$, then we could apply a section of $\phi_1 \ldots \phi_{p-1}$ to deduce that $w$ was also contained in $Y^{n-1}$, contradicting our assumption; it thus follows that $w(\phi_1 \ldots \phi_{p-1})$ is not contained in $Y^{n-1}$, and must therefore be contained in $\sk_n Y$. As $w(\phi_1 \ldots \phi_{p-1})$ is of dimension $n + p - 1 > n$, it follows that this cube is the image of some non-degenerate cube of $Y$ under an epimorphism in $\BoxA$. Our assumption that $w$ is non-degenerate and the uniqueness condition of \cref{wcs-EZ} thus imply that $\phi_1 \ldots \phi_{p-1}$ is in $\BoxA$, allowing us to again see that $y = w(\phi_1 \ldots \phi_{p-1} \phi_p) = w(\phi)$.

Thus every $m$-cube of $Y^n$ is of the form $w(\phi)$ for some unique cube $w \colon \BoxA^{m'} \to W^{n-1}$, non-degenerate with respect to $\Gamma_{W^n}$, and some unique epimorphism $\phi \colon [1]^m \to [1]^{m'}$ of $\BoxB$. We define $\Gamma_{Y^n}$ as follows: for $w, \phi$ as above and $\psi \colon [1]^k \to [1]^m$ in $\BoxB$, we set $w(\phi)(\psi) = w(\phi \psi)$. That this respects the structure maps of $Y$ follows from the corresponding fact for $\Gamma_{W^n}$. That it defines a strong connection structure and agrees with $\Gamma_{{W^n}}$ on $i^* i_! W^n$ is immediate from the definition and the fact that $\Gamma_{W^n}$ restricts to the strong connection structure $\Gamma_{Y^{n-1}}$ on $i^* i_! Y^{n-1}$.

Lastly, we must prove the third point of the induction hypothesis; for this it suffices to show that each square of the following diagram commutes:
\[
\xymatrix{
i^* i_! Y^{n-1} \ar[r] \ar[d]_{\Gamma_{Y^{n-1}}} & i^* i_! Y^n \ar[d]_{\Gamma_{Y^n}} \ar[r] & i^* i_! Y \ar[r]^{i^* i_! f} & i^* i_! Z \ar[d]_{\Gamma_Z} \\
Y^{n-1} \ar[r] & Y^n \ar[r] & Y \ar[r]^{f} & Z \\
}
\]
Commutativity of the left-hand square follows from the fact that $\Gamma_{Y^n}$ extends $\Gamma_{W^n}$ by construction, and $\Gamma_{W^n}$ in turn extends $\Gamma_{Y^{n-1}}$. 

To show commutativity of the right-hand square, recall that every cube of $Y^n$ is of the form $w(\phi)$ for some cube $w$ of $W^{n}$ and some epimorphism $\phi$. Therefore, by \cref{unit-characterization}, the cubes of $i^* i_! Y^n$ are formal connections $x(\phi)\psi$. We can see that the top-right composite sends such a formal connection to the cube $f(w(\phi))(\psi)$, while the left-bottom composite sends it to $f(w(\phi)(\psi))$. Using the commutativity of the diagram \cref{eq:W-commutes} and the fact that both $\Gamma_{Y^n}$ and $\Gamma_Z$ are strong connection structures, we see that both of these expressions reduce to $f(w)(\phi \psi)$.

Therefore, by induction, we obtain a natural transformation of sequential diagrams, depicted below:
\[
\xymatrix{
i^* i_! X \ar[r] \ar[d]_{\Gamma_X} & i^* i_! Y^{-1} \ar[r]  \ar[d]_{\Gamma_{Y^{-1}}} & i^* i_! Y^0 \ar[d]_{\Gamma_{Y^0}} \ar[r] & \cdots \\
X \ar[r] & Y^{-1} \ar[r] & Y^0 \ar[r] & \cdots \\ 
}
\]
Moreover, as each $\Gamma_{Y^n}$ is a retraction of $\eta_{Y^n}$, this is a retraction of the natural transformation of sequential diagrams induced by $\eta$.

The colimit of the bottom sequence is the union of all regular subcomplexes $Y^n$; as each $Y^n$ contains the regular $n$-skeleton of $Y$, this union is $Y$ itself. Because both $i^*$ and $i_!$ preserve colimits as left adjoints, the colimit of the top sequence is $i^* i_! Y$. Let the induced map between colimits be denoted $\Gamma_Y \colon i^* i_! Y \to Y$; then the diagram given in the statement factors as specified. To see that $\Gamma_Y$ is a weak connection structure on $Y$, observe that since each $\Gamma_{Y^n}$ is a retraction of $\eta_{Y^n}$, the same is true of the induced maps between colimits, and the map between colimits induced by the $\eta_{Y^n}$ is $\eta_Y$ because both the identity functor and $i^* i_!$ preserve colimits. Finally, to see that $\Gamma_Y$ is a strong connection structure, consider a cube in $y$ of the form $y(\phi)(\psi)$ for $\phi, \psi$ in $\BoxB$. The cube $Y$ must be contained in some $Y^n$, implying $y(\phi)(\psi) = y(\phi \psi)$ by the fact that $\Gamma_{Y^n}$ is a strong connection structure.
\end{proof}

\cref{scs-extension} allows us to prove that $i^*$ is essentially surjective on fibrations and cofibrations between comical sets; the proofs of both of these statements arise as special cases of this lemma.

\begin{proposition}\label{i-surj-fib}
A map $X \to Y$ in $\cA^+$ is a fibration between comical sets if and only if it is isomorphic to the image under $i^*$ of a fibration between comical sets in $\cB^+$. In particular, an object $X \in \cA^+$ is a comical set if and only if it is the image under $i^*$ of a comical set in $\cB^+$.
\end{proposition}

\begin{proof}
The ``if'' direction follows from \cref{i-create-fib}. Now let $X \to Y$ be a fibration of comical sets in $\cA^+$. By \cref{i-scs}, to show that $X \to Y$ is in the essential image of $i^*$, it suffices to define strong connection structures $\Gamma_X, \Gamma_Y$ fitting into a commuting diagram
\[
\xymatrix{
i^* i_! X \ar[d]_{\Gamma_X} \ar[r] & i^* i_! i^* Y \ar[d]^{\Gamma_Y} \\
X \ar[r] & Y \\
}
\]
To construct $\Gamma_Y$, we may apply \cref{scs-extension} to the sequence of maps $\varnothing \hookrightarrow Y \to \BoxA^0$. The strong connection structure $\Gamma_Y$ then allows us to apply \cref{scs-extension} to the sequence of maps $\varnothing \hookrightarrow X \to Y$, thereby defining $\Gamma_X$ such that the diagram above commutes. This shows that $f$ is isomorphic to $i^*\overline{f}$ for some map $\overline{f}$ in $\cB^+$; that $\overline{f}$ is itself a fibration between fibrant objects follows from \cref{i-create-fib}.
\end{proof}

\begin{proposition}\label{i-surj-cof}
A map $X \to Y$ in $\cA^+$ is a cofibration between comical sets if and only if it is isomorphic to the image under $i^*$ of a cofibration between comical sets in $\cB^+$.
\end{proposition}

\begin{proof}
The ``if'' direction follows from \cref{i-precomp-cofs,i-create-fib}. Now let $X \hookrightarrow Y$ be a cofibration between comical sets in $\cA^+$. Once again, by \cref{i-scs}, to show that $X \to Y$ is in the essential image of $i^*$ it suffices to define strong connection structures $\Gamma_X, \Gamma_Y$ fitting into a commuting diagram
\[
\xymatrix{
i^* i_! X \ar[d]_{\Gamma_X} \ar[r] & i^* i_! i^* Y \ar[d]^{\Gamma_Y} \\
X \ar[r] & Y \\
}
\]
Here we first construct $\Gamma_X$ by applying \cref{scs-extension} to the sequence of maps $\varnothing \hookrightarrow X \to \BoxA^0$. We may then construct $\Gamma_Y$ such that the diagram commutes by applying \cref{scs-extension} to the sequence of maps $X \hookrightarrow Y \to \BoxA^0$.
\end{proof}

It is natural to ask whether \cref{i-surj-fib,i-surj-cof} can be extended to show that all maps between comical sets are in the essential image of $i^*$. Failing this, we might at least hope to show that this is true for equivalences of comical sets. Even this, however, turns out not to be the case (except in the trivial case $A = B$), as we will now show. Specifically, we will prove the following:

\begin{proposition}\label{i-not-surj}
For any $A \subsetneq B \subseteq \{0,1\}$, there exists an equivalence of comical sets in $\cA$ which is not isomorphic to any map in the image of $i^* \colon \cB^+ \to \cA^+$.
\end{proposition}

We will briefly describe some of the intuition behind the proof of this result. Given cubes $x, x'$ in a comical set $X \in \cB^+$ and a map $\phi$ in $\BoxB$, any map $f \colon X \to Y$ which identifies $x$ and $x'$ must also identify $x\phi$ and $x'\phi$. However, if $\phi$ is not in $\BoxB$ then the analogous statement for a comical set in $\cA^+$ equipped with a strong connection structure need not hold; $f$ may identify $x$ and $x'$ without identifying $x(\phi)$ and $x'(\phi)$. Thus it may not be possible to define strong connection structures on $X$ and $Y$ in a manner compatible with $f$. In the case where $f$ is a fibration this issue can be avoided, as the right lifting property of $f$ allows us to choose $x(\phi)$ and $x'(\phi)$ in such a way that they will be identified under $f$. On the other hand, in the case where $f$ is a cofibration the issue cannot arise, as a cofibration cannot identify distinct cubes at all.

\begin{proof}[Proof of \cref{i-not-surj}]
Let $X$ be a comical set in $\cA^+$ equipped with a cofibration $x \colon {\BoxA^1 \otimes \wBox_A^1} \hookrightarrow X$; for instance, we may take $X$ to be a fibrant replacement of $\BoxA^1 \otimes \wBox_A^1$. This map corresponds to a marked 2-cube $x \in X_2$ whose faces are all distinct, such that the edges $x \bd_{1,0}$ and  $x \bd_{1,1}$ are marked. Now let $X'$ denote the pushout object defined by the following diagram:
\[
\xymatrix{
\BoxA^1 \otimes \wBox_A^1 \ar[r] \ar[d] & X \ar[d] \\
\BoxA^1 \ar[r] & X' \pushoutcorner \\
}
\]
The left-hand map in this diagram is the projection map obtained by taking the lax Gray tensor product of the identity on $\BoxA^1$ with the unique map $\wBox_A^1 \to \BoxA^0$. Intuitively, the object $X'$ is obtained from $X$ by ``collapsing the cube $x$ down to an edge''. More precisely, it is the quotient of $X$ obtained by identifying cubes in the image of the monomorphism $x \colon \BoxA^1 \otimes \wBox_A^1 \hookrightarrow X$ whenever they correspond to cubes of $\BoxA^1 \otimes \wBox_A^1$ which are identified in $\BoxA^1$. In particular, we may note that the quotient map $X \to X'$ does not identify any cubes not in the image of $x$.

The map $\wBox_A^1 \to \BoxA^0$ is a weak equivalence, as a retraction of the one-dimensional comical open box inclusions, \ie endpoint inclusions into the marked interval. It thus follows  that $\BoxA^1 \otimes \wBox_A^1 \to \BoxA^1$ is a weak equivalence, by monoidality of the comical model structure. Thus the map $X \to X'$ is a pushout of a weak equivalence along a cofibration. The comical model structure is left proper, as all objects are cofibrant; it thus follows that $X \to X'$ is  a weak equivalence. 

Now let $Y$ denote a fibrant replacement of $X'$, with a trivial cofibration $X' \hookrightarrow Y$. Then the composite map $X \to Y$ is an equivalence of comical sets; denote this map by $f$. As $x \bd_{2,0}$ and $x \bd_{2,1}$ are identified by the quotient map $X \to X'$, we have, in particular, $f(x \bd_{2,0}) = f(x\bd_{2,1})$; denote this edge of $Y$ by $y$.

Now assume that $X$ is equipped with a strong connection structure, and let $\gamma_{1,\varepsilon} \colon [1]^2 \to [1]$ denote a connection map in $\BoxB$ which is not present in $\BoxA$. Consider the cubes $x\bd_{2,0}(\gamma_{1,\varepsilon}), x\bd_{2,1}(\gamma_{1,\varepsilon})$ provided by the strong connection structure on $X$. As $x \bd_{2,0}$ and $x \bd_{2,1}$ are distinct, these cubes must be distinct as well. Moreover, neither is in the image of the map $x$, as the fact that $\gamma_{1,\varepsilon}$ is not in $\BoxA$ implies that there can be no 2-cubes $\BoxA^2 \to \BoxA^1 \otimes \wBox_A^1$ satisfying the face identities for connections on the non-degenerate edges of $\BoxA^1 \otimes \wBox_A^1$. Thus the images of $x\bd_{2,0}(\gamma_{1,\varepsilon})$ and $x\bd_{2,1}(\gamma_{1,\varepsilon})$ in $X'$ are distinct, and the same is true of the images of these cubes under the monomorphism $X' \to Y$. Thus there can be no strong connection structure on $Y$ such that both $f(x\bd_{2,0}(\gamma_{1,\varepsilon}))$ and $f(x\bd_{2,1}(\gamma_{1,\varepsilon}))$ are equal to $y(\gamma_{1,\varepsilon})$. By \cref{i-scs}, it follows that $f$ is not isomorphic to any map in the image of $i^*$.
\end{proof}

\cref{i-not-surj} may seem surprising in light of \cref{i-surj-fib,i-surj-cof}, given that any map in a model category can be factored as a composite of a cofibration with a fibration. We may note, however, that given such a composite $X \hookrightarrow Y \to Z$, although $X \hookrightarrow Y$ and $Y \to Z$ are each  isomorphic to the image under $i^*$ of some map in $\cB^+$, these maps in $\cB^+$ may not themselves be composable. To put it another way, we may define connections on $X$ and $Y$ which $X \hookrightarrow Y$ will respect, and we may separately define connections on $Y$ and $Z$ which $Y \to Z$ will respect, but it may not be possible to perform these two procedures in such a way that the connections they define on $Y$ coincide.

\begin{appendix}

\section{Models of $(\infty,1)$-categories}\label{app:inf-1}

The theory of $(\infty,1)$-categories can also be modeled using unmarked cubical sets, or cubical sets with markings only on their edges; model structures suited to these purposes were constructed in \cite{doherty-kapulkin-lindsey-sattler}. As with the treatment of the comical model structures in \cite{doherty-kapulkin-maehara}, versions of these model structures were constructed for cubical sets with and without connections, but only those with connections were shown to be Quillen equivalent to the analogous simplicial models.

The results and proofs of the preceding sections generalize easily to these categories. One can, for instance, define weak connection structures on maps between unmarked cubical sets and use them to prove that the model structure on $\cmin$ constructed in \cite{doherty-kapulkin-lindsey-sattler} is Quillen equivalent to the Joyal model structure on $\sSet$.
Alternatively, one can establish a Quillen equivalence between each of these model structures and the saturated 1-trivial comical model structure on the corresponding category of marked cubical sets, and derive the desired Quillen equivalences as a consequence of \cref{T-min-Quillen-eqv}; it is this approach which we now pursue. Throughout this section, all categories of the form $\cA^+$ will be assumed to be equipped with the saturated 1-trivial comical model structure unless otherwise specified.

\subsection{The cubical Joyal and cubical marked model structures}

We first recall the model structures for $(\infty,1)$-categories developed in \cite{doherty-kapulkin-lindsey-sattler}, beginning with the cubical Joyal model structures on (unmarked) cubical sets. Recall that the \emph{critical edge} of the $n$-cube $\BoxA^n$ with respect to a face $\bd_{i,\varepsilon}$ is the unique edge which is a critical face with respect to $\bd_{i,\varepsilon}$ in the sense of \cref{crit-face-def}, \ie the edge $\bd_{n,1-\varepsilon} \ldots \bd_{i+1,1-\varepsilon} \bd_{i-1,1-\varepsilon} \ldots \bd_{1,1-\varepsilon} \colon [1] \to [1]^n$.

\begin{definition}
For $A \subseteq \{0,1\}$, we define certain objects and maps which play a key role in the development of the cubical Joyal model structure on $\cA$.
\begin{itemize}
\item For $n \geq 2$, $1 \leq i \leq n, \varepsilon \in \{0,1\}$, the $(i,\varepsilon)$-\emph{inner cube}, denoted $\widehat{\Box}^{n}_{A,i,\varepsilon}$, is the quotient of $\BoxA^n$ in which the critical edge with respect to $\bd_{i,\varepsilon}$ is made degenerate. The $(i,\varepsilon)$-\emph{inner open box}, denoted $\widehat{\sqcap}^{n}_{A,i,\varepsilon}$, is the subcomplex of $\widehat{\Box}^n_{A,i,\varepsilon}$ whose non-degenerate cubes consist of all faces of the $n$-cube other than the interior face $\id_{[1]^n}$ and $\bd_{i,\varepsilon}$. The $(i,\varepsilon)$-\emph{inner open box inclusion} is the inclusion $\widehat{\sqcap}^{n}_{A,i,\varepsilon} \hookrightarrow \widehat{\Box}^{n}_{A,i,\varepsilon}$.
\item The \emph{invertible interval} $K$ is the cubical set depicted below:
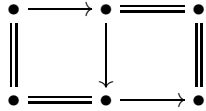
\[
\xymatrix{
  \bullet \ar[r] \ar@{=}[d] & \bullet \ar[d] \ar@{=}[r] & \bullet \ar@{=}[d] \\
  \bullet \ar@{=}[r] & \bullet \ar[r] & \bullet }
\]
\item A \emph{cubical quasicategory} is a cubical set $X$ having the right lifting property with respect to all inner open box inclusions.
\end{itemize}
\end{definition}

\begin{theorem}[{\cite{doherty-kapulkin-lindsey-sattler}}]\label{cubical-Joyal}
Each category $\cA$ carries a model structure in which:
\begin{itemize}
\item The cofibrations are the monomorphisms;
\item The fibrant objects are the cubical quasicategories;
\item A map with fibrant codomain is a fibration if and only if it has the right lifting property with respect to all inner open box inclusions and endpoint inclusions $\BoxA^0 \to K$.
\end{itemize}

Moreover, this model structure is monoidal with respect to the geometric product.
\end{theorem}

\begin{proof}
The existence of the model structure and characterization of the cofibrations are given by \cite[Thm.~4.2]{doherty-kapulkin-lindsey-sattler}. Monoidality is given by \cite[Cor.~4.11]{doherty-kapulkin-lindsey-sattler}, while the characterizations of fibrant objects and fibrations with fibrant codomain is given by \cite[Thm.~4.16]{doherty-kapulkin-lindsey-sattler}.
\end{proof}

For each category of cubical sets, the model structure of \cref{cubical-Joyal} is referred to as the \emph{cubical Joyal model structure}.

When viewing cubical quasicategories as $(\infty,1)$-categories, filling of open boxes represents composition of higher morphisms, and invertible 1-morphisms are represented by edges which factor through $K$.

\begin{definition}
The \emph{homotopy category} $\Ho X$ of a cubical quasicategory $X$ is defined as follows:

\begin{itemize}
\item the objects of $\Ho X$ are the 0-cubes of $X$;
\item for $x_0, x_1 \colon \BoxA^0 \to X$, morphisms $x_0 \to x_1$ in $\Ho X$ are equivalence classes of edges from $x_0$ to $x_1$, subject to the relation $f \sim g$ if there is a 2-cube in $X$ as depicted below:
\[
\xymatrix{
x_0 \ar[r]^{f} \ar@{=}[d] & x_1 \ar@{=}[d] \\
x_0 \ar[r]^{g} & x_1 \\
}
\]
\item the identity map on $x \in X_0$ is given by $x \sigma_1$;
\item composition of $f \colon x \to y$ and $g \colon y \to z$ is given by open box filling, as depicted below:
\[
\xymatrix{
x \ar[d]_{f} \ar@{..>}[r]^{gf} & z \ar@{=}[d] \\
y \ar[r]^{g} & z \\
}
\]
\end{itemize}
That these data define a category is proven in \cite[Lem.~2.20]{doherty-kapulkin-lindsey-sattler}.
\end{definition}

It is also useful to model $(\infty,1)$-categories using cubical sets with markings on edges, but not on cubes of higher dimension, analogous to the model structure on simplicial sets with markings on edges developed in \cite{lurie:htt}.

\begin{definition}
For $A \subseteq \{0,1\}$, the category $\cA'$ of \emph{cubical sets with weak equivalences} is defined as follows.
\begin{itemize}
\item An object in $\cA'$ consists of a cubical set $X \in \cA$, together with a set of \emph{marked edges} $X_e \subseteq X_1$ including all degenerate edges.
\item Morphisms in $\cA'$ are cubical set maps which preserve marked edges.
\end{itemize}
\end{definition}

As in the marked case, we will use the notation $|X|$ for the underlying cubical set of $X \in \cA'$, and taking the underlying cubical set defines a functor $|-| \colon .\cA' \to \cA$. Likewise, we may take the underlying cubical set with weak equivalences of a marked cubical set by forgetting the markings on its cubes of dimension greater than 1. Thus we obtain a functor $|-| \colon \cA^+ \to \cA'$; when using this notation we will rely on context to distinguish whether the underlying cubical set or cubical set with weak equivalences is meant.

As in the marked case, a map of cubical sets with weak equivalences is \emph{regular} if it creates markings, and \emph{entire} if the underlying cubical set map is an isomorphism.

The geometric product can be adapted to the setting of cubical sets with weak equivalences. Recall (\cf \cref{geo-prod-description}) that for $X, Y \in \cA$, an edge of $X \otimes Y$ consists of either an edge of $X$ and a vertex of $Y$, or a vertex of $X$ and an edge of $Y$.

\begin{definition}
For $X, Y \in \cA'$, the \emph{geometric product} $X \otimes Y \in \cA'$ has as its underlying cubical set the geometric product $|X| \otimes |Y|$, with an edge $x \otimes y$ marked if either $x$ is a marked edge of $X$ or $y$ is a marked edge of $Y$.
\end{definition}

As with the geometric product on $\cA$ and the lax Gray tensor product on $\cA^+$, this defines a biclosed monoidal product on $\cA'$ with unit object $\BoxA^0$.

Comparing the definitions of the monoidal products on $\cA, \cA'$, and $\cA^+$, we obtain the following result.

\begin{proposition}\label{forgetful-monoidal}
The functors $|-| \colon \cA^+ \to \cA', |-| \colon \cA' \to \cA$ are monoidal with respect to the geometric products on $\cA$ and $\cA'$ and the lax Gray tensor product on $\cA^+$. \qed
\end{proposition}

One can similarly define a category of \emph{simplicial sets with weak equivalences} $\sSet'$, whose objects are simplicial sets with markings on edges. Note that $\sSet'$ is more commonly referred to as the category of marked simplicial sets (notably in \cite{lurie:htt}), and likewise $\cA'$ is referred to in \cite{doherty-kapulkin-lindsey-sattler} as the category of marked cubical sets. We have adopted the present terminology in order to avoid confusion with the categories $\sSet^+$ and $\cA^+$.

The category $\sSet'$ admits a model structure, the \emph{marked model structure}, which models the theory of $(\infty,1)$-categories. A complete description of this model structure is beyond the scope of this paper, but can be found in \cite[Sec.~3.1]{lurie:htt}, where it arises from a more general model structure on slice categories $\sSet' \downarrow X$ by taking $X = \Delta^0$.

We can adapt the definition of the triangulation adjunction to the setting of cubical and simplicial sets with weak equivalences, using the fact that for a cubical set $X$, the edges of $X$ are in bijection with the edges of $TX$.

\begin{definition}\label{T-prime}
For $X \in \cSet^{\prime}$, we define $TX \in \sSet^{\prime}$ as follows:
\begin{itemize}
\item The underlying simplicial set of $TX$ is $T|X|$, the triangulation of the underlying cubical set of $X$;
\item An edge of $TX$ is marked if and only if the corresponding edge of $X$ is marked.
\end{itemize}
\end{definition}

This definition naturally extends to morphisms, and implies an analogous definition for the right adjoint $U \colon \sSet' \to \cSet'$.

\begin{definition}
For $A \subseteq \{0,1\}$, we define certain objects and maps which play a key role in the model structure on $\cSet_A'$.

\begin{itemize}
\item For $n \geq 1$, $1 \leq i \leq n, \varepsilon \in \{0,1\}$, the \emph{$(i,\varepsilon)$-marked cube}, denoted $\overline{\Box}^{n}_{A,i,\varepsilon}$, is obtained from $\BoxA^n$ by marking the critical edge. The \emph{$(i,\varepsilon)$-marked open box}, denoted $\overline{\sqcap}^n_{A,i,\varepsilon}$, is the regular subcomplex of $\overline{\Box}^n_{A,i,\varepsilon}$ whose non-degenerate cubes consist of all faces of the $n$-cube other than the interior face $\id_{[1]^n}$ and $\bd_{i,\varepsilon}$. The \emph{$(i,\varepsilon)$-marked open box inclusion} is the regular subcomplex inclusion $\overline{\Box}^n_{A,i,\varepsilon} \hookrightarrow \overline{\sqcap}^n_{A,i,\varepsilon}$.
  \item Let $K'$ denote the cubical set with weak equivalences whose underlying cubical set is the invertible interval $K \in \cSet$, with the middle edge marked. The \emph{saturation} map is the entire map $K \to K'$.
  \item For each of the four faces of the square, the \emph{three-out-of-four map} associated to that face is the inclusion of $\BoxA^2$ with all but that face marked into $\BoxA^2$ with all faces marked. 
  \item A \emph{marked cubical quasicategory} is a cubical set with weak equivalences having the right lifting property with respect to the marked open box fillings and the saturation map.
\end{itemize}
\end{definition}

\begin{theorem}[{\cite{doherty-kapulkin-lindsey-sattler}}]
For any $A \subseteq \{0,1\}$, the category $\cSet'_A$ carries a model structure in which:
\begin{itemize}
\item The cofibrations are the monomorphisms;
\item A map with fibrant codomain is a fibration if and only if it has the right lifting property with respect to the marked open box inclusions, the saturation map, and the three-out-of-four maps;
\item The fibrant objects are the marked cubical quasicategories.
\end{itemize}

Moreover, this model structure is monoidal with respect to the geometric product.
\end{theorem}

\begin{proof}
The existence of the model structure and characterization of the cofibrations are given by \cite[Thm.~2.44]{doherty-kapulkin-lindsey-sattler}. The characterizations of fibrant objects and fibrations with fibrant codomain are given by \cite[Lem.~2.6 \& Prop.~2.50]{doherty-kapulkin-lindsey-sattler}. Monoidality is given by \cite[Cor.~2.49]{doherty-kapulkin-lindsey-sattler}. 
\end{proof}

As in comical sets, marked edges in marked cubical quasicategories are thought of as equivalences. The following result validates this intuition.

\begin{proposition}\label{edge-marked-iff-invertible}
In a marked cubical quasicategory $X$, an edge $f \colon \Box^1 \to X$ is marked if and only if the corresponding morphism in the homotopy category of the cubical quasicategory $|X|$ is an isomorphism.
\end{proposition}

\begin{proof}
By \cite[Lem.~2.5]{doherty-kapulkin-lindsey-sattler}, the marked edges of $X$ are precisely those which factor through $K$. That these are precisely the edges corresponding to isomorphisms in $\Ho X$ follows from \cite[Lem.~2.21]{doherty-kapulkin-lindsey-sattler}.
\end{proof}

\subsection{Marking adjunctions}

The forgetful functors $|-| \colon \cA^+ \to \cA'$ and $|-| \colon \cA' \to \cA$ each admit a left and a right adjoint; in fact, in each case the right adjoint to $|-|$ admits a further right adjoint, forming an adjoint quadruple. We now describe these functors.

\begin{definition}\label{marking-functors-def}
For $A \subseteq \{0,1\}$, we have the following functors relating the categories $\cA, \cA'$, and $\cA^+$.
\begin{itemize}
\item The \emph{minimal marking} functor $(-)^\flat \colon \cA \to \cA'$ sends $X \in \cA$ to $X^\flat \in \cA'$, having $X$ as its underlying cubical set, with only degenerate edges marked. 
\item The \emph{minimal marking} functor $(-)^\flat \colon \cA' \to \cA^+$ sends $X \in \cA'$ to $X^\flat \in \cA^+$, having $X$ as its underlying cubical set with weak equivalences, with the only marked cubes of dimension greater than 1 being the degenerate ones.
\item The \emph{maximal marking} functor $(-)^\sharp \colon \cA \to \cA'$ sends $X \in \cA$ to $X^\sharp \in \cA'$, having $X$ as its underlying cubical set, with all edges marked.
\item The \emph{maximal marking} functor $(-)^\sharp \colon \cA' \to \cA^+$ sends $X \in \cA'$ to $X^\sharp \in \cA^+$, having $X$ as its underlying cubical set with weak equivalences, with all cubes of dimension greater than 1 marked.
\item The \emph{$\infty$-groupoid core} functor $c \colon \cA' \to \cA$ sends $X \in \cA'$ to $cX \in \cA$, defined by $(cX)_n = \cA'((\BoxA^n)^\sharp,X)$. More intuitively, $cX$ is obtained by taking the largest regular subcomplex of $X$ in which all edges are marked, and then forgetting the markings.
\item The \emph{$(\infty,1)$-core} functor $c \colon \cA^+ \to \cA'$ is given by $(cX)_n = \cA^+((\BoxA^n)^\sharp,X)$. More intuitively, $cX$ is obtained by taking the largest regular subcomplex of $X$ in which all cubes of dimension greater than 1 are marked, and then forgetting the markings on these higher-dimensional cubes.
\end{itemize}

In all cases, $X^\flat \to Y^\flat$ and $ X^\sharp \to Y^\sharp$ act identically to $X \to Y$ on underlying cubical sets, while $c$ acts on morphisms by post-composition.
\end{definition}

In keeping with our use of the notation $|-|$ for the forgetful functor $\cA^+ \to \cA$, we will also use the notations $(-)^\flat, (-)^\sharp, c$ for the composite functors $\cA \to \cA^+, \cA^+ \to \cA$.

A basic analysis of the definitions shows:

\begin{proposition}
For each pair of categories $(\cA, \cA'), (\cA', \cA^+)$, the  forgetful functor $|-|$ and the functors of \cref{marking-functors-def} form an adjoint quadruple $(-)^\flat \adjoint |-| \adjoint (-)^\sharp \adjoint c$. \qed
\end{proposition}

We have similar families of adjoint quadruples $(-)^\flat \adjoint |-| \adjoint (-)^\sharp \adjoint c$ relating $\sSet$ and $\sSet'$, and likewise $\sSet'$ and $\sSet^+$; each of these functors is defined analogously to its cubical counterpart.

Our goal is now to show that these models of $(\infty,1)$-categories are equivalent to each other, and to the saturated 1-trivial comical and complicial model structures. We begin by recalling established comparisons of the Joyal and marked model structures, in both the simplicial and cubical cases.

\begin{proposition}\label{simplicial-unmarked-we-QE}
The adjunction $(-)^\flat : \sSet \rightleftarrows \sSet' : |-|$ defines a Quillen equivalence between the Joyal model structure on $\sSet$ and the marked model structure on $\sSet'$.
\end{proposition}

\begin{proof}
This is a special case of \cite[Thm.~3.1.5.1(A0)]{lurie:htt}, taking $S = \Delta^0$ in the statement of that result.
\end{proof}

\begin{proposition}[{\cite[Prop.~4.3]{doherty-kapulkin-lindsey-sattler}}] \label{cubical-unmarked-we-QE}
For each $A \subseteq \{0,1\}$, the adjunction $(-)^\flat : \cSet_A \rightleftarrows \cSet_A' : |-|$ defines a Quillen equivalence between the cubical Joyal model structure on $\cSet_A$ and the cubical marked model structure on $\cSet_A'$. \qed
\end{proposition}

In both the simplicial and cubical cases we may note that, of the three adjunctions forming the adjoint quadruple $(-)^\flat \adjoint |-| \adjoint (-)^\sharp \adjoint c$, only the minimal marking adjunction forms a Quillen equivalence between the Joyal and marked model structures. To see this, we may note that the functor $|-|$ is not left Quillen, as the inclusion of an endpoint into the marked interval is a trivial cofibration while its underlying simplcial (resp.~cubical) set map is not. On the other hand, while the adjunction $(-)^\sharp \adjoint c$ is Quillen in both cases, it is not a Quillen equivalence in either. To see this, we may note that the maximal marking of an endpoint inclusion into the unmarked interval is an endpoint inclusion into the marked interval. Thus the left adjoint $(-)^\sharp$ does not reflect weak equivalences; as all objects in these model categories are cofibrant, \cref{QuillenEquivCreate-original} thus implies that the adjunction is not a Quillen equivalence.

Next we consider comparisons between $\cSet_A'$ and $\cSet_A^+$; here we will show that all three adjunctions in the adjoint quadruple are Quillen equivalences. In the course of this analysis, we will have considerable use for the following result.

\begin{lemma}\label{n-triv-entire}
For $A \subseteq \{0,1\}$ and $n \geq 0$, let $X \to Y$ be an entire map in $\cA^+$, such that for all $m \leq n$, an $m$-simplex is marked in $Y$ if and only if it is marked in $X$. Then $X \to Y$ is a trivial cofibration in the (saturated) $n$-trivial comical model structure on  $\cA^+$.
\end{lemma}

\begin{proof}
We can express the map $X \to Y$ as a pushout of markers, as in the following diagram:
\[
\xymatrix{
\coprod\limits_{\substack{x \colon \Box^m \to X \\ x \in e Y \setminus e X}} \Box^m \ar[r] \ar[d] & X \ar[d] \\
\coprod\limits_{\substack{x \colon \Box^m \to X \\ x \in e Y \setminus e X}} \wBox^m \ar[r] & Y \pushoutcorner \\
}
\]
Where the coproducts are taken over all unmarked cubes of $X$ which are marked in $Y$, and the map between coproducts acts as the marker on each component. By assumption, each of these cubes is of dimension greater than $n$; thus we see that the map between coproducts is a trivial cofibration, hence the same is true of its pushout $X \to Y$.
\end{proof}

\begin{corollary}\label{flat-sharp-nat-trans}
For $A \subseteq \{0,1\}$, there is a natural transformation $m \colon (-)^\flat \Rightarrow (-)^\sharp$ of functors from $\cA'$ to $\cA^+$, which acts as the identity on underlying cubical sets, and whose components are all trivial cofibrations in the (saturated) $1$-trivial comical model structure.
\end{corollary}

\begin{proof}
The existence of the natural transformation $m$ follows from the fact that for any $X \in \cSet_A'$, the marked cubical sets $X^\flat$ and $X^\sharp$ have the same underlying cubical set, namely that of $X$, with every marked cube of $X^\flat$ being marked in $X^\sharp$ as well. That its components are trivial cofibrations in the (saturated) $1$-trivial comical model structure follows from  \cref{n-triv-entire}, together with the fact that $X^\flat$ and $X^\sharp$ have the same marked 1-cubes, namely those of $X$.
\end{proof}

\begin{proposition}\label{sharp-core-QA}
For $A \subseteq \{0,1\}$, the adjunction $(-)^\sharp : \cA' \rightleftarrows \cA^+ : c$ is a Quillen adjunction between the marked model structure on $\cA'$ and the saturated 1-trivial comical model structure on $\cA^+$.
\end{proposition}

\begin{proof}
That $(-)^\sharp$ preserves cofibrations follows from the fact that for any map $f$ in $\cA'$, the underlying cubical set maps of $f$ and $f^\sharp$ coincide.
It thus suffices to show that $(-)^\sharp$ sends the pseudo-generating trivial cofibrations of $\cA'$ to trivial cofibrations.

We first note that in the maximal marking of a marked open box $\overline{\sqcap}^n_{A,i,\varepsilon}$, the non-degenerate marked cubes consist of the critical edge, together with all faces of dimension greater than 1. In particular, this implies that all critical faces are marked. The object $(\overline{\Box}^n_{A,i,\varepsilon})^\sharp$ is obtained from $(\overline{\sqcap}^n_{A,i,\varepsilon})^\sharp$ by adding a marked filler for the interior, together with the $(i,\varepsilon)$-face, which is marked if $n \geq 3$. Thus, in the case $n = 2$, the image under $(-)^\sharp$ of a marked open box inclusion $\overline{\sqcap}^n_{A,i,\varepsilon} \hookrightarrow \overline{\Box}^n_{A,i,\varepsilon}$ is a pushout of a comical open box inclusion. In the case $n \geq 3$, it is a composite consisting of a pushout of a comical open box inclusion (adding the marked interior $n$-cube and the $(i,\varepsilon)$-face) followed by a pushout of a comical marking extension (marking the $(i,\varepsilon)$-face). Thus this map is a trivial cofibration.

Similar analyses show that $(-)^\sharp$ sends the saturation map to a pushout of the elementary Rezk map, while the three-out-of-four maps are sent to two-dimensional comical marking extensions.
\end{proof}

\begin{proposition}\label{forgetful-QA}
The adjunction $|-| : \cA^+ \rightleftarrows \cA' : (-)^\sharp$ is Quillen.
\end{proposition}

\begin{proof}
That the left adjoint $|-|$ preserves cofibrations, \ie monomorphisms, follows from the fact that it preserves underlying cubical set maps. It thus suffices to show that $|-|$ sends the pseudo-generating trivial cofibrations of $\cA^+$ to trivial cofibrations. We first note the following:

\begin{itemize}
\item the image under $|-|$ of a comical open box inclusion is a marked open box inclusion;
\item the image under $|-|$ of a comical marking extension $(\Box^2_{A,i,\varepsilon})' \to \tau_{0} \Box^2_{A,i,\varepsilon}$ is a three-out-of-four map;
\item for $n \geq 2$, the image under $|-|$ of a comical marking extension $(\Box^n_{A,i,\varepsilon})' \to \tau_{n-2} \Box^n_{A,i,\varepsilon}$ is an identity;
\item for $n \geq 2$, the image under $|-|$ of the $n$-marker
\end{itemize}

It remains to be shown that $|-|$ sends the Rezk maps to trivial cofibrations. By \cref{forgetful-monoidal} and the monoidality of the cubical marked model structure, it suffices to show that the image of the elementary Rezk map $L \to L'$ under $|-|$ is a trivial cofibration. As $|L| \to |L'|$ is a  cofibration, it suffices to show that it has the left lifting property with respect to all fibrations between fibrant objects. Furthermore, as $|L| \to |L'|$ is entire, and hence an epimorphism, to prove that it has this left lifting property it suffices to show that it has the left lifting property with respect to all marked cubical quasicategories. To this end, consider a diagram as depicted below, with $X$ a marked cubical quasicategory:

\[
\xymatrix{
|L| \ar[r] \ar[d] & X \ar[d] \\
|L'| \ar[r] & \BoxA^0 \\
}
\]
The map $|L| \to X$ consists of a pair of 2-cubes in $X$ as depicted below:
\[
\xymatrix{
  x \ar[r]^{f} \ar[d]_{\sim} & y \ar[d]_{g} \ar[r]^{\sim} & a \ar[d]^{\sim} \\
  b \ar[r]^{\sim} & z \ar[r]^{h} & w }
\]
This map admits a lift $|L'| \to X$ if and only if the edges $f, g, h$ in the diagram above are marked in $X$. That this condition holds is immediate from \cref{edge-marked-iff-invertible} and the two-out-of-six property for isomorphisms.
\end{proof}

\begin{proposition}\label{forgetful-maximal-QE}
For $A \subseteq \{0,1\}$, the Quillen adjunctions $|-| \adjoint (-)^\sharp$ and $(-)^\sharp \adjoint c$ are Quillen equivalences between the cubical marked model structure on $\cA'$ and the saturated 1-trivial comical model structure on $\cA^+$.
\end{proposition}

\begin{proof}
We first note that the composite $|-| \circ (-)^\sharp$ is the identity on $\cA'$, as $(-)^\sharp$ adds markings on all cubes of dimension at least 2, while $|-|$ forgets such markings. Now we consider the composite $(-)^\sharp \circ |-|$. Given $X \in \cA^+$, the marked cubical set $|X|^\sharp$ is obtained from $X$ by marking all of its cubes of dimension at least 2; the unit $X \to |X|^\sharp$ is the unique entire map between these objects. This map is a trivial cofibration by \cref{n-triv-entire}. That the two adjunctions are Quillen equivalences thus follows from \cref{nat-we-QE}.
\end{proof}

\begin{proposition}\label{minimal-QE}
For $A \subseteq \{0,1\}$, the adjunction $\cA' : (-)^\flat \rightleftarrows |-| : \cA^+$ is a Quillen equivalence between the cubical marked model structure on $\cA'$ and the saturated 1-trivial comical model structure on $\cA^+$.
\end{proposition}

\begin{proof}
We begin by showing that the adjunction is Quillen. The functor $(-)^\flat$ preserves monomorphsisms follows from the fact that it preserves underlying cubical set maps. That it preserves weak equivalences follows from \cref{flat-sharp-nat-trans} and the fact that $(-)^\sharp$ preserves weak equivalences by \cref{sharp-core-QA}. The natural weak equivalence $(-)^\flat \Rightarrow (-)^\sharp$ of \cref{flat-sharp-nat-trans} then allows us to obtain a natural isomorphism between the left derived functors of $(-)^\flat$ and $(-)^\sharp$. As the left derived functor of $(-)^\sharp$ is an equivalence of categories by \cref{forgetful-maximal-QE}, it thus follows that the same is true of the left derived functor of $(-)^\flat$.
\end{proof}

By similar reasoning, we obtain an adjoint quadruple of Quillen equivalences between $\sSet'$ and $\sSet^+$. 

\begin{proposition}\label{all-marking-QE-simplicial}
Each adjunction in the adjoint quadruple $(-)^\flat \adjoint |-| \adjoint (-)^\sharp \adjoint c$ defines a Quillen equivalence between the marked model structure on $\sSet'$ and the saturated 1-trivial complicial model structure on $\sSet^+$. \qed
\end{proposition}

We will not record the proof of this result in full, as it is effectively identical to the preceding set of proofs for the cubical case. The only substantial difference is that when showing that ${|-|} \colon \sSet^+ \to \sSet'$ sends the simplicial analogues of the Rezk maps (the saturation extensions of \cite[Def.~1.19(4)]{ozornova-rovelli}) to weak equivalences, one one cannot appeal to monoidality as was done in the proof of \cref{forgetful-QA}; it is easy to show, however, that each such map is sent to a pushout of a simplicial analogue of the saturation map.

Using the Quillen equivalence results above, we can obtain analogues of \cref{T-min-Quillen-eqv} for unmarked cubical sets and cubical sets with weak equivalences; in particular, these results generalize \cite[Thm.~6.1]{doherty-kapulkin-lindsey-sattler} to the case of cubical sets without connections.

\begin{theorem}\label{T-Quillen-eqv-marked-MS}
For each $A \subseteq \{0,1\}$, the adjunction $T : \cA' \rightleftarrows \sSet'$ is a Quillen equivalence between the cubical marked model structure on $\cA'$ and the marked model structure on $\sSet'$.
\end{theorem}

\begin{proof}
We begin by considering the following diagram of adjunctions:
\[
\xymatrix{
\cSet_A' \ar@<1ex>[rr]^{(-)^\sharp} \ar@{}|{\rotatebox{-90}{$\adjoint$}}[rr] \ar@<1ex>@{<-}[dd]^{U} \ar@{}|{\adjoint}[dd] && \cA^+ \ar@<1ex>[ll]^{c} \ar@<1ex>@{<-}[dd]^{U} \ar@{}|{\adjoint}[dd]
\\
\\
\sSet' \ar@<1ex>[rr]^{(-)^\sharp} \ar@{}|{\rotatebox{-90}{$\adjoint$}}[rr] \ar@<1ex>@{<-}[uu]^{T} && \sSet^+ \ar@<1ex>[ll]^{c} \ar@<1ex>@{<-}[uu]^{T} 
\\
}
\]

To see that this diagram commutes, it suffices to verify that the diagram of left adjoints commutes on representable objects in $\cA'$. For the marked 1-cube $\wBox^1_A$, we may note that both $T \circ (-)^\sharp$ and $(-)^\sharp \circ T$ send this object to the marked simplicial set $\widetilde{\Delta}^1$. For the unmarked cubes $\BoxA^n, n \geq 0$ in $\cA'$, we may note that $T \BoxA^n$ has no non-degenerate marked edges, and this is likewise true for the image of $\BoxA^n \in \cA^+$ under the marked triangulation functor. Thus $(T \BoxA^n)^\sharp = T((\BoxA^n)^\sharp)$, as both of these are obtained from the unmarked simplicial set $(\Delta^1)^n$ by marking all non-degenerate simplices of dimension at least 2, but no edges.

We now show that  $T : \cA' \rightleftarrows \sSet' : U$ is a Quillen adjunction. That $T$ preserves cofibrations follows from the analogous fact for the unmarked triangulation functor $T \colon \cA \to \sSet$. To see that $T$ creates weak equivalences, let $f$ denote a trivial cofibration in $\cA'$. Note that the top, right, and bottom adjunctions in this diagram are Quillen equivalences by \cref{T-min-Quillen-eqv,forgetful-maximal-QE,all-marking-QE-simplicial}.

Applying \cref{QuillenEquivCreate-original} to the adjunction $(-)^\sharp : \sSet' \rightleftarrows \sSet^+ : c$, we see that the cofibration $Tf$ is trivial if and only if $(Tf)^\sharp$ is a weak equivalence. By the commutativity of the diagram above, this map is equal to $T(f^\sharp)$, which is a trivial cofibration as the image of the trivial cofibration $f$ under the left Quillen functor $T \circ (-)^\sharp$.

Thus the adjunction $T : \cA' \rightleftarrows \sSet' : U$ is indeed Quillen; that it is a Quillen equivalence then follows from the commutativity of the diagram above and the two-out-of-three property for Quillen equivalences.
\end{proof}

\begin{theorem}\label{T-Quillen-equiv-unmarked}
For each $A \subseteq \{0,1\}$, the adjunction $T : \cA \rightleftarrows \sSet : U$ is a Quillen equivalence between the cubical Joyal model structure on $\cA$ and the Joyal model structure on $\sSet$.
\end{theorem}

\begin{proof}
The adjunction $T : \cA \rightleftarrows \sSet$ is Quillen by \cite[Prop.~4.28]{doherty-kapulkin-lindsey-sattler}. To see that it is a Quillen equivalence, consider the following diagram of adjunctions:
\[
\xymatrix{
\cSet_A \ar@<1ex>[rr]^{(-)^\flat} \ar@{}|{\rotatebox{-90}{$\adjoint$}}[rr] \ar@<1ex>@{<-}[dd]^{U} \ar@{}|{\adjoint}[dd] && \cA' \ar@<1ex>[ll]^{|-|} \ar@<1ex>@{<-}[dd]^{U} \ar@{}|{\adjoint}[dd]
\\
\\
\sSet \ar@<1ex>[rr]^{(-)^\flat} \ar@{}|{\rotatebox{-90}{$\adjoint$}}[rr] \ar@<1ex>@{<-}[uu]^{T} && \sSet' \ar@<1ex>[ll]^{|-|} \ar@<1ex>@{<-}[uu]^{T} 
\\
}
\]
Once again, to see that this diagram commutes it suffices to verify that the composites of left adjoints agree on representables; for this, we note that both $T \circ (-)^\flat$ and $(-)^\flat \circ T$ send each $\BoxA^n$ to the simplicial set $(\Delta^1)^n$ equipped with its minimal marking. That $T : \cA \rightleftarrows \sSet$ is a Quillen equivalence thus follows from \cref{minimal-QE,all-marking-QE-simplicial,T-Quillen-eqv-marked-MS} together with the two-out-of-three property for Quillen equivalences.
\end{proof}

We can also use two-out-of-three for Quillen equivalences to show that the analogues of the adjoint triples $i_! \adjoint i^* \adjoint i_*$ for unmarked cubical sets and cubical sets with weak equivalences are composed of Quillen equivalences.

For each $A \subseteq B \subseteq \{0,1\}$, the inclusion $i \colon \BoxA \hookrightarrow \BoxB$ induces a functor $i^* \colon \cB \to \cA$ (\resp $i^* \colon \cB' \to \cA'$). As in the case of marked cubical sets, this functor admits both a right and a left adjoint, which admit descriptions analogous to \cref{i-left-explicit,i-right-adj}. Specifically:

\begin{itemize}
\item for $X \in \cA$, the object $i_! X$ is obtained by freely adding the connections of $\BoxB$ to $X$;
\item for $X \in \cA$ and $n \geq 0$, $(i_* X)_n = \cA'(i^* \BoxB^n,X)$; 
\item for $X \in \cA'$, the object $i_! X$ is obtained by freely adding the connections of $\BoxB$ to $X$, with the marked edges being precisely those of $X$;
\item for $X \in \cA'$ and $n \geq 0$, $(i_* X)_n = \cA'(i^* \BoxB^n,X)$ with an edge marked if and only if it sends $\id_{[1]}$ to a marked edge of $X$.
\end{itemize}

\begin{proposition}
For any $A \subseteq B \subseteq \{0,1\}$, the adjunctions $i_! \adjoint i^*$ and $i^* \adjoint i_*$ induced by the inclusion $\BoxA \hookrightarrow \BoxB$ are Quillen equivalences between the cubical Joyal model structures on $\cA$ and $\cB$ (\resp cubical marked model structures on $\cA'$ and $\cB'$.)
\end{proposition}

\begin{proof}
Consider the following diagram of adjoint triples:
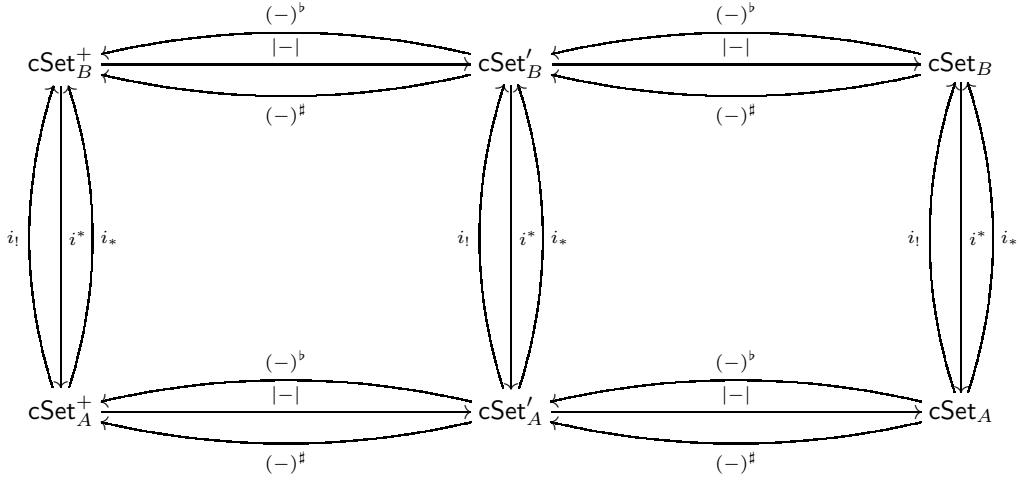
\[
\vcenter{\xymatrix@C+1.5cm{
  \cB^+
  \ar[rr]^{|-|}
  \ar[dddd]^{i^*}
&&
  \cB'
  \ar@/^1pc/[ll]^{(-)^\sharp}
  \ar@/_1pc/[ll]_{(-)^\flat}
  \ar[dddd]^{i^*}
   \ar[rr]^{|-|}
&&
   \cB
  \ar@/^1pc/[ll]^{(-)^\sharp}
  \ar@/_1pc/[ll]_{(-)^\flat}
  \ar[dddd]^{i^*}
\\ 
&
&
\\
&
&
\\
\\
\cA^+
 \ar[rr]^{|-|}
 \ar@/^1pc/[uuuu]^{i_!}
 \ar@/_1pc/[uuuu]_{i_*}
&
&
 \cA'
 \ar@/^1pc/[ll]^{(-)^\sharp}
 \ar@/_1pc/[ll]_{(-)^\flat}
 \ar@/^1pc/[uuuu]^{i_!}
 \ar@/_1pc/[uuuu]_{i_*} 
 \ar[rr]^{|-|}
 &&
 \cA
 \ar@/^1pc/[ll]^{(-)^\sharp}
 \ar@/_1pc/[ll]_{(-)^\flat}
 \ar@/^1pc/[uuuu]^{i_!}
 \ar@/_1pc/[uuuu]_{i_*} 
}}
\]

To see that the diagram commutes, it suffices to observe that the middle functors in each of the two squares satisfy $i^* \circ |-| = |-| \circ i^*$; this is immediate from the definitions.

Next we show that the adjunctions $i_! \adjoint i^*$ and $i^* \adjoint i_*$ for cubical sets with weak equivalences are Quillen. For $i_! \adjoint i^*$, an argument similar to the proof of \cref{i-free-Quillen} shows that $i_!$ sends each of the generating cofibrations and pseudo-generating trivial cofibrations of $\cA'$ to the corresponding map in $\cB'$. 
For $i^* \adjoint i_*$, we may note that $i^* \colon \cB' \to \cA'$ is equal to the composite $|-| \circ i^* \circ (-)^\flat$ (where $i^*$ here denotes the functor $\cA^+ \to \cB^+$); as each of these functors is left Quillen, the same is true of their composite.

Finally, we note that each of the adjoint triples of the form $(-)^\flat \adjoint |-| \adjoint (-)^\sharp$ depicted above is composed of Quillen adjunctions by \cref{forgetful-maximal-QE,minimal-QE}, as is the adjoint triple $i_! \adjoint i^* \adjoint i_*$ between $\cA^+$ and $\cB^+$ by \cref{i-Quillen-eqv}. It follows that all remaining adjunctions in the diagram are Quillen equivalences by two-out-of-three.
\end{proof}

\end{appendix}

\bibliographystyle{amsalphaurlmod}
\bibliography{general-bibliography}

\providecommand{\bysame}{\leavevmode\hbox to3em{\hrulefill}\thinspace}
\providecommand{\MR}{\relax\ifhmode\unskip\space\fi MR }
\providecommand{\MRhref}[2]{%
  \href{http://www.ams.org/mathscinet-getitem?mr=#1}{#2}
}
\providecommand{\href}[2]{#2}
\begin{thebibliography}{DKLS20}

\bibitem[Ant02]{antolini:geometric-realisations}
Rosa Antolini, \emph{Geometric realisations of cubical sets with connections,
  and classifying spaces of categories}, Applied Categorical Structures
  \textbf{10} (2002), 481--494, \href
  {http://dx.doi.org/10.1023/A:1020506404904}
  {\path{doi:10.1023/A:1020506404904}}.

\bibitem[BH77]{brown-higgins:T-complexes}
Ronald Brown and Philip Higgins, \emph{Sur les complexes croisés,
  $\omega$-groupoïdes et t-complexes}, C.R. Acad. Sci. Paris Sér. A
  \textbf{285} (1977), 997--999.

\bibitem[BM17]{buchholtz-morehouse:varieties-of-cubes}
Ulrik Buchholtz and Edward Morehouse, \emph{Varieties of cubical sets},
  Relational and Algebraic Methods in Computer Science. RAMICS 2017. (Cham),
  vol. 10226, Springer, 2017, pp.~77--92.

\bibitem[Cis06]{CisinskiAsterisque}
Denis-Charles Cisinski, \emph{Les pr\'efaisceaux comme mod\`eles des types
  d'homotopie}, Ast\'erisque (2006), no.~308, xxiv+390.

\bibitem[CKM20]{campion-kapulkin-maehara}
Timothy Campion, Krzysztof Kapulkin, and Yuki Maehara, \emph{A cubical model
  for $(\infty,n)$-categories}, Submitted, 2020.

\bibitem[DKLS20]{doherty-kapulkin-lindsey-sattler}
Brandon Doherty, Krzysztof Kapulkin, Zachery Lindsey, and Christian Sattler,
  \emph{Cubical models of $(\infty,1)$-categories}, Submitted, 2020, \href
  {http://arxiv.org/abs/2005.04853} {\path{arXiv:2005.04853}}.

\bibitem[DKM21]{doherty-kapulkin-maehara}
Brandon Doherty, Krzysztof Kapulkin, and Yuki Maehara, \emph{Equivalence of
  cubical and simplicial approaches to $(\infty,n)$-categories}, Submitted,
  2021, \href {http://arxiv.org/abs/2106.09428} {\path{arXiv:2106.09428}}.

\bibitem[GM03]{grandis-mauri}
Marco Grandis and Luca Mauri, \emph{Cubical sets and their site}, Theory and
  Applications of Categories \textbf{11} (2003), no.~8, 185–211.

\bibitem[GZ67]{gabriel-zisman}
Peter Gabriel and Michel Zisman, \emph{Calculus of fractions and homotopy
  theory}, Ergebnisse der Mathematik und ihrer Grenzgebiete, Band 35,
  Springer-Verlag New York, Inc., New York, 1967.

\bibitem[Hov99]{hovey:book}
Mark Hovey, \emph{Model categories}, Mathematical Surveys and Monographs,
  vol.~63, American Mathematical Society, Providence, RI, 1999.

\bibitem[Isa11]{isaacson:symmetric}
Samuel~B. Isaacson, \emph{Symmetric cubical sets}, J. Pure Appl. Algebra
  \textbf{215} (2011), no.~6, 1146--1173, \href
  {http://dx.doi.org/10.1016/j.jpaa.2010.08.001}
  {\path{doi:10.1016/j.jpaa.2010.08.001}},
  \url{https://doi.org/10.1016/j.jpaa.2010.08.001}.

\bibitem[Jar06]{jardine:categorical-homotopy-theory}
J.~F. Jardine, \emph{Categorical homotopy theory}, Homology Homotopy Appl.
  \textbf{8} (2006), no.~1, 71--144,
  \url{http://projecteuclid.org/euclid.hha/1140012467}.

\bibitem[JT07]{joyal-tierney:qcat-vs-segal}
Andr{\'e} Joyal and Myles Tierney, \emph{Quasi-categories vs {S}egal spaces},
  Categories in algebra, geometry and mathematical physics, Contemp. Math.,
  vol. 431, Amer. Math. Soc., Providence, RI, 2007, pp.~277--326, \href
  {http://dx.doi.org/10.1090/conm/431/08278}
  {\path{doi:10.1090/conm/431/08278}}.

\bibitem[Kan55]{kan:abstract-htpy-1}
Daniel~M. Kan, \emph{Abstract homotopy. {I}}, Proceedings of the National
  Academy of Sciences of the United States of America \textbf{41} (1955),
  no.~12, 1092--1096.

\bibitem[Lur09]{lurie:htt}
Jacob Lurie, \emph{Higher topos theory}, Annals of Mathematics Studies, vol.
  170, Princeton University Press, Princeton, NJ, 2009,
  \url{http://www.math.harvard.edu/~lurie/papers/croppedtopoi.pdf}.

\bibitem[Mal09]{maltsiniotis:connections-strict-test-cat}
Georges Maltsiniotis, \emph{La cat\'{e}gorie cubique avec connexions est une
  cat\'{e}gorie test stricte}, Homology Homotopy Appl. \textbf{11} (2009),
  no.~2, 309--326, \url{http://projecteuclid.org/euclid.hha/1296138523}.

\bibitem[OR20]{ozornova-rovelli}
Viktoriya Ozornova and Martina Rovelli, \emph{Model structures for
  ({$\infty,n$})-categories on (pre)stratified simplicial sets and
  prestratified simplicial spaces}, Algebr. Geom. Topol. \textbf{20} (2020),
  no.~3, 1543--1600, \href {http://dx.doi.org/10.2140/agt.2020.20.1543}
  {\path{doi:10.2140/agt.2020.20.1543}},
  \url{https://doi.org/10.2140/agt.2020.20.1543}.

\bibitem[Ton92]{tonks:cubical-groups-kan}
A.~P. Tonks, \emph{Cubical groups which are {K}an}, J. Pure Appl. Algebra
  \textbf{81} (1992), no.~1, 83--87, \href
  {http://dx.doi.org/10.1016/0022-4049(92)90136-4}
  {\path{doi:10.1016/0022-4049(92)90136-4}},
  \url{https://doi.org/10.1016/0022-4049(92)90136-4}.

\bibitem[Ver08]{verity:weak-complicial-1}
Dominic Verity, \emph{Weak complicial sets. {I}. {B}asic homotopy theory}, Adv.
  Math. \textbf{219} (2008), no.~4, 1081--1149, \href
  {http://dx.doi.org/10.1016/j.aim.2008.06.003}
  {\path{doi:10.1016/j.aim.2008.06.003}},
  \url{https://doi.org/10.1016/j.aim.2008.06.003}.

\end{thebibliography}

\end{document}